\documentclass[12pt]{amsart}
\usepackage{amsmath}
\usepackage{graphicx}
\usepackage{latexsym}
\usepackage{color}
\usepackage{amscd}
\usepackage[all]{xy}
\usepackage{enumerate}
\usepackage{hyperref}
\usepackage{subfigure}
\usepackage{soul}
\usepackage{comment}
\usepackage{epsfig}
\usepackage{epstopdf}
\usepackage{marginnote}
\usepackage{eqnarray, amsthm}
\usepackage{xcolor} % text color
\usepackage{tikz-cd} % diagrams
\usetikzlibrary{positioning}
\usetikzlibrary{knots}
\usetikzlibrary{decorations.pathreplacing}

\usepackage{multirow}

\newtheorem{thm}{Theorem}

\newtheorem{prop}[thm]{Proposition}

\newtheorem{theorem}[thm]{Theorem}
\newtheorem{lemma}[thm]{Lemma}
\newtheorem{corollary}[thm]{Corollary}
\newtheorem{proposition}[thm]{Proposition}

\theoremstyle{definition}

\newtheorem*{definition*}{Definition}

\newtheorem{definition}[thm]{Definition}

\newtheorem{remark}[thm]{Remark}
\newtheorem{example}[thm]{Example}

\newcommand{\CPb}{\overline{\mathbb{CP}}{}^{2}}
\newcommand{\CP}{{\mathbb{CP}}{}^{2}}

\newcommand{\CPo}{{\mathbb{CP}}{}^{1}}

\newcommand{\RP}{{\mathbb{RP}}{}^{2}}
\newcommand{\R}{\mathbb{R}}
\newcommand{\Q}{\mathbb{Q}}

\newcommand{\Z}{\mathbb{Z}}

\newcommand{\twprod}{\mathbin{\mathchoice%
		{\ooalign{\raise1.15ex\hbox{$\scriptstyle\sim$}\cr\hidewidth$\times$\hidewidth\cr}}%
		{\ooalign{\raise1.15ex\hbox{$\scriptstyle\sim$}\cr\hidewidth$\times$\hidewidth\cr}}%
		{\ooalign{\raise.85ex\hbox{$\scriptscriptstyle\sim$}\cr\hidewidth$\scriptstyle\times$\hidewidth\cr}}%
		{\ooalign{\raise.65ex\hbox{$\scriptscriptstyle\sim$}\cr\hidewidth$\scriptscriptstyle\times$\hidewidth\cr}}%   
}}

\newcommand{\M}{\operatorname{Mod}}

\def \x {\times}

\newcommand{\ZZ}{\mathbb{Z}}

\newcommand{\RR}{\mathbb{R}}

\newcommand{\PP}{\mathbb{P}}
\newcommand{\Diff}{\operatorname{Diff}}
\newcommand{\inv}{^{-1}}
\newcommand{\Mod}{\operatorname{Mod}}
\newcommand{\lan}{\langle}
\newcommand{\ran}{\rangle}

\newcommand{\twist}{\Tilde{\times}}

\renewcommand{\M}{\operatorname{Mod}}
\newcommand{\D}{\operatorname{Diff}}

\begin{document}
	
	\title[On nonorientable  $4$--manifolds]
	{On nonorientable  $4$--manifolds}

	\author[R. \.{I}. Baykur]{R. \.{I}nan\c{c} Baykur}
	\address{Department of Mathematics and Statistics, University of Massachusetts, Amherst, MA 01003-9305, USA}
	\email{inanc.baykur@umass.edu}
	
	\author[P. Morgan]{Porter Morgan}
	\address{Department of Mathematics and Statistics, University of Massachusetts, Amherst, MA 01003-9305, USA}
	\email{pamorgan@umass.edu}

	\begin{abstract}
		We present several structural results on closed, nonorientable, smooth $4$--manifolds, extending analogous results and machinery for the orientable case. We prove the existence of simplified broken Lefschetz fibrations and simplified trisections on nonorientable $4$--manifolds, yielding descriptions of them via factorizations in mapping class groups of nonorientable surfaces. With these tools in hand, we classify low genera simplified broken Lefschetz fibrations on nonorientable $4$--manifolds. We also establish that every closed, smooth $4$--manifold is obtained by surgery along a link of tori in a connected sum of copies of $\CP$, $S^1 \x S^3$, $S^2\times \mathbb{R}\mathbb{P}^2,$ and $\mathbb{R}\mathbb{P}^4$. Our proofs make use of topological modifications of singularities, handlebody decompositions, and mapping classes of surfaces.
	\end{abstract}

	\maketitle
	
	\setcounter{secnumdepth}{2}
	\setcounter{section}{0}
	
	\tableofcontents
	
	\section{Introduction}
	
	The goal of this article is to extend some tools and results pertaining to orientable \linebreak \mbox{$4$--manifolds} over to nonorientable ones. We work in the smooth category. Below, we present four theorems.
	
	Our first result concerns \emph{simplified broken Lefschetz fibrations} (SBLF), which are certain generalizations of Lefschetz fibrations that are known to be supported on all closed, connected, orientable $4$--manifolds \cite{ADK, baykurPJM, Gay-Kirby:BLF, Lekili, Baykur-Saeki:TAMS}. 
	
	\begin{theorem} \label{thm:SBLFexistence}
		Every closed, connected, nonorientable $4$--manifold $X$ admits a simplified broken Lefschetz fibration that can be obtained by homotoping a given generic map $X \to S^2$.
	\end{theorem}
	
	In the nonorientable setting, the local models around the Lefschetz critical points are necessarily unoriented, as in \cite{Miller-Ozbagci, Onaran-Ozbagci}; in this sense, these singularities are \emph{achiral}. In Proposition~\ref{prop:notall} we demonstrate that not every nonorientable $4$--manifold $X$ admits a Lefschez fibration or pencil; one needs to allow indefinite folds for a general existence result, that is, consider \emph{broken} Lefschetz fibrations instead, as we did here.
	
	The homotopy in the theorem can be obtained by an explicit algorithm adapted from Baykur--Saeki's work \cite{Baykur-Saeki:TAMS}; also see~\cite{saeki:simplifying-general} for a similar result in any \mbox{dimension $\geq 4$.} A  consequence of Theorem~\ref{thm:SBLFexistence} is that any closed, nonorientable $4$-manifold can be prescribed by an ordered tuple of curves on a nonorientable surface, subject to an easy-to-check condition; see Corollary~\ref{cor:combinatorial}
	
	Our second result concerns \emph{simplified trisections}, which are special Gay-Kirby trisections \cite{Gay-Kirby:trisection} of a $4$--manifold where the corresponding trisection map has an embedded singular image \cite{Baykur-Saeki:TAMS, Baykur-Saeki:PNAS, Hayano:STS}. Another algorithm we adapt from \cite{Baykur-Saeki:TAMS} yields:
	
	\begin{theorem} \label{thm:STexistence}
		Every closed, connected, nonorientable $4$--manifold admits a simplified trisection. 
	\end{theorem}
	
	In particular, we obtain an alternate proof of the existence of trisections on nonorientable $4$--manifolds, shown earlier in \cite{Rubinstein-Tillman, Miller-Naylor}. Notably, the stable equivalence of trisection decompositions also holds for the \emph{simplified} trisection decompositions of nonorientable $4$--manifolds; cf.~\cite[Remark~7.8]{Baykur-Saeki:TAMS}. We also give several examples where our methods produce \emph{minimal genus} trisections. 
	
	\smallskip
	Our third result is a nonorientable analog of the classification of low genera fibrations in the orientable case \cite{Baykur_Kamada, hayano:genus1, hayano:genus1-II, Miller-Ozbagci}. The genus of the SBLF is defined as the highest genus among all fibers, which, in the nonorientable case, turns out to be the genus of a nonorientable surface $N_k:= \#_k \, \mathbb{RP}^2$. If a broken Lefschetz fibration has a nodal fiber containing $S^2$ or $\RP$, then we can remove the corresponding singularity, without destroying the rest of the fibration, by the usual blow-down operation or the \emph{nonorientable blow-down}. Indeed, any broken Lefschetz fibration is a blow-up of a \emph{relatively minimal} fibration, which by definition does not contain such nodes; see Remark~\ref{rk:relminimal}. Therefore, it suffices to classify relatively minimal fibrations. 
	
	There is little to say about nonorientable genus $k < 2$ fibrations (see Remark~\ref{rmk:low_genus_2}), whereas the classification in the case $k = 2$ turns out to be fairly rich:
	
	\begin{theorem}
		\label{thm:g2_classification}
		Let $X$ be a closed, connected, nonorientable $4$--manifold, admitting a relatively minimal simplified broken Lefschetz fibration $f\colon X \to S^2$ of \emph{nonorientable genus two}. Then the diffeomorphism type of $X$ is one of the following:
		\[
		X \cong
		\left\{
		\begin{array}{ll}
			\quad K_n, \hfill
			& \quad \text{if } C_f = \emptyset,\ Z_f = \emptyset, \\[0.5ex]
			
			\quad N_n \text{ or } N'_n 
			& \quad \text{if } C_f = \emptyset,\ Z_f \neq \emptyset, \\[0.5ex]
			
			\quad M_{m,n} 
			& \quad \text{if } C_f \neq \emptyset,\ Z_f = \emptyset.
		\end{array}
		\right.
		\]     
		for some $n \in \Z_{\geq 0}$, $m \in \Z_{>0}$, where $C_f$ and $Z_f$ are the set of Lefschetz and fold singularities of $f$, respectively. 
	\end{theorem}
	
	The manifolds $K_n$, which are Klein bottle bundles over $S^2$, and
	$N_n$, $N'_n$, and $M_{m,n}$ above are explicitly defined through Kirby diagrams in Section~\ref{sec:classification}. Some specific diffeomorphism types that arise are $K_0\cong Kb\times S^2$, $N_0\cong S^1\twist S^3\# S^2\times S^2$, and $N'_0\cong S^1\twist S^3\# 2\CP$. Note that the case when both $C_f$ and $Z_f$ are nonempty arises only when $f$ is not relatively minimal. Using the algorithm we employed in the proof of Theorem~\ref{thm:STexistence}, we moreover derive new low genera (simplified) trisections on some of these $4$--manifolds, extending the conjectural list of genus--$3$ trisections with nonorientable examples; see Remark~\ref{rk:genus3}.
	
	Our last theorem concerns \emph{torus surgeries}, which are $4$-dimensional analogs of Dehn surgeries, performed by cutting out and gluing back in trivial neighborhoods of tori.
	
	\begin{theorem} \label{thm:surgery}
		Every closed, connected, nonorientable $4$--manifold can be obtained by surgery along a link of tori in $Z=a \, (S^2\times \mathbb{R}\mathbb{P}^2) \# b \, \mathbb{R}\mathbb{P}^4 \# c \, \mathbb{C}\mathbb{P}^2\#d \, (S^1\times S^3)$, for some $a, b, c, d \in \Z_{\geq 0}$.
	\end{theorem}
	
	This result is a nonorientable version of a theorem of Iwase \cite{Iwase}; also see \cite{Baykur_Sunukjian_2012}. Since torus surgeries do not alter orientability, we necessarily have $a+b \geq 1$ and we can in fact assume $a+b=1$. The statement of Theorem~\ref{thm:surgery} extends to all closed, connected $4$--manifolds, whether orientable or not, yielding a $4$-dimensional version of the generalized Lickorish-Wallace theorem for \mbox{$3$--manifolds;} see Corollary~\ref{cor:4Dlickorish-wallace}. 
	
	\smallskip
	To obtain our results, we carefully adapt tools and arguments involving singular fibrations, topological surgeries, handle decompositions, Kirby calculus, mapping class groups of surfaces, and round handle cobordisms to the nonorientable setting. We expect this to be a fruitful direction for further study of nonorientable $4$--manifolds.

	\vspace{0.2in}
	\noindent \textit{Acknowledgements.} This work was partially supported by the NSF grant DMS-2005327. The authors thank Rafael Torres and Valentina Bais for their comments on the first draft of this paper.

	\clearpage
	\smallskip
	\section{Preliminaries}
	
	Below we summarize our conventions, review the main definitions and background results needed throughout the paper, and provide a few preliminary results for later sections. 
	
	\smallskip
	\noindent \textit{\underline{Conventions}:} All manifolds and maps in this article are smooth.  We denote a compact, connected surface $S$ with $b$ boundary components by $\Sigma_g^b$ when it is orientable, and by $N_k^b$ when it is nonorientable, where $g$ and $k$ are the genus in each case. We drop $b$ from the notation when $\partial S = \emptyset$. 
    In proving Theorem \ref{thm:g2_classification}, we will introduce the $4$--manifolds $N_n$ and $N'_n$, which are different than the nonorientable surface $N_k$. We assume it will be clear from context when this notation refers to a surface or a $4$--manifold.
    By 
	$\D(S)$ we denote the group of diffeomorphisms of $S$ that restrict to the identity near the boundary whenever it is nonempty. The \textit{mapping class group} of $S$ is $\M(S):=\pi_0(\D(S))$. When $S$ is orientable, we use $\D^+(S)$ and $\M^+(S)$ to denote the subgroups that consist of orientation-preserving elements. By a \emph{curve} on $S$, we always mean a simple closed curve, and often we denote its isotopy class by the same symbol. We call $c$ \emph{essential} if it does not bound a disk or a M\"obius band on $S$. A \textit{Dehn twist} along a curve $c$ on $S$ is denoted by $t_c \in \M(S)$. Given two spaces $A$ and $B$ where there are exactly two diffeomorphism types for the total spaces of $B$ fiber bundles over $A$, $A\twprod B$ denotes the unique non-trivial bundle.

	\subsection{Lefschetz, fold and cusp singularities}\label{sec:singularity_types} 
	
	Let $f\colon X\to \Sigma$ be a map from a compact, smooth $4$--manifold to a surface. When $\partial X \neq \emptyset$, assume $f^{-1}(\partial \Sigma)=\partial X$ and that $f$ is a submersion in a neighborhood of $\partial X$. A point $x\in\text{Int}(X)$ is a {\it singular point} if $\rm{rank}(df_x)< 2$. 
	
	We are concerned with the following three types of singularities that may occur at a singular point $x$, which can be characterized by the existence of local coordinates around $x \in X$ and $f(x) \in \Sigma$, in which $f$ conforms to the given local model:
	
	\begin{enumerate}
		\item a {\it Lefschetz singularity} for $(z_1,z_2)\mapsto z_1z_2$;
		\item a {\it fold singularity} for $(t,x_1,x_2,x_3)\mapsto (t,\pm x_1^2\pm x_2^2\pm x_3^2)$;
		\item a {\it cusp singularity} for $(t,x_1,x_2,x_3)\mapsto (t,x_1^3+tx_1\pm x_2^2\pm x_3^2)$.
	\end{enumerate} A fold or cusp singularity is {\it definite} if all coefficients of the quadratic terms in the local charts have the same sign, and is {\it indefinite} otherwise. When $X$ is orientable, one typically asks that the local model in (1) is compatible with the orientation.
	
	Maps from a $4$--manifold $X$ to a surface $\Sigma$ with only fold and cusp singularities, called \emph{generic maps} or \emph{Morse $2$--functions}, are open and dense in $C^\infty (X,\Sigma)$ endowed with the Whitney topology \cite{Whitney1955}. Their singular set constitutes an embedded $1$--manifold in $X$ with cusp singularities at discrete points. An {\it indefinite generic map} from a $4$--manifold to a surface is assumed to have only indefinite cusp and fold singularities. 
	
	Given a map $X\to \Sigma$ with only fold, cusp, and Lefschetz singularities, its {\it base diagram} is a picture of $\Sigma$ decorated with the singular image of $f$. We denote the image of a Lefschetz singularity with an ``x'' and indefinite fold and cusp singularities by smoothly embedded arcs and circles with at most transverse double-point intersections along fold points. The cusp singularities occur exactly over the cuspidal points along these folds arcs. Definite fold circles will be shown in red circles. Small transverse arrows decorate fold circles and point in the direction of the higher index handle attachment. In other words, the arrows point toward the region whose regular fiber has a higher euler characteristic. Figure \ref{fig:s1s3_homotopy1} gives an example of base diagrams with this notation.

	\subsection{Simplified broken Lefschetz fibrations} 
	A {\it broken Lefschetz fibration} (BLF) is a surjective map $f \colon X\to \Sigma$ with only Lefschetz and indefinite fold singularities \cite{ADK}. 
	The sets of Lefschetz and indefinite fold singularities constitute disjoint embedded submanifolds $C_f$ and $Z_f$ of $X$ of dimensions zero and one, respectively. 
	
	When $X$ is a nonorientable $4$--manifold, there is no orientation compatibility for the local models around Lefschetz singular points, akin to the \emph{achiral} Lefschetz fibrations on orientable $4$--manifolds \cite{Gompf-Stipsicz}. 
	Moreover, the fibers can be both orientable and nonorientable closed surfaces, possibly disconnected. The local models for (indefinite) folds are already indifferent to orientations. One can assume that $f(Z_f)$ is immersed in the base $\Sigma$ and is always a two-sided curve, where a regular fiber over one side is cobordant to a regular fiber over the other through a three dimensional $1$--handle attachment. The side resulting from the $1$--handle attachment is the \emph{higher genus side}.
	
	\begin{definition} \label{def:SBLF}
		A {\it simplified broken Lefschetz fibration} (SBLF) $f\colon X\to S^2$ is a BLF with the following additional properties:
		\begin{enumerate}[(i)]
			\item $f$ is injective on $C_f \sqcup Z_f$,
			\item $Z_f$ is connected,
			\item all fibers of $f$ are connected, and 
			\item $f(C_f)$ lies on the higher genus side. 
		\end{enumerate}
	\end{definition}
	
	When $Z_f=\emptyset$, this is an (achiral) Lefschetz fibration. If $f\colon X\to S^2$ is an SBLF away from finitely many points around which $f$ is modeled by $(z_1,z_2)\mapsto z_1/z_2$, we get an (achiral) {\it simplified broken Lefschetz pencil}. Although there is no good way to depict a general BLF using handle diagrams or factorizations in the mapping class group, the additional assumptions for a simplified BLF allow for fairly simple descriptions that generalize those of Lefschetz fibrations, which we discuss shortly.
	
	\subsection{Simplified trisections}
	Let $X$ be a closed, connected $4$--manifold and let $g\geq k$. A \emph{$(g,k)$ trisection} of $X$ is a decomposition $X=X_1\cup X_2\cup X_3$, where each $X_i$ is diffeomorphic to a $4$--dimensional handlebody, with pairwise intersections along a $3$--dimensional handlebody, and triple intersection along a \emph{central surface} $S$ \cite{Gay-Kirby:trisection}, such that:
	\begin{enumerate}
		\item $X_i = \natural_k S^1\x D^3$, $X_i \cap X_j = \natural_g S^1\x D^2$, $S = \Sigma_g$ when $X$ is orientable, and
		\item $X_i = \natural_k S^1\twist \, D^3$, $X_i \cap X_j = \natural_{g} S^1\twist \, D^2$, $S = N_{2g}$ when $X$ is nonorientable,
	\end{enumerate}
	where the indices run through all $i \neq j$ in $\{1, 2, 3\}$. Here $g$ is the \emph{genus} of the trisection. (Note that other sources, e.g. \cite{Spreer-Tillman}, use different conventions when counting the genus of a nonorientable trisection.) We will sometimes drop the $(g,k)$ prefix when the genus is not important. Similar to how Heegaard decompositions of $3$--manifolds correspond to certain Morse functions to $\R$, trisections on a $4$--manifold correspond to certain Morse $2$--functions to $\R^2$ that we call {\it trisection maps} \cite{Gay-Kirby:trisection}. \emph{Simplified trisections} are a special class of trisections that correspond to trisection maps with embedded singular images \cite{Baykur-Saeki:TAMS, Baykur-Saeki:PNAS}. We refer the reader to \cite{Baykur-Saeki:TAMS, Hayano:STS} for more details on simplified trisections, including how they induce rather special \emph{trisection diagrams} (which we do not discuss in this article), as well as on their stable equivalence.

	\subsection{Nonorientable Kirby diagrams}\label{sec:Kirby} 
	Much of the theory developed for Kirby diagrams of orientable $4$--manifolds extends in a straightforward manner to the nonorientable case, though there are a few distinctions to note.

	One can determine if a given handlebody is orientable or not by looking at its $1$--skeleton. A $1$--handle is \emph{untwisted} if it forms a copy of $S^1\times D^3$ when attached to $D^4$. By contrast, a $1$--handle is \emph{twisted} if it forms a copy of $S^1\twist D^3$, the nonorientable $D^3$--bundle over $S^1$, when glued to $D^4$. For an example, consider the Kirby diagram shown on the left in Figure \ref{fig:rp4_rp4}, which has a twisted and an untwisted $1$--handle. The twisted $1$--handle is distinguished from the untwisted one by the pink shading, which indicates an additional reflection in how its attaching feet are identified. If the $x$ and $y$--axes correspond to horizontal and vertical lines on the page, the untwisted $1$--handle attaching feet are identified via $(x,y,z)\sim (x,-y,z)$, whereas the twisted attaching feet are identified via $(x,y,z)\sim (-x,-y,z)$.
	
	A $2$--handle's framing is indicated with an integer coefficient, which gives the linking number between the $2$--handle's attaching circle and its push-off. When a $2$--handle goes over a twisted $1$--handle, we assign a framing coefficient to each arc of the link rather than the entire link. This is because the linking number of an attaching circle and its push-off changes signs when the crossings are pushed through the twisted $1$--handle; see e.g.  \cite{Akbulut1984} (cf. \cite{Miller-Naylor}). In this paper, we adopt the following convention: whenever a $2$--handle goes over a twisted $1$--handle, any labeled arc of its attaching circle has framing {\it along that arc} given by the label. The reader may assume that any unlabeled arcs have zero framing. Note that this convention will be used for any Kirby diagram that has twisted $1$--handles, \emph{including the orientable Kirby diagrams given by Proposition \ref{prop:dble_cover_alg}}. We also write $\tilde h^3$ to denote a twisted $3$--handle, which is an upside-down twisted $1$--handle. 
	
	Recall that if a $4$--manifold $X$ contains an embedded $2$--sphere $S$ with a regular neighborhood a $D^2$--bundle over $S^2$ with Euler number $\pm 1$, then one can \emph{blow down} $S$ to obtain a new $4$--manifold $X':= (X \setminus \nu S) \cup D^4$. In the nonorientable case, we can similarly perform a \emph{nonorientable blow-down} of an embedded real projective plane $R$ whose regular neighborhood is a copy of $\RP\twist D^2$, now to get $X'':= (X \setminus \nu R) \cup \, S^1\twist D^3$ \cite{Akbulut1984}. Each one of these surgeries are realized by trading the two pieces of the decompositions $\CP = \nu \CPo \cup D^4$ and $\mathbb{RP}^4 = \nu \mathbb{RP}^2 \cup S^1\twist D^3$.

	\begin{example} \label{eg:RP4_RP4}
		The left hand side of Figure \ref{fig:rp4_rp4} shows a Kirby diagram for a fiber bundle over $S^2$ with Klein bottle fibers, where one $2$--handle is attached to a copy of $D^2\times Kb$ with nontrivial framing. The following handle moves identify the total space of this bundle as $\mathbb{RP}^4\#\mathbb{RP}^4$: slide the fiber $2$--handle of the obvious copy of $Kb\times D^2$ over the other $2$--handle, so that it slides off the twisted $1$--handle. Then slide the untwisted $1$--handle over the twisted one. The result is shown in the middle of Figure \ref{fig:rp4_rp4}. Removing both of the $2$--handles would amount to performing two nonorientable blow-downs, leaving a total space of $S^1\twist S^3\# S^1\twist S^3$, shown on the right of Figure \ref{fig:rp4_rp4}.
	\end{example}
	
	\begin{figure}
		\centering
		\resizebox{!}{1.5in}{
			\begin{tikzpicture}
				
				\begin{scope}[scale=.5]
					
					\fill[color = pink] (-3.5,0) .. controls +(0,.2775) and +(-.27775,0) .. (-3,.5) .. controls +(.2775,0) and +(0,.2775) .. (-2.5,0);
					\fill[color = pink] (3.5,0) .. controls +(0,-.2775) and +(.27775,0) .. (3,-.5) .. controls +(-.2775,0) and +(0,-.2775) .. (2.5,0);

					\draw[thick] (-3,0) circle [radius = .5cm];
					\draw[thick] (3,0) circle [radius = .5cm];
					\draw[thick] (0,3) circle [radius = .5cm];
					\draw[thick] (0,-3) circle [radius = .5cm];
					
					\begin{knot}[
						%draft mode = crossings,
						clip width = 5pt,
						flip crossing/.list={2}
						]
						\strand[thick] (.5,3) -- (2,3) .. controls +(1,0) and +(0,1) .. (3,2) -- (3,.5);
						\strand[thick, rotate = 90] (.5,3) -- (2,3) .. controls +(1,0) and +(0,1) .. (3,2) -- (3,.5);
						\strand[thick, rotate = 180](.5,3) -- (2,3) .. controls +(1,0) and +(0,1) .. (3,2) -- (3,.5);
						\strand[thick, rotate = -90] (.5,3) -- (2,3) .. controls +(1,0) and +(0,1) .. (3,2) -- (3,.5);

						\strand[thick] ({-3+.25*sqrt(3)},-.25) .. controls +(0,0) and +(-2,0) .. (0,-.5) .. controls +(2,0) and +(0,0).. ({3-.25*sqrt(3)},-.25);
						\strand[thick] ({-3+.25*sqrt(3)},.25) .. controls +(0,0) and +(-2,0) .. (0,.5) .. controls +(2,0) and + (-.5,1) .. (3.25,2.5)  .. controls +(.3,-.75) and +(-.5,.5) .. ({3-.25*sqrt(3)},.25);
					\end{knot}
					
					\filldraw[fill=white,thick] (0,.5) circle [radius = .3cm] (0,.5) node[scale=.7] {$1$};
					\draw (-2,-4.5) node[anchor=west,scale=1.3] {$\cup\, \tilde h^3, h^3, h^4$};
					\draw (-3.2,3.2) node{$0$};

				\end{scope}

				\draw (3.1,0) node {\huge $\cong$};
				\draw (3.1,-.3) node[scale=.7, anchor = north] {handle slides};
				
				\begin{scope}[xshift = 6cm, scale = .5]
					\fill[color = pink] (-3.5,-2) .. controls +(0,.2775) and +(-.27775,0) .. (-3,-1.5) .. controls +(.2775,0) and +(0,.2775) .. (-2.5,-2);
					\fill[color = pink] (3.5,-2) .. controls +(0,-.2775) and +(.27775,0) .. (3,-2.5) .. controls +(-.2775,0) and +(0,-.2775) .. (2.5,-2);
					\draw[thick] (-3,-2) circle [radius = .5cm];
					\draw[thick] (3,-2) circle [radius = .5cm];
					
					\draw[thick] ({-3+.25*sqrt(3)},-2.25) .. controls +(0,0) and +(-2,0) .. (0,-2.75) .. controls +(2,0) and +(0,0).. ({3-.25*sqrt(3)},-2.25);
					\draw[thick] ({-3+.25*sqrt(3)},-1.75) .. controls +(0,0) and +(-2,0) .. (0,-1.25) .. controls +(2,0) and +(0,0).. ({3-.25*sqrt(3)},-1.75);
					
					\filldraw[fill=white,thick] (0,-1.25) circle [radius = .3cm] (0,-1.25) node[scale=.7] {$1$};

					\fill[color = pink] (-3.5,2) .. controls +(0,-.2775) and +(-.27775,0) .. (-3,1.5) .. controls +(.2775,0) and +(0,-.2775) .. (-2.5,2);
					\fill[color = pink] (3.5,2) .. controls +(0,.2775) and +(.27775,0) .. (3,2.5) .. controls +(-.2775,0) and +(0,.2775) .. (2.5,2);
					\draw[thick] (-3,2) circle [radius = .5cm];
					\draw[thick] (3,2) circle [radius = .5cm];
					
					\draw[thick] ({-3+.25*sqrt(3)},2.25) .. controls +(0,0) and +(-2,0) .. (0,2.75) .. controls +(2,0) and +(0,0).. ({3-.25*sqrt(3)},2.25);
					\draw[thick] ({-3+.25*sqrt(3)},1.75) .. controls +(0,0) and +(-2,0) .. (0,1.25) .. controls +(2,0) and +(0,0).. ({3-.25*sqrt(3)},1.75);
					
					\filldraw[fill=white,thick] (0,2.75) circle [radius = .3cm] (0,2.75) node[scale=.7] {$1$};
					
					\draw (-2,-4.5) node[anchor=west,scale=1.3] {$\cup\, \tilde h^3, \tilde h^3, h^4$};
					
				\end{scope}
				
				\draw (9,0) node {\huge $\rightarrow$};
				\draw (9,-.2) node [anchor=north,scale=.7] {blow-down};
				
				\begin{scope}[xshift = 12cm, scale = .5]
					\fill[color = pink] (-3.5,-2) .. controls +(0,.2775) and +(-.27775,0) .. (-3,-1.5) .. controls +(.2775,0) and +(0,.2775) .. (-2.5,-2);
					\fill[color = pink] (2.5,-2) .. controls +(0,-.2775) and +(.27775,0) .. (2,-2.5) .. controls +(-.2775,0) and +(0,-.2775) .. (1.5,-2);
					\draw[thick] (-3,-2) circle [radius = .5cm];
					\draw[thick] (2,-2) circle [radius = .5cm];

					\fill[color = pink] (-3.5,2) .. controls +(0,-.2775) and +(-.27775,0) .. (-3,1.5) .. controls +(.2775,0) and +(0,-.2775) .. (-2.5,2);
					\fill[color = pink] (2.5,2) .. controls +(0,.2775) and +(.27775,0) .. (2,2.5) .. controls +(-.2775,0) and +(0,.2775) .. (1.5,2);
					\draw[thick] (-3,2) circle [radius = .5cm];
					\draw[thick] (2,2) circle [radius = .5cm];

					\draw (-2,-4.5) node[anchor=west,scale=1.3] {$\cup\, \tilde h^3, \tilde h^3, h^4$};
					
				\end{scope}
				
			\end{tikzpicture}
		}

		\caption
        {Left: a Kb--bundle over $S^2$. Middle: the standard Kirby diagram for  $\mathbb{RP}^4\#\mathbb{RP}^4$. Right: a diagram of $S^1\twist S^3\#S^1\twist S^3$, obtained from the middle diagram after a pair of nonorientable blow-downs.}
		\label{fig:rp4_rp4}
	\end{figure}
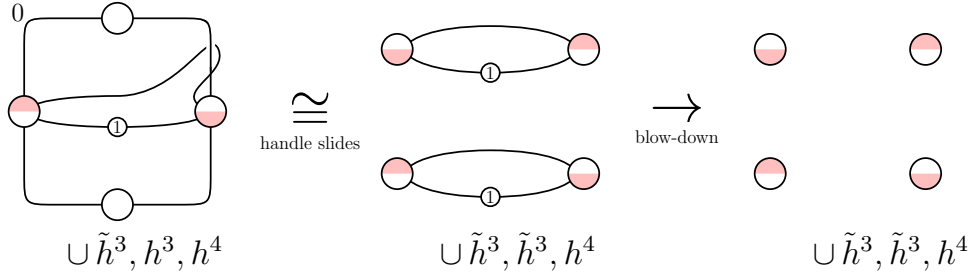

	In the next proposition, we explain how to derive a handle diagram of the orientation double cover of a nonorientable $4$--manifold with a particular type of handle decomposition. This will come in handy when we classify SBLFs with Klein bottle fibers in Section \ref{sec:KB_bundles}.
	
	\begin{proposition} \label{prop:dble_cover_alg}
    Let $N$ be a nonorientable Kirby diagram with one $0$--handle, one twisted $1$--handle, and $k$ untwisted $1$--handles. Let $f\colon N\to \mathbb{R}$ be the associated Morse function, and let $p\colon M\to N$ be an orientation double cover. A Kirby diagram for $M$ with associated Morse function $f\circ p$ is constructed as follows.\begin{enumerate}
			\item Start by drawing two $0$--handles $h_1^0$ and $h_2^0$, separated on the page by a dashed line. On each of these $0$--handles, glue $k$ untwisted $1$--handles, exactly as they're glued to the $0$--handle of $N$.
			
			\item Add two twisted $1$--handles $\tilde h_1^1$ and $\tilde h_2^1$ to the Kirby diagram in the following way: let $a_i$ and $b_i$ be the attaching feet of $\tilde h^1_i$ for $i=1,2$. Put $a_1$ and $a_2$ in $\partial h^0_1$, drawn exactly as the feet of the twisted 1--handle in $N$ are drawn. Do the same for $b_1$ and $b_2$ in $\partial h^0_2$. 
			
			\item Draw all the remaining handles of $N$ (namely, the $n$--handles for $n\geq 2$) on $h^0_1$ and the $1$--handles attached to it, treating $a_1\cup a_2$ as the attaching feet of the twisted $1$--handle in $N$. Do the same for $h^0_2$.
			
		\end{enumerate}
		
	\end{proposition}
	
	\begin{proof}
		Recall that given a nonorientable manifold $N$, the orientation double cover $p\colon M\to N$ corresponds to the kernel of the map $w_1(N)\colon \pi_1(N)\to \ZZ_2$, where $w_1$ is the evaluation of the first Stiefel-Whitney class on $\pi_1$; it sends homotopy classes of orientation--preserving loops to $0$ and orientation--reversing loops to $1$. In other words, we have a short exact sequence
		\begin{center}
			\begin{tikzcd}
				0\arrow[r] &\pi_1(M) \arrow[r,"p_*"] & \pi_1(N)\arrow[r,"w_1(N)"] & \ZZ_2 \arrow[r] &0.
			\end{tikzcd}
		\end{center}
		\noindent Now let $N$ be as in the proposition statement. Let $N^{(m)}$ denote the $m$--skeleton of $N$, which is comprised of all $i$--handles of $N$ for $i\leq m$. By assumption, $N^{(1)}$ has $k$ untwisted $1$--handles and one twisted $1$--handle. So $\pi_1(N)$ has a presentation \[
		\pi_1(N)= \lan a_1,a_2,\cdots a_k, n | r_1,r_2,\cdots r_\ell\ran,
		\] where $n$ is the one generator corresponding to a twisted $1$--handle. Let $N_O^{(1)}$ be all of $N^{(1)}$ except the twisted $1$--handle. It's clear that $N_O^{(1)}\cong \natural_k (S^1\times D^3)$ and it admits a double cover $p^{(1)}_O\colon M^{(1)}_O\to N_O^{(1)}$ which is two disjoint copies of $N_O^{(1)}$ covering one. We may add two twisted $1$--handles to $M^{(1)}_O$ as prescribed in step (2) of the proposition statement\footnote{The careful reader may notice that modulo swapping handles, there are two possible ways to attach the twisted $1$--handle attaching feet to $M^{(1)}_O$ which agree with step (2) of the proposition statement. The diffeomorphism type of the resulting double cover is not affected by this choice.} to obtain $M^{(1)}$, and also add the twisted $1$--handle to $N_O^{(1)}$ to complete it to the 1--skeleton $N^{(1)}$. The cover $p^{(1)}_O$ naturally extends over these new $1$--handles, and its extension is the orientation double--cover $p^{(1)}\colon \natural_{2k+1}(S^1\times D^3)\to \natural_k (S^1\twist D^3).$ The fundamental group of $M^{(1)}$ is generated by $\{a_1^1\cdots a_k^1,a_1^2,\cdots a_k^2,\gamma\}$, where each $a_i^j$ is a lift of $a_i\in \pi_1(N^{(1)})$, and $\gamma$ is the loop formed by both of the twisted $1$--handles in $M^{(1)}$, which maps to $n^2\in \pi_1(N^{(1)})$ under $p^{(1)}$. 
		
		At this point, we can move on to step (3) by adding the $2$--handles of $N$ to $N^{(1)}$, and then add the obvious candidates for their lifts to $M^{(1)}$. Since $M^{(1)}$ is drawn in such a way that it resembles two disjoint copies of $N^{(1)}$, except for a pair of swapped $1$--handle attaching feet, there's an intuitive way to draw a pair of $2$--handle attaching circles on $M^{(1)}$ for every $2$--handle attaching circle drawn on $N^{(1)}$. Since $\pi_1(M^{(2)})$ (resp. $\pi_1(N^{(2)})$) have the same generators as $\pi_1(M^{(1)})$ (resp. $\pi_1(N^{(1)})$), the map $p^{(2)}\colon M^{(2)}\to N^{(2)}$ is completely determined. Moreover, the $2$--handles in $M^{(2)}$ are attached in such a way that each relation in $\pi_1(M^{(2)})$ is sent to a relation in $\pi_1(N^{(2)})$, so this map is well-defined. Since the induced image on $\pi_1$ of $p^{(2)}$ is still $\langle a_1,\cdots a_k, n^2\rangle \subset \pi_1(N^{(2)})\cong \pi_1(N)$, $p^{(2)}$ remains an orientation double cover after this extension over the $2$--handles.
		
	   We have the completed 2--skeleton of $N$, whose boundary is $\#_c (S^1\twist S^2)$ for some $c>0$. It remains to extend the orientation double cover $p^{(2)}$ to all of $N$ by adding $3$-- and $4$--handles. Since $\partial(N^{(2)})\cong \#_c (S^1\twist S^2)$ for some $c\geq 1$, and $\partial(M^{(2)})\cong \#_{2c+1}  (S^1\times S^2)$ is the orientation double cover of $\partial(N^{(2)})$, the orientation double cover $\natural_{2c+1} (S^1\times D^3) \to \natural_c (S^1\twist D^3)$ agrees with $p^{(2)}$ along their common boundaries, and is comprised of only $0$-- and $1$--handles. Dually, this double cover is comprised of only $3$-- and $4$--handles. Turning this double cover upside down and gluing the domain (resp. codomain) to $\partial(M^{(2)})$ (resp. $\partial (N^{(2)})$) completes the orientation double cover construction. 
		
	   In sum, we've constructed an orientation double cover $p\colon M\to N$. The construction of $M$ is exactly as described in the proposition statement. It follows by construction that the Morse function corresponding to $M$ is obtained by post-composing $p$ with the Morse function on $N$.
	\end{proof}
	
	\begin{example}
		Figure \ref{fig:dble_cov_eg} shows the Kirby diagram of the orientation double cover obtained from the diagram of $\RR\PP^4\#\RR\PP^4$ given on the \emph{left} of Figure \ref{fig:rp4_rp4}. The dashed line in the center separates the two $0$--handles. There are four $2$--handles in the orientation double cover, distinguished by their shading.  Although this is a Kirby diagram for an orientable 4-manifold, it nonetheless has lifts of twisted $1$--handles (with now feet on different $0$--handles), and the $2$--handle framing changes sign whenever it's pushed through one of them. For this reason, we continue to assign framing coefficients to arcs rather than entire links. A sequence of Kirby moves turns this diagram into a standard handle diagram of $S^1\times S^3$; first cancel one of the $0$--handles with either of the pink or green-shaded $1$--handles. From there, the gray $2$--handle may be canceled with one of the unshaded $1$--handles. Next slide the black $2$--handle over both the blue and red $2$--handles (which, at this point, have framing coefficients with opposite signs), so that the black $2$--handle becomes a $0$--framed unknot. After sliding the purple $2$--handle over the blue one (or vice versa), everything cancels except for one $i$--handle for $i=0, 1, 3, 4$.

	\end{example}
		\begin{figure} 
		\centering
		
		\resizebox{!}{1.5in}{
			\begin{tikzpicture}
	
            	\begin{scope}[scale=.7]
            		
            		%\draw (-4,-4) grid (4,4);
            		
            		\fill[color = green] (-3.5,0) .. controls +(0,.2775) and +(-.27775,0) .. (-3,.5) .. controls +(.2775,0) and +(0,.2775) .. (-2.5,0);
            		\fill[color = pink] (3.5,0) .. controls +(0,-.2775) and +(.27775,0) .. (3,-.5) .. controls +(-.2775,0) and +(0,-.2775) .. (2.5,0);

            		\draw[thick] (-3,0) circle [radius = .5cm];
            		\draw[thick] (3,0) circle [radius = .5cm];
            		\draw[thick] (0,3) circle [radius = .5cm];
            		\draw[thick] (0,-3) circle [radius = .5cm];
            		
            		\begin{knot}[
            			%draft mode = crossings,
            			clip width = 5pt,
            			flip crossing/.list={2}
            			]
            			\strand[thick,color=gray] (.5,3) -- (2,3) .. controls +(1,0) and +(0,1) .. (3,2) -- (3,.5);
            			\strand[thick, rotate = 90] (.5,3) -- (2,3) .. controls +(1,0) and +(0,1) .. (3,2) -- (3,.5);
            			\strand[thick, rotate = 180](.5,3) -- (2,3) .. controls +(1,0) and +(0,1) .. (3,2) -- (3,.5);
            			\strand[thick, rotate = -90,color=gray] (.5,3) -- (2,3) .. controls +(1,0) and +(0,1) .. (3,2) -- (3,.5);

            			\strand[thick,color=purple] ({-3+.25*sqrt(3)},-.25) .. controls +(0,0) and +(-2,0) .. (0,-.5) .. controls +(2,0) and +(0,0).. ({3-.25*sqrt(3)},-.25);
            			\strand[thick,color=blue] ({-3+.25*sqrt(3)},.25) .. controls +(0,0) and +(-2,0) .. (0,.5) .. controls +(2,0) and + (-.3,.75) .. (3.25,2) .. controls +(.3,-.75) and +(-.5,.5) .. ({3-.25*sqrt(3)},.25);
            		\end{knot}
            		
            		\filldraw[color=blue,fill=white,thick] (0,.5) circle [radius = .3cm] (0,.5) node {$1$};
            		\draw (-1,-4.25) node[scale=1.3,anchor=west] {$\cup\, \tilde h^3, h^3, h^4$};
            		\draw (-3.2,3.2) node{$0$};

            	\end{scope}
            	
            	\draw[thick,dashed] (4.8,-2.7) -- (4.8,2.5);
            	
            	\begin{scope}[xshift = 9cm, scale = .7]
            		
            		%\draw (-4,-4) grid (4,4);
            		
            		\fill[color = pink] (-3.5,0) .. controls +(0,.2775) and +(-.27775,0) .. (-3,.5) .. controls +(.2775,0) and +(0,.2775) .. (-2.5,0);
            		\fill[color = green] (3.5,0) .. controls +(0,-.2775) and +(.27775,0) .. (3,-.5) .. controls +(-.2775,0) and +(0,-.2775) .. (2.5,0);

            		\draw[thick] (-3,0) circle [radius = .5cm];
            		\draw[thick] (3,0) circle [radius = .5cm];
            		\draw[thick] (0,3) circle [radius = .5cm];
            		\draw[thick] (0,-3) circle [radius = .5cm];
            		
            		\begin{knot}[
            			%draft mode = crossings,
            			clip width = 5pt,
            			flip crossing/.list={2}
            			]
            			\strand[thick] (.5,3) -- (2,3) .. controls +(1,0) and +(0,1) .. (3,2) -- (3,.5);
            			\strand[thick, rotate = 90,color = gray] (.5,3) -- (2,3) .. controls +(1,0) and +(0,1) .. (3,2) -- (3,.5);
            			\strand[thick, rotate = 180, color = gray](.5,3) -- (2,3) .. controls +(1,0) and +(0,1) .. (3,2) -- (3,.5);
            			\strand[thick, rotate = -90] (.5,3) -- (2,3) .. controls +(1,0) and +(0,1) .. (3,2) -- (3,.5);

            			\strand[thick,color=blue] ({-3+.25*sqrt(3)},-.25) .. controls +(0,0) and +(-2,0) .. (0,-.5) .. controls +(2,0) and +(0,0).. ({3-.25*sqrt(3)},-.25);
            			\strand[thick,color=purple] ({-3+.25*sqrt(3)},.25) .. controls +(0,0) and +(-2,0) .. (0,.5) .. controls +(2,0) and + (-.3,.75) .. (3.25,2) .. controls +(.3,-.75) and +(-.5,.5) .. ({3-.25*sqrt(3)},.25);
            		\end{knot}
            		
            		\filldraw[color=purple,fill=white,thick] (0,.5) circle [radius = .3cm] (0,.5) node {$1$};
            		\draw (-1,-4.25) node[scale=1.3,anchor=west] {$\cup\, \tilde h^3, h^3, h^4$};
            		\draw[color=gray] (-3.2,3.2) node{$0$};
            	\end{scope}
            	
            \end{tikzpicture}
		}
		\caption{The orientation double cover of $\mathbb{RP}^4\#\mathbb{RP}^4$. The Kirby diagram shown has two $0$--handles, two twisted $1$--handles, two untwisted $1$--handles, and four $2$--handles. The vertical dashed line distinguishes the two $0$--handles, and the $2$--handles are drawn in black, gray, red, and blue. Rotating one of the $0$--handles, all of the $1$--handles become untwisted.} 
		\label{fig:dble_cov_eg}
	\end{figure}
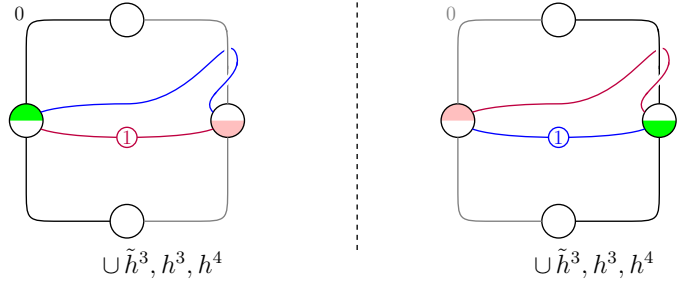

    \begin{example}
		Consider the handlebody for a $2$-skeleton $R_n$ shown on the left in Figure \ref{fig:Rn_1_2}, defined for $n> 1$. Let $B_n$ denote the orientable $2$-handlebody shown on the right in Figure \ref{fig:Rn_1_2}. As described in \cite{fintushel_stern_blowdown}, each $B_n$ is a rational homology $4$--ball with $\pi_1=\Z_n$, and has a lens space $L(n^2,n-1)$ as its boundary. We use Proposition \ref{prop:dble_cover_alg} to show that the orientation double-cover of $R_n$ is $D(B_n)-(S^1\times D^3)$, where $D(B_n)$ is the double of $B_n$. The handle moves to make this identification, illustrated in Figure \ref{fig:Rn_2_2}, are as follows. First, cancel the right-hand $0$--handle with the green $1$--handle to obtain (2). When this happens, the part of the diagram drawn on the canceled $0$--handle reverses orientation. With this, the entire right-hand side of the diagram is vertically reflected and the $2$--handle framing coefficients swap signs. After this, there are no longer twisted $1$--handles. Next slide the black $2$--handle over the blue one to obtain (3). The black $2$--handle may then be isotoped as in (4), so that it goes around the upper $n-1$ strands of the blue $2$--handle. Isotoping the black $2$-handle around (but not through) the left-side $1$--handle attaching foot, it forms a meridian on the bottom-most stand of the blue $2$--handle as in (5). This confirms that the orientation double cover of $R_n$ is $D(B_n) - (S^1\times D^3)$.\footnote{We expect that the nonorientable $4$--manifold $R_n$ is the $2$--skeleton of the $4$--manifolds $N(2n,1,1)$ studied in \cite{teichner-etal}. This would generalize the case of $R_1 = \mathbb{RP}^4$, giving a family of $4$--manifolds admitting perfect Morse functions, answering a question communicated to us by Peter Teichner.}
		
		\begin{figure}
			\centering
			\resizebox{!}{1.5in}{
				\begin{tikzpicture}
					\begin{scope}[scale=.9, xshift = -5cm]
						
						\begin{knot}[
							%draft mode = crossings,
							clip width = 7pt,
							%flip crossing/.list={2, 4,5,6,7}
							]

							\strand[thick] (-3,0) .. controls +(.1,1) and +(-2,0) .. (0,1.25) .. controls +(2,0) and +(-.5,.5).. (3,0);
							\strand[thick] (-3,0) .. controls +(1,.2) and +(-2,0) .. (0,.2) .. controls +(2,0) and + (-.5,.2) .. (2.25,1.5) .. controls +(.5,-.2) and +(-.5,1) .. (3,0) ;
							
							\strand[thick] (-3,0) .. controls +(.1,-1) and +(-2,0) .. (0,-1.25) .. controls +(2,0) and +(-.1,-1).. (3,0);
							\strand[thick] (-3,0) .. controls +(1,-.2) and +(-2,0) .. (0,-.2) .. controls +(2,0) and +(-1,-.2).. (3,0);

						\end{knot}

						% FRAMING AND LABELS
						
						%\filldraw[fill=white] (-.5,.2) circle [radius = .25cm] (-.5,.2) node[scale=.8] {$1$};
						\draw (-1.75,.85) node{$\vdots$};
						\draw (-1.7, .7) node[anchor = west, scale=.75] {($n$)};
						
						\draw  (-1.75,-.5) node {$\vdots$};
						\draw (-1.75,-.7) node[anchor = west, scale =.75] {($n$)};
						
						% 1 HANDLE ATTACHING FEET
						
						\fill[color=white] (-3,0) circle [radius = .5cm];
						\fill[color=white] (3,0) circle [radius = .5cm];
						
						\fill[color = pink] (-3.5,0) .. controls +(0,.2775) and +(-.27775,0) .. (-3,.5) .. controls +(.2775,0) and +(0,.2775) .. (-2.5,0);
						\fill[color = pink] (3.5,0) .. controls +(0,-.2775) and +(.27775,0) .. (3,-.5) .. controls +(-.2775,0) and +(0,-.2775) .. (2.5,0);
						
						\draw[thick] (-3,0) circle [radius = .5cm];
						\draw[thick] (3,0) circle [radius = .5cm];
						
					\end{scope}
					
					\begin{scope} [xshift = 5cm]
						
						% 1 HANDLE 
						\draw[thick] (-.1,-1.05) .. controls +(0,0) and +(-.1,.-.05) .. (0,-.8) .. controls + (.5,0) and +(0,1) .. (.5,-2) .. controls +(0,-1) and + (.5,0).. (0,-3) .. controls +(-.1,0) and +(0,-.2) .. (-.1,-2.55);
						\draw[thick,dashed] (-.1,-2.05) -- (-.1,-1.55);
						\fill[black](0,-2.95) circle[radius = 4pt];
						
						% 2 HANDLE STRANDS
						
						\draw[thick] (.35,-2.5) .. controls +(-1,0) and +(0,-1) .. (-3,0) .. controls +(0,1) and +(-1,0) .. (0,2) .. controls +(1,0) and +(0,1).. (3,0) ..controls +(0,-1) and + (1,0) .. (.6,-2.25);
						\draw[thick] (.35,-2.25) .. controls +(-1,0) and +(0,-1) .. (-2.75,0) .. controls +(0,1) and +(-1,0) .. (0,1.75) .. controls +(1,0) and +(0,1) .. (2.75,0) .. controls +(0,-1) and +(1,0) .. (.6,-2);
						\draw[thick] (.35, -1.5) ..controls +(-.5, 0) and +(0,-.5) .. (-1.75,0) .. controls +(0,.5) and +(-.5,0) .. (0,1) .. controls +(.5,0) and +(0,.5) .. (1.75,0) .. controls +(0,-.5) and +(.5,0) .. (.6,-1.25);
						\draw[thick] (.35, -1.25) .. controls +(-.5,0) and +(0,-.5) .. (-1.5,0) .. controls +(0,.5) and +(-.5,0) .. (-.25,.7) ..controls +(0,0) and +(-.15,-.15).. (.5,.75);
						
						\draw[thick] (1.75,1.55) .. controls +(.25,.25) and +(0,1) .. (3.25,0) .. controls +(0,-1) and +(1,0) .. (.6,-2.5);
						
						\draw[thick,dashed] (.75,.95) -- (1.4,1.3);
						% LABELS
						
						\draw (-2.2,0) node[scale=1.2]{$\cdots$};
						\draw (-2.2,.4) node[scale=.8] {$(n)$};
					\end{scope}

				\end{tikzpicture}
			}
			\caption{Left: the 2-skeleton $R_n$. On the upper half of the Kirby diagram, the higher $n-1$ strands run parallel, and the lowest one goes behind all of them. In the lower half of the diagram, all $n$ strands run parallel. Right: the rational homology ball $B_n$. The $2$--handle passes over the $1$--handle $n$ times. In both Kirby diagrams shown, the $2$--handles have blackboard framing.}
			\label{fig:Rn_1_2}
		\end{figure}
		
		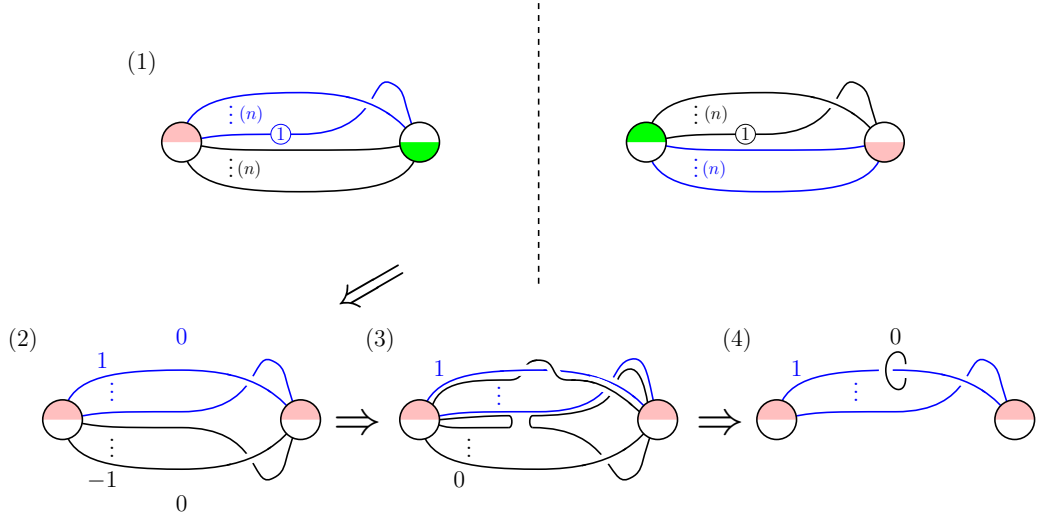
\begin{figure}
			\centering
			\resizebox{5.5in}{!}{
				
				\begin{tikzpicture}
					
					% ONE
					
					\begin{scope}[scale=.7, xshift = -6cm]
						
						\draw (-4,2) node {(1)};
						
						\begin{knot}[
							%draft mode = crossings,
							clip width = 7pt,
							%flip crossing/.list={2, 4,5,6,7}
							]

							\strand[thick,color=blue] (-3,0) .. controls +(.1,1) and +(-2,0) .. (0,1.25) .. controls +(2,0) and +(-.5,.5).. (3,0);
							\strand[thick,color=blue] (-3,0) .. controls +(1,.2) and +(-2,0) .. (0,.2) .. controls +(2,0) and + (-.5,.2) .. (2.25,1.5) .. controls +(.5,-.2) and +(-.5,1) .. (3,0) ;
							
							\strand[thick] (-3,0) .. controls +(.1,-1) and +(-2,0) .. (0,-1.25) .. controls +(2,0) and +(-.1,-1).. (3,0);
							\strand[thick] (-3,0) .. controls +(1,-.2) and +(-2,0) .. (0,-.2) .. controls +(2,0) and +(-1,-.2).. (3,0);

						\end{knot}

						% FRAMING AND LABELS
						
						\draw (-1.75,.85) node[color=blue]{$\vdots$};
						\draw (-1.7, .7) node[anchor = west, scale=.75,color=blue] {($n$)};
						
						\draw  (-1.75,-.5) node {$\vdots$};
						\draw (-1.75,-.7) node[anchor = west, scale =.75] {($n$)};
						
						% 1 HANDLE ATTACHING FEET
						
						\fill[color=white] (-3,0) circle [radius = .5cm];
						\fill[color=white] (3,0) circle [radius = .5cm];
						
						\fill[color = pink] (-3.5,0) .. controls +(0,.2775) and +(-.27775,0) .. (-3,.5) .. controls +(.2775,0) and +(0,.2775) .. (-2.5,0);
						\fill[color = green] (3.5,0) .. controls +(0,-.2775) and +(.27775,0) .. (3,-.5) .. controls +(-.2775,0) and +(0,-.2775) .. (2.5,0);
						
						\draw[thick] (-3,0) circle [radius = .5cm];
						\draw[thick] (3,0) circle [radius = .5cm];
						
					\end{scope}
					
					\draw[thick,dashed] (0,-2.5) -- (0,2.5);
					
					\begin{scope}[xshift = 4cm, scale = .7]
						
						\begin{knot}[
							%draft mode = crossings,
							clip width = 7pt,
							%flip crossing/.list={2, 4,5,6,7}
							]

							\strand[thick] (-3,0) .. controls +(.1,1) and +(-2,0) .. (0,1.25) .. controls +(2,0) and +(-.5,.5).. (3,0);
							\strand[thick] (-3,0) .. controls +(1,.2) and +(-2,0) .. (0,.2) .. controls +(2,0) and + (-.5,.2) .. (2.25,1.5) .. controls +(.5,-.2) and +(-.5,1) .. (3,0) ;
							
							\strand[thick,color=blue] (-3,0) .. controls +(.1,-1) and +(-2,0) .. (0,-1.25) .. controls +(2,0) and +(-.1,-1).. (3,0);
							\strand[thick,color=blue] (-3,0) .. controls +(1,-.2) and +(-2,0) .. (0,-.2) .. controls +(2,0) and +(-1,-.2).. (3,0);

						\end{knot}

						% FRAMING AND LABELS
						
						\draw (-1.75,.85) node{$\vdots$};
						\draw (-1.7, .7) node[anchor = west, scale=.75] {($n$)};
						
						\draw  (-1.75,-.5) node[color=blue] {$\vdots$};
						\draw (-1.75,-.7) node[anchor = west, scale =.75,color=blue] {($n$)};
						
						% 1 HANDLE ATTACHING FEET
						
						\fill[color=white] (-3,0) circle [radius = .5cm];
						\fill[color=white] (3,0) circle [radius = .5cm];
						
						\fill[color = green] (-3.5,0) .. controls +(0,.2775) and +(-.27775,0) .. (-3,.5) .. controls +(.2775,0) and +(0,.2775) .. (-2.5,0);
						\fill[color = pink] (3.5,0) .. controls +(0,-.2775) and +(.27775,0) .. (3,-.5) .. controls +(-.2775,0) and +(0,-.2775) .. (2.5,0);
						
						\draw[thick] (-3,0) circle [radius = .5cm];
						\draw[thick] (3,0) circle [radius = .5cm];
						
					\end{scope}
					
					\draw (-3,-2.5) node[scale=2,rotate = 210] {$\Longrightarrow$};
					% TWO
					
					\begin{scope}[scale=.7, xshift = -9cm, yshift = -7cm]
						\draw (-4,2) node {(2)};
						
						\begin{knot}[
							%draft mode = crossings,
							clip width = 7pt,
							%flip crossing/.list={2, 4,5,6,7}
							]

							\strand[thick,color=blue] (-3,0) .. controls +(.1,1) and +(-2,0) .. (0,1.25) .. controls +(2,0) and +(-.5,.5).. (3,0);
							\strand[thick,color=blue] (-3,0) .. controls +(1,.2) and +(-2,0) .. (0,.2) .. controls +(2,0) and + (-.5,.2) .. (2.25,1.5) .. controls +(.5,-.2) and +(-.5,1) .. (3,0) ;

						\end{knot}

						% REFLECTED KNOT
						
						\begin{scope}[yscale=-1]
							
							\begin{knot}[
								%draft mode = crossings,
								clip width = 7pt,
								%flip crossing/.list={2, 4,5,6,7}
								]

								\strand[thick] (-3,0) .. controls +(.1,1) and +(-2,0) .. (0,1.25) .. controls +(2,0) and +(-.5,.5).. (3,0);
								\strand[thick] (-3,0) .. controls +(1,.2) and +(-2,0) .. (0,.2) .. controls +(2,0) and + (-.5,.2) .. (2.25,1.5) .. controls +(.5,-.2) and +(-.5,1) .. (3,0) ;

							\end{knot}
							
						\end{scope}
						
						% ELLIPSES AND FRAMING
						\draw (-1.75,-.55) node[]{$\vdots$};
						\draw (-1.75,.85) node[color=blue]{$\vdots$};
						%\draw (0,1.7) node[anchor=south,color=blue] {$0$};
						%\draw (0,-1.7) node[anchor=north] {$0$};
						\draw (-2,1.1) node[anchor=south, color=blue] {$0$};
						\draw (-2,-1.1) node[anchor=north] {$0$};
						
						% 1 HANDLE ATTACHING FEET
						
						\fill[color=white] (-3,0) circle [radius = .5cm];
						\fill[color=white] (3,0) circle [radius = .5cm];
						
						\fill[color = pink] (-3.5,0) .. controls +(0,.2775) and +(-.27775,0) .. (-3,.5) .. controls +(.2775,0) and +(0,.2775) .. (-2.5,0);
						\fill[color = pink] (3.5,0) .. controls +(0,.2775) and +(.27775,0) .. (3,.5) .. controls +(-.2775,0) and +(0,.2775) .. (2.5,0);
						
						\draw[thick] (-3,0) circle [radius = .5cm];
						\draw[thick] (3,0) circle [radius = .5cm];
						
					\end{scope}
					
					\draw (-3.2,-5) node[scale=2] {$\Rightarrow$};
					
					% THREE
					
					\begin{scope}[scale=.7, xshift =0cm, yshift = -7cm]
						
						\draw (-4,2) node {(3)};
						
						\begin{knot}[
							%draft mode = crossings,
							clip width = 7pt,
							flip crossing/.list={3,4}
							]

							\strand[thick,color=blue] (-3,0) .. controls +(.25,1) and +(-2,0) .. (0,1.25) .. controls +(2,0) and +(-.5,.5).. (3,0);
							\strand[thick,color=blue] (-3,0) .. controls +(1,.2) and +(-2,0) .. (0,.2) .. controls +(2,0) and + (-.5,-.2) .. (2.25,1.5) .. controls +(.5,.2) and +(-.25,1) .. (3,0) ;
							
							% HANDLESLIDE
							
							\strand[thick] (-.15,-.2) .. controls +(-.1,0) and +(-.1,0) .. (-.15, .1) .. controls +(2.5,0) and +(-.5,-.25) .. (2.2,1.25) .. controls + (.5,.25) and + (0, 0) .. (2.8,0);

							\strand[thick] (-.75,-.2) .. controls +(.1,0) and +(.1,0) .. (-.75,.1) .. controls +(-1,0) and + (-1,-.1) .. (-2.5,0);

							%\strand[thick] (-2.73,0) .. controls +(0,1) and +(-1,0) .. (-.75, 1) .. controls+(.5,0) and +(-.5,0) .. (0,1.5) .. controls+(.5,0) and +(-.5,0) .. (.75,1) ..controls +(.5,0) and +(-1,1) .. (2.75,0);
							\strand[thick] (-2.73,0) .. controls +(0,1) and +(-1,0) .. (0, 1)  ..controls +(1,0) and +(-1,1) .. (2.75,0);

						\end{knot}
						
						% REFLECTED KNOT
						
						\begin{scope}[yscale=-1]
							
							\begin{knot}[
								%draft mode = crossings,
								clip width = 7pt,
								%flip crossing/.list={2, 4,5,6,7}
								]

								\strand[thick] (-3,0) .. controls +(.1,1) and +(-2,0) .. (0,1.25) .. controls +(2,0) and +(-.5,.5).. (3,0);

								\strand[thick] (-3,0) .. controls +(1,.25) and +(-2,0) .. (-.75,.2);
								
								\strand[thick] (-.15,.2) .. controls +(2,0) and + (-.5,.2) .. (2.25,1.5) .. controls +(.5,-.2) and +(-.5,1) .. (3,0) ;

							\end{knot}

						\end{scope}

						% ELLIPSES AND FRAMING
						\draw (-1.75,-.55) node[]{$\vdots$};
						\draw (-1,.75) node[color=blue]{$\vdots$};
						\draw (-2.5,.9) node[anchor=south, color=blue] {$0$};
						\draw (-2,-1.1) node[anchor=north] {$0$};
						
						% 1 HANDLE ATTACHING FEET
						
						\fill[color=white] (-3,0) circle [radius = .5cm];
						\fill[color=white] (3,0) circle [radius = .5cm];
						
						\fill[color = pink] (-3.5,0) .. controls +(0,.2775) and +(-.27775,0) .. (-3,.5) .. controls +(.2775,0) and +(0,.2775) .. (-2.5,0);
						\fill[color = pink] (3.5,0) .. controls +(0,.2775) and +(.27775,0) .. (3,.5) .. controls +(-.2775,0) and +(0,.2775) .. (2.5,0);
						
						\draw[thick] (-3,0) circle [radius = .5cm];
						\draw[thick] (3,0) circle [radius = .5cm];
						
					\end{scope}
					
					\draw (3.2,-5) node[scale=2] {$\Rightarrow$};
					% FOUR
					
					\begin{scope}[scale=.7, xshift = 9cm, yshift = -7cm]
						
						\draw (-4,2) node {(4)};
						
						\begin{knot}[
							%draft mode = crossings,
							clip width = 7pt,
							flip crossing/.list={3,4}
							]

							\strand[thick,color=blue] (-3,0) .. controls +(.1,1) and +(-2,0) .. (0,1.25) .. controls +(2,0) and +(-.5,.8).. (3,0);
							\strand[thick,color=blue] (-3,.25) .. controls +(1,0) and +(-2,0) .. (0,.6) .. controls +(2,0) and +(-1,0).. (3,.25);

							\strand[thick,color=blue] (-3,0) .. controls +(1,.2) and +(-2,0) .. (0,.2) .. controls +(2,0) and + (-.5,.2) .. (2.25,1.5) .. controls +(.5,-.2) and +(-.5,1) .. (3,0) ;

						\end{knot}
						
						% MERIDIAN 2 HANDLE
						
						\fill[white](-.25,1.25) circle[radius = 5pt];
						\fill[white](-.25,.6) circle[radius = 5pt];
						\draw[thick] (.25,1.4) .. controls +(-.05,.5) and +(0,.5) .. (-.25,1.25) .. controls + (0,-.5) and +(-.3,0) .. (0,.35) .. controls +(.1,0) and + (-.05,-.1) .. (.25,.45); 
						\draw[thick] (.25,.7) .. controls +(.02,.1) and + (.02,-.1) .. (.25,1.15);

						% FRAMING AND LABELS

						\draw (-1,1) node[color=blue,scale=.8]{$\vdots$};
						\draw (0,1.7) node[anchor=south] {$0$};
						\draw (-2.5,.9) node[anchor=south, color=blue] {$0$};
						
						% 1 HANDLE ATTACHING FEET
						
						\fill[color=white] (-3,0) circle [radius = .5cm];
						\fill[color=white] (3,0) circle [radius = .5cm];
						
						\fill[color = pink] (-3.5,0) .. controls +(0,.2775) and +(-.27775,0) .. (-3,.5) .. controls +(.2775,0) and +(0,.2775) .. (-2.5,0);
						\fill[color = pink] (3.5,0) .. controls +(0,.2775) and +(.27775,0) .. (3,.5) .. controls +(-.2775,0) and +(0,.2775) .. (2.5,0);
						
						\draw[thick] (-3,0) circle [radius = .5cm];
						\draw[thick] (3,0) circle [radius = .5cm];
						
						% ARROW
						\draw (-4,-2.5) node[scale=2,rotate = 240] {$\Longrightarrow$};
						
					\end{scope}
					
					% FIVE 

					\begin{scope}[scale=.7, xshift = 2cm, yshift = -13cm]
						
						\draw (-4,2) node {(5)};
						
						\begin{knot}[
							%draft mode = crossings,
							clip width = 7pt,
							flip crossing/.list={3,4}
							]

							\strand[thick,color=blue] (-3,0) .. controls +(.1,1) and +(-2,0) .. (0,1.25) .. controls +(2,0) and +(-.5,.8).. (3,0);
							\strand[thick,color=blue] (-3,.25) .. controls +(1,0) and +(-2,0) .. (0,.6) .. controls +(2,0) and +(-1,0).. (3,.25);

							\strand[thick,color=blue] (-3,0) .. controls +(1,.2) and +(-2,0) .. (0,.2) .. controls +(2,0) and + (-.5,.2) .. (2.25,1.5) .. controls +(.5,-.2) and +(-.5,1) .. (3,0) ;

						\end{knot}
						
						% MERIDIAN 2 HANDLE
						
						\fill[white](-.25,.2) circle[radius = 5pt];
						\draw[thick] (-.75,.3) ..controls +(0,.2) and +(0,.3) .. (-.25,.2) .. controls +(0,-.3) and +(0,-.2).. (-.75,.1);

						% FRAMING AND LABELS

						\draw (-1,1) node[color=blue,scale=.8]{$\vdots$};
						\draw (-.5,-.2) node[anchor=north] {$0$};
						\draw (-2.5,.9) node[anchor=south, color=blue] {$0$};
						
						% 1 HANDLE ATTACHING FEET
						
						\fill[color=white] (-3,0) circle [radius = .5cm];
						\fill[color=white] (3,0) circle [radius = .5cm];
						
						\fill[color = pink] (-3.5,0) .. controls +(0,.2775) and +(-.27775,0) .. (-3,.5) .. controls +(.2775,0) and +(0,.2775) .. (-2.5,0);
						\fill[color = pink] (3.5,0) .. controls +(0,.2775) and +(.27775,0) .. (3,.5) .. controls +(-.2775,0) and +(0,.2775) .. (2.5,0);
						
						\draw[thick] (-3,0) circle [radius = .5cm];
						\draw[thick] (3,0) circle [radius = .5cm];
						
					\end{scope}

				\end{tikzpicture}
			}
			\caption{The orientation double cover of $R_n$.}
			\label{fig:Rn_2_2}
		\end{figure}
		
	\end{example}

	\subsection{Topology of nonorientable SBLFs}\label{sec:SBLF_topology} 
	We proceed to describe the handle diagrams for nonorientable SBLFs. For a thorough account of the orientable case and details on round $2$--handle attachments, we refer the reader to \cite{baykurPJM}.

	By the assumptions (i) and (ii) of Definition~\ref{def:SBLF} for a simplified broken Lefschetz fibration $f \colon X \to S^2$, the \emph{round locus} $f(Z_f)$, when nonempty, is an embedded circle with an annulus neighborhood $A$, splitting the base as $S^2=D^2_+ \cup A \cup D^2_-$. This yields a particular decomposition of $X= X_+ \cup X_0 \cup X_-$, where $X_\pm$ is a (achiral) Lefschetz fibration over the disk $D^2_\pm$, and $X_0$ is a cobordism between their boundary fibrations, prescribed by an $S^1$--parameter of $1$--handle attachments to the regular fibers over $ \partial D^2_-$. Note that by (iv), $X_- \to D^2_-$ is a \emph{trivial} fibration. Furthermore, by (ii) and (iii), we have  the following possibilities for the diffeomorphism types of the regular fibers $F_\pm$ over $D^2_\pm$: 
	
	\begin{enumerate}
		\item $F_+ = \Sigma_g$ and $F_- = \Sigma_{g-1}$, when $X$ is orientable, and 
		\item $F_+ = N_k$ and $F_- = N_{k-2}$ or $\Sigma_{g-1}$ for $k=2g$, when $X$ is nonorientable.
	\end{enumerate}
	It is a good mental exercise to check that these are all the possible cases. We refer to $g$ (in the orientable case) or $k$ (in the nonorientable case) as the \emph{genus} of the SBLF $f \colon X \to S^2$. We are now ready to describe a handle diagram of a nonorientable SBLF, which we do in three steps. 
	
	We start with the usual handle diagram for the Lefschetz fibration $X_+ \to D^2_+$. This is obtained from the standard handle decomposition for $F_+ \x D^2 \cong N_k \x D^2$ by attaching a $2$--handle with fiber framing $\pm 1$ corresponding to each Lefschetz singularity \cite[Section 8]{Gompf-Stipsicz}. Even though these attaching circles, namely the \emph{vanishing cycles} $\{c_1, \ldots, c_n\}$ of the Lefschetz thimbles, are necessarily two-sided curves on $N_k$, without a canonical orientation to choose, there is no distinction we can make between $+1$ and $-1$ framing. See \cite{Miller-Ozbagci} for a more detailed discussion of these caveats for nonorientable Lefschetz fibrations. 
	
	Next, we attach a round $2$--handle to $\partial X_+$, as prescribed by the indefinite fold $Z_f$, to obtain a handle decomposition for $X_+ \cup X_0$. This consists of a $2$--handle $h_2$ attached to the fiber along a $2$--sided curve $c$ with zero framing, followed by a $3$--handle $h_3$ that goes over the $2$--handle geometrically twice. By our assumption on the connectivity of the fibers, $c$ is nonseparating. It is possible that $\nu(Z_f)$ is a copy of $S^1 \twprod D^3$ in $X$, in which case we have a twisted $3$--handle attachment. This necessarily occurs when $F_-$ is orientable.
	
	The remaining piece $X_-$ has the standard handle decomposition for $F_- \times D^2$. When added upside down to the handle decomposition of $X_+ \cup X_0$, this amounts to attaching a single $2$--handle $h'_2$, several $3$--handles (namely, $k-2$ of them), and a $4$--handle. Per the nonorientable Laudenbach--Po\'{e}naru theorem \cite{cesardesa:thesis, Miller-Naylor}, there is a unique way to extend a handlebody over $3$-- and $4$--handles, up to diffeomorphism. Therefore, it suffices to draw the Kirby diagram for $X_+ \cup h_2 \cup h'_2$, which entirely captures the diffeomorphism type of $X$, as well as the monodromy of the fibration by the ordered tuple of circles $(c; c_1, \ldots, c_n)$. Notice how we threw the $3$--handle $h_3$ of the round $2$--handle into the pile of the other $3$--handles here; this is because one can always attach the lower index handles (here $h'_2$) before the higher index ones. 
	
	The process above outlines an explicit construction for SBLF Kirby diagrams. A couple of remarks are in order. As in the case of orientable (SB)LFs, there is no canonical way to draw $h'_2$ without more information on the fibration. For instance, if the SBLF has a section, the attaching circle of $h'_2$ would be a certain meridian to the attaching circle of the unique $2$--handle of $F_+ \x D^2$; cf. examples in \cite{Gompf-Stipsicz, baykurPJM}. (In this case, the attaching circle of $h'_2$ doesn't go over the $1$--handles.) More generally, fiber-preserving gluings between the boundaries of $X_\pm$ and $X_0$ are unique up to isotopy when the fiber is $\Sigma_g$ with $g \geq 2$ or $N_k$ with $k \geq 3$ by the work of Earle and Eells \cite{Earle_Eells}. So although it may not be obvious how to attach $h_2'$, there is a unique gluing choice of $X_-$ in these cases. On the other hand, for lower genus cases, the various attaching maps must be taken into consideration, which manifests in nonstandard attachments of $h'_2$ to the rest of the diagrams. These nonstandard gluings will be explored further in Section \ref{sec:classification}.
	
	\begin{example}\label{eg:notall}
		Consider the handle diagram of a BLF on $\#_g S^1\twist S^3$ as shown in Figure \ref{fig:notall_eg}, which is a nonorientable analog of the \emph{step fibration in \cite{baykurPJM}.} When $g=1$, there's only one indefinite fold circle, giving an SBLF on $S^1\twist S^3$ with higher and lower genus side fibers $F_+ \cong N_4$ and $F_- \cong T^2$ and no Lefschetz singularities. By applying a nonorientable blow-up, we obtain an SBLF on $\mathbb{RP}^4$ with the same regular fibers as the SBLF on $S^1\twist S^3$, but with one additional Lefschetz singularity, attached along the twisted $1$--handle. More generally, for all $k\leq g$, we obtain a BLF on $(\#_{g-k} S^1\twist S^3) \# (\#_k \mathbb{RP}^4)$ by applying $k$ nonorientable blow-ups to $\#_g S^1\twist S^3$.
		
		\begin{figure}
			\centering
			\resizebox{5.5in}{!}{
				\begin{tikzpicture}
					\begin{scope}[scale=.5]
                    
						\begin{knot}[
							%draft mode = crossings,
							clip width = 7pt,
							flip crossing/.list={1}
							]
							\strand[thick] (.5,3) -- (2,3) .. controls +(1,0) and +(0,1) .. (3,2) -- (3,.5);
							\strand[thick, rotate = 90] (.5,3) -- (2,3) .. controls +(1,0) and +(0,1) .. (3,2) -- (3,.5);
							\strand[thick, rotate = 180](.5,3) -- (2,3) .. controls +(1,0) and +(0,1) .. (3,2) -- (3,.5);
							
							\strand[thick,color = red] (.4,3) -- (.4,-3);
							\strand[thick,color = red] (-.4,3) -- (-.4,-3);

							\strand[thick] ({-3+.25*sqrt(3)},.25) .. controls +(1,0) and +(-3,0) .. (0,4) .. controls +(3,0) and +(-1,0).. ({3-.25*sqrt(3)},.25);
						\end{knot}
						
						\draw (-1.3,-1.5) node[color=red]{$1$};
						\draw (1.3, -1.5) node[color=red]{$-1$};
						\draw (0,4.5) node{$1$};
						
						\filldraw[thick,fill=white] (-3,0) circle [radius = .5cm];
						\filldraw[thick,fill=white] (3,0) circle [radius = .5cm];
						\filldraw[thick,fill=white] (0,3) circle [radius = .5cm];
						\filldraw[thick,fill=white] (0,-3) circle [radius = .5cm];
						
						\begin{scope}[xshift=8cm]
							
							\begin{knot}[
								%draft mode = crossings,
								clip width = 7pt,
								flip crossing/.list={}
								]
								\strand[thick] (.5,3) -- (2,3) .. controls +(1,0) and +(0,1) .. (3,2) -- (3,.5);
								\strand[thick, rotate = 90] (.5,3) -- (2,3) .. controls +(1,0) and +(0,1) .. (3,2) -- (3,.5);

								\strand[thick,color = blue] (0,2.5) -- (0,-2.5);
								
								\strand[thick,color=red] (-3,.4) -- (3,.4);
								\strand[thick,color=red] (-3,-.4) -- (3,-.4);
								
							\end{knot}
							
							\filldraw[fill=white,draw=red] (-2, -.4) circle[radius=.25cm] (-2,-.4) node[scale=.6,color=red]{$1$};
							
							\draw (.5,2) node[color=blue] {$0$};
							
							\filldraw[thick,fill=white] (-3,0) circle [radius = .5cm];
							\filldraw[thick,fill=white] (3,0) circle [radius = .5cm];
							\filldraw[thick,fill=white] (0,3) circle [radius = .5cm];
							\filldraw[thick,fill=white] (0,-3) circle [radius = .5cm];
							
							\fill[color = pink] (-3.5,0) .. controls +(0,.2775) and +(-.27775,0) .. (-3,.5) .. controls +(.2775,0) and +(0,.2775) .. (-2.5,0);
							\fill[color = pink] (3.5,0) .. controls +(0,-.2775) and +(.27775,0) .. (3,-.5) .. controls +(-.2775,0) and +(0,-.2775) .. (2.5,0);
							
							\draw[thick] (-3,0) circle [radius = .5cm];
							\draw[thick] (3,0) circle [radius = .5cm];
							\draw[thick] (0,3) circle [radius = .5cm];
							\draw[thick] (0,-3) circle [radius = .5cm];
							
						\end{scope}
						
						\begin{scope}[xshift=17cm]
							
							\begin{knot}[
								%draft mode = crossings,
								clip width = 7pt,
								flip crossing/.list={}
								]
								\strand[thick] (.5,3) -- (2,3) .. controls +(1,0) and +(0,1) .. (3,2) -- (3,.5);
								\strand[thick, rotate = 90] (.5,3) -- (2,3) .. controls +(1,0) and +(0,1) .. (3,2) -- (3,.5);

								\strand[thick,color = blue] (0,2.5) -- (0,-2.5);
								
								\strand[thick,color=red] (-3,.4) -- (3,.4);
								\strand[thick,color=red] (-3,-.4) -- (3,-.4);
								
							\end{knot}
							
							\filldraw[fill=white,draw=red] (-2, -.4) circle[radius=.25cm] (-2,-.4) node[scale=.6,color=red]{$1$};
							
							\draw (.5,2) node[color=blue] {$0$};
							
							\filldraw[thick,fill=white] (-3,0) circle [radius = .5cm];
							\filldraw[thick,fill=white] (3,0) circle [radius = .5cm];
							\filldraw[thick,fill=white] (0,3) circle [radius = .5cm];
							\filldraw[thick,fill=white] (0,-3) circle [radius = .5cm];
							
							\fill[color = pink] (-3.5,0) .. controls +(0,.2775) and +(-.27775,0) .. (-3,.5) .. controls +(.2775,0) and +(0,.2775) .. (-2.5,0);
							\fill[color = pink] (3.5,0) .. controls +(0,-.2775) and +(.27775,0) .. (3,-.5) .. controls +(-.2775,0) and +(0,-.2775) .. (2.5,0);
							
							\draw[thick] (-3,0) circle [radius = .5cm];
							\draw[thick] (3,0) circle [radius = .5cm];
							\draw[thick] (0,3) circle [radius = .5cm];
							\draw[thick] (0,-3) circle [radius = .5cm];
							
						\end{scope}
						
						\begin{scope}[xshift=26cm]
							
							\begin{knot}[
								%draft mode = crossings,
								clip width = 7pt,
								flip crossing/.list={2,3,4}
								]
								\strand[thick] (.5,3) -- (2,3) .. controls +(1,0) and +(0,1) .. (3,2) -- (3,.5);
								\strand[thick, rotate = 90] (.5,3) -- (2,3) .. controls +(1,0) and +(0,1) .. (3,2) -- (3,.5);
								\strand[thick, rotate = -90] (.5,3) -- (2,3) .. controls +(1,0) and +(0,1) .. (3,2) -- (3,.5);

								\strand[thick,color = blue] (0,2.5) -- (0,-2.5);
							\end{knot}
							
							\draw (.5,2) node[color=blue] {$0$};
							
							\filldraw[thick,fill=white] (-3,0) circle [radius = .5cm];
							\filldraw[thick,fill=white] (3,0) circle [radius = .5cm];
							\filldraw[thick,fill=white] (0,3) circle [radius = .5cm];
							\filldraw[thick,fill=white] (0,-3) circle [radius = .5cm];
							
							\fill[color = pink] (-3.5,0) .. controls +(0,.2775) and +(-.27775,0) .. (-3,.5) .. controls +(.2775,0) and +(0,.2775) .. (-2.5,0);
							\fill[color = pink] (3.5,0) .. controls +(0,-.2775) and +(.27775,0) .. (3,-.5) .. controls +(-.2775,0) and +(0,-.2775) .. (2.5,0);
							
							\draw[thick] (-3,0) circle [radius = .5cm];
							\draw[thick] (3,0) circle [radius = .5cm];
							\draw[thick] (0,3) circle [radius = .5cm];
							\draw[thick] (0,-3) circle [radius = .5cm];
							
						\end{scope}
						
						\draw[thick] (.5,-3) -- (7.5,-3);
						\draw[thick] (3,-.5) .. controls +(0,-1) and +(-1,0).. (4,-1.5) .. controls +(1,0) and +(0,-1) ..  (5,-.5);
						\draw[thick] (11,-.5) .. controls +(0,-1) and +(-.5,0).. (11.5,-1.5);
						\draw[thick] (8.5,-3) -- (11.5,-3);
						\draw[thick] (14,-.5) .. controls +(0,-1) and +(.5,0) .. (13.5,-1.5);
						\draw[thick](16.5,-3) -- (13.5,-3);
						
						\draw (12.6,-2.2) node[scale=1.5] {$\cdots$};
						
						\begin{scope}[xshift = 9cm]
							\draw[thick] (11,-.5) .. controls +(0,-1) and +(-.5,0).. (11.5,-1.5);
							\draw[thick] (8.5,-3) -- (11.5,-3);
							\draw[thick] (14,-.5) .. controls +(0,-1) and +(.5,0) .. (13.5,-1.5);
							\draw[thick](16.5,-3) -- (13.5,-3);
							
							\draw (12.6,-2.2) node[scale=1.5] {$\cdots$};
						\end{scope}

						\draw (-3.2,3.2) node{$0$};
						
						\draw (17,-5.7) node[anchor=north,scale=.8] {$g$ of these};
						
						\draw [decorate,
						decoration = {brace,mirror, amplitude=10pt}, thick] (4.5,-5) --  (29.5,-5);
						
						\draw [decorate,
						decoration = {brace,mirror, amplitude=10pt}, thick] (4.5,-3.5) --  (20.5,-3.5);
						
						\draw(12.5, -4.2) node[scale=.8, anchor=north] {$k$ of these};
						
						\draw (27, 4.5) node {$\cup\, \textcolor{blue}{g\tilde h^3}, 2h^3, h^4$};
						
					\end{scope}

				\end{tikzpicture}
			}
			\caption{A BLF on $(\#_{g-k} S^1\twist S^3) \# (\#_k \mathbb{RP}^4)$. Figure \ref{fig:notall_htpy} shows the base diagram of this BLF when $k=0$.}
			\label{fig:notall_eg}
		\end{figure}
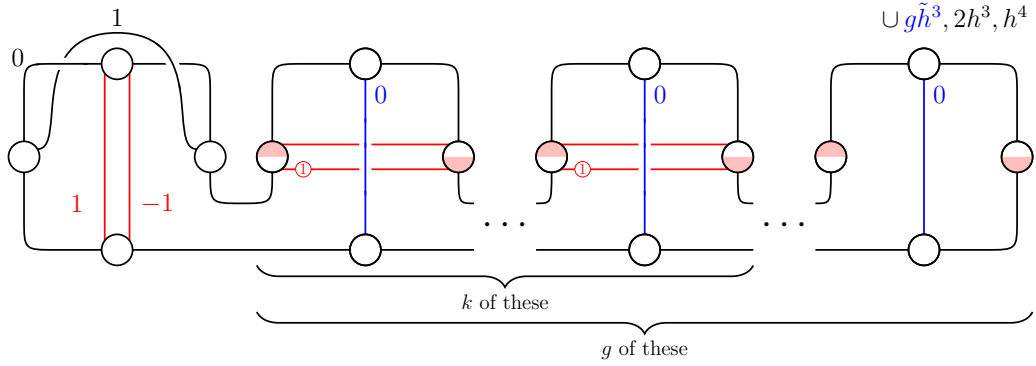

	\end{example}

	\subsection{Mapping class group factorizations for SBLFs}\label{sec:MCGcharacterization}
	
	As a converse to the monodromy reading of a given SBLF, certain algebraic data in mapping class groups of surfaces prescribe an SBLF via its Kirby diagram. Once again, we focus on the nonorientable case in our discussion below; the orientable case is treated in \cite{baykurPJM, Baykur-Hayano:BLFandMCG}.
	
	Let $(c; c_1,...,c_n)$ be two-sided simple closed curves on $S = N_k$. Let $\M(S; c)$ denote the stabilizer subgroup of $\M(S)$, fixing $c$ setwise, possibly reversing its orientation. Let $S'$ be the surface obtained from $S$ by first cutting the latter along an annular neighborhood of $c \subset S$, and then capping off the two boundary components we get by a pair of disks. So, $S' \cong N_{k-2}$ or $\Sigma_{g-1}$ for $k=2g$. There is a well-defined homomorphism $\Phi_c \colon \M(S; c) \longrightarrow \M(S')$ induced by this operation. This may be best called the \emph{cut-cap-forget} map, as it is a composition of these three well-studied homomorphisms, induced by first cutting $S$ along an open annulus neighborhood of $c_0$, then capping the two new boundary components with disks (with marked points), and finally forgetting the marked points \cite{Farb-Margalit, baykurPJM}.
	
	Now, assume that for $\mu := t_{c_n} t_{c_{n-1}} \cdots t_{c_1}$ in $\M(S)$, we have $\mu \in \M(S; c)$ and, furthermore, that $\Phi_c(\mu) = 1$ in $\M(S')$. Then we can reverse the procedure described above and build a handle decomposition, which can be capped off with $S' \times D^2$ to obtain a nonorientable SBLF $X \to S^2$. Unless the genus of $S$ is small ($k \leq 4$), this data determines the diffeomorphism type of $X$ completely.

    \enlargethispage{0.3in}
	\begin{remark} \label{rk:essential}
		If $c$ is not an \emph{essential} two-sided curve on a compact surface $S$, that is, if it bounds a disk or a M\"{o}bius band, the Dehn twist $t_c=1$ in $\M(S)$.
	\end{remark}

	\section{Simplified BLFs and trisections on nonorientable $4$--manifolds} \label{sec:SBLFs_and_trisections}
	
	Here we prove the existence of simplified broken Lefschetz fibrations and simplified trisections on nonorientable $4$--manifolds and present some related results.

	\subsection{Existence of SBLFs} We will argue the existence of SBLFs on arbitrary nonorientable $4$--manifolds by adapting Baykur and Saeki's proof in \cite{Baykur-Saeki:TAMS}. It is worth noting that most other proofs of the existence of (S)BLFs on orientable $4$--manifolds in the literature rely on results specific to the orientable case, such as the use of Eliashberg's classification of tight contact structures and Giroux's correspondence to match boundary open books of singular fibrations over disks; cf. \cite{Gay-Kirby:BLF}. Whereas in \cite{Baykur-Saeki:TAMS}, the construction of an SBLF on an arbitrary closed orientable $4$--manifold is given via topological modifications of generic maps. The local models for (indefinite) generic maps are unoriented, and orientability of the ambient $4$--manifold plays no essential role for carrying out these modifications. Therefore, the proof of \cite[Theorem 1.1]{Baykur-Saeki:TAMS}, which we outline below, applies to nonorientable $4$--manifolds as well. (In fact, the same arguments also generalize to maps from higher dimensional manifolds  \cite{saeki:simplifying-general}.)   We direct the interested reader to \cite{Baykur-Saeki:TAMS} for further details.
	
	In Section~\ref{sec:singularity_types}, we listed three types of singularities: Lefschetz, fold, and cusp. In what follows, a singular fibration $f \colon X \to S^2$ is called an \emph{indefinite fibration} if it is indefinite generic outside a finite collection of Lefschetz singularities. Let us denote by $Z_f$ the $1$--dimensional singular locus, which now includes cusps as well, and by $C_f$ the $0$--dimensional locus consisting of Lefschetz singularities. After a small perturbation of $f$, we may assume that $f|_{Z_f \cup C_f}$ is immersed, with at most double points along fold points. (It is smooth everywhere except at cusp points.) We decorate each indefinite fold circle with a transverse arrow indicating the direction of the fiberwise $2$--handle attachment—that is, toward the lower-genus side. A \emph{base diagram} of an indefinite fibration consists of the singular image $f(Z_f)$ drawn on $\mathbb{R}^2 = S^2 \setminus \{\infty\}$ with these decorations. (On a couple of occasions, where we more generally work with generic maps to a surface, we will also include definite fold circles in the base diagram, with an arrow pointing in the direction of the fiber-wise $0$--handle attachment.) Our proof of Theorem~\ref{thm:SBLFexistence}, mirroring that of Baykur and Saeki, will consist of a sequence of modifications of the base diagram that can be realized by homotopies of the singular fibrations \cite{Baykur-Saeki:TAMS}.

	\begin{proof}[Proof of Theorem \ref{thm:SBLFexistence}.]
		Let $X$ be a closed, connected, nonorientable $4$--manifold. Starting with any generic map $f\colon X\to S^2$, we are going to homotope it to an SBLF in four  steps. We continue to denote the  map we get after each modification by $f$.
		
		\noindent \underline{\textit{Step I}}: We first eliminate the \textit{definite fold circles} applying the homotopy sequence given in \cite[Theorem 2.2]{SAEKI2019}. An index argument allows one to perform isotopies and Reidemeister type moves in the base to separate the image of each definite fold circle from the rest of the base diagram. After that, one locally homotopes each  definite circle to an indefinite fold with cusps. Regardless of the orientability of $X$, there is a unique local model for the definite fold circle whose image encloses a disk with  no other singularities. 
		
		\noindent \underline{\textit{Step II}}: Let $\mathcal{D} \subset S^2$ be a disk containing the singular image $f(Z_f \cup C_f)$ of the indefinite generic map $f\colon X \to S^2$. We  apply the sequence of homotopies from \cite[Theorem~4.1]{Baykur-Saeki:TAMS} to modify $f$ over $\mathcal{D}$ (rel boundary) so that while tracing a straight path from the center of $\mathcal D$ to $\partial \mathcal D$, the Euler characteristic of a regular fiber along that path is non-decreasing. Such indefinite fibration is called \emph{directed}. When $X$ is nonorientable, the fiberwise $1$--handle attachments corresponding to (m)any of these indefinite folds may be incompatible with orientations; even if it is attached to an orientable fiber, the $1$--handle may be nonorientable, yielding a nonorientable surface on the higher genus side.

		\noindent \underline{\textit{Step III}}: Next, applying the sequence of homotopies of \cite[Theorem~4.3]{Baykur-Saeki:TAMS}, we transform the now directed indefinite fibration $f \colon X \to S^2$ into a directed indefinite fibration with \emph{embedded} singular image. Moreover, at this step, by repeatedly applying  the \emph{flip-and-slip} move from \cite{Baykur:IMRN}, at the expense of increasing the fiber genus, we can ensure that all fiberwise $1$--handle attachments are made starting from a connected fiber, and thus, all fibers are now connected \cite[Theorem~4.3]{Baykur-Saeki:TAMS}. 
		
		\noindent \underline{\textit{Step IV}}: Finally, by applying the sequence of homotopies given in \cite[Proposition~4.4]{Baykur-Saeki:TAMS}, which repeatedly uses cusp merges of Levine (after creating cusps on the fold circles if needed), we can connect all the indefinite fold circles into a single circle, while keeping the singular image embedded. The only requirement for implementing this algorithm successfully is that any two cusped fold circles we'd like to merge meet the same fiber component so that the cusp merges can be carried out as desired; the orientability of the fibers is irrelevant.
		
		We arrive at an indefinite map $f \colon X \to S^2$ with connected $Z_f$, embedded singular image, and connected fibers. By replacing any cusp points (contained in an indefinite circle) with Lefschetz singularities and then pushing them into the higher-genus side, we obtain a simplified broken Lefschetz fibration \cite{Baykur-Saeki:TAMS}. (In fact, these two homotopies could be carried out in earlier steps, as done so in \cite{Baykur-Saeki:TAMS}.) The auxiliary choice of whether a Lefschetz singularity agrees or disagrees with the orientation of $X$ becomes irrelevant when $X$ is nonorientable and the SBLF is de facto achiral. 
		
	\end{proof}
	
	\begin{example}[$S^1\twist S^3$]\label{eg:s1s3}
		Consider $X=S^1\twist S^3$ as the quotient space 
        \[S^3\times I/(x_1,x_2,x_3,x_4,0)\sim (-x_1,x_2,x_3,x_4,1),\] where $I=[0,1]$, and let $f\colon S^3\times I\to I\times I$ be the projection map $f(x_1,x_2,x_3,x_4,t)= (x_4,t)$. Since $f(x_1,x_2,x_3,x_4,0)=f(-x_1,x_2,x_3,x_4,1)$, this maps descends to a map $\tilde f\colon X \to I\times S^1$. Postcomposing $\tilde f$ with the embedding $I\times S^1 \hookrightarrow S^2$ gives a generic map $g\colon X\to S^2$. The regular fiber over $\textrm{Int}(I\times S^1)\subset S^2$ is $S^2$, whereas we get the empty set over $S^2\setminus (I\times S^1)$. That is, $g$ is a generic map with two definite fold circles along the boundary of $S^1\times I$, each corresponding to a round $0$-handle attachment. We homotope $g$ into an SBLF following the steps in Theorem \ref{thm:SBLFexistence}. (See \cite{Baykur-Saeki:TAMS} for the definitions and justifications of these base diagram moves.) Figure \ref{fig:s1s3_homotopy1} shows how the base diagram changes. First, we replace the two definite folds with indefinite ones, so that the map now surjects onto $S^2$ and the regular fibers are $S^2$ or $Kb$. Next we perform two homotopies which realize Reideister moves on the base, so that we may flip each of the indefinite fold circles twice. Then we merge them into one connected component of indefinite folds and cusps. By again applying Reidemister homotopies, the singular image becomes embedded again. Lastly, we replace the six cusp singularities with Lefschetz singularities and push them to the higher genus side. With this, we have a nonorientable genus six  SBLF with six Lefschetz singularities.\\
			
			\begin{figure}
				\centering
                
				\resizebox{5.5in}{!} {
					
					\begin{tikzpicture}
			\draw[thick,color=red] (0,0) circle[radius = .5cm];
			\draw[thick,color=red] (0,0) circle [radius = 2cm];
			\draw[->,thick] (0,.25) -- (0,.75);
			\draw[->,thick] (0,2.25) -- (0,1.75);
			
			\draw(3,0) node[scale=2] {$\Rightarrow$};
			\draw(3,1.5) node{definite to};
			\draw(3,1) node{indefinite};
			
			\begin{scope}[xshift = 6cm]
				\draw[thick] (0,0) circle[radius = .5cm];
				\draw[thick] (0,0) circle [radius = 2cm];
				\draw[->,thick] (0,.25) -- (0,.75);
				\draw[->,thick] (0,2.25) -- (0,1.75);
				
				\draw(3,0) node[scale=2] {$\Rightarrow$};
				\draw(3,1) node {R2};
			\end{scope}
			
			\begin{scope}[xshift = 12cm]
				\draw[thick] (0,0) circle[radius = .5cm];
				\draw[thick] (0,0) circle [radius = 2cm];
				\draw[->,thick] (0,.75) -- (0,.25);
				\draw[->,thick] (0,1.75) -- (0,2.25);
				
				\draw(3,0) node[scale=2] {$=$};
			\end{scope}
			
			\begin{scope}[xshift = 18cm]
				\draw[thick] (-1,0) circle[radius = .5cm];
				\draw[thick] (1,0) circle [radius = .5cm];
				\draw[->,thick] (-1,.75) -- (-1,.25);
				\draw[->,thick] (1,.75) -- (1,.25);
				
				\draw(0,-2.5) node[scale=2] {$\Downarrow$};
				\draw (.25,-2.5) node[anchor=west] {flip};
			\end{scope}
			
			\begin{scope}[xshift=17cm, yshift = -5cm]
				\draw[thick] (-1,.5).. controls +(0,0) and +(-.5,0) .. (0,-.5) .. controls +(.5,0) and +(0,0) .. (1,.5) -- (1,-.5) .. controls +(0,0) and +(.5,0) .. (0,.5) .. controls +(-.5,0) and +(0,0) .. (-1,-.5)-- (-1,.5);
				\draw[->,thick] (0,.75) -- (0,.25);
				
				\begin{scope}[xshift=3cm]
					\draw[thick] (-1,.5).. controls +(0,0) and +(-.5,0) .. (0,-.5) .. controls +(.5,0) and +(0,0) .. (1,.5) -- (1,-.5) .. controls +(0,0) and +(.5,0) .. (0,.5) .. controls +(-.5,0) and +(0,0) .. (-1,-.5)-- (-1,.5);
					\draw[->,thick] (0,.75) -- (0,.25);
				\end{scope}
				
				\draw (-2,0) node[scale=2] {$\Leftarrow$};
				\draw (-2, 1) node {cusp merge};
			\end{scope}
			
			\begin{scope}[xshift = 12cm, yshift = -5cm]
				\draw[thick] (-2,.5).. controls +(0,0) and +(-.5,0) .. (-1,-.5) .. controls +(.5,0) and +(-.5,0) .. (0,.5) .. controls +(.5,0) and +(-.5,0) .. (1,-.5) .. controls +(.5,0) and +(0,0) .. (2,.5) -- (2,-.5) .. controls +(0,0) and +(.5,0) .. (1,.5) .. controls +(-.5,0) and +(0,0) .. (.4,-.5) .. controls +(0,0) and +(.25,0) .. (0,0) .. controls +(-.25,0) and +(0,0) .. (-.4,-.5) .. controls +(0,0) and +(.5,0) .. (-1,.5);
				
				\draw[thick] (-1,.5) .. controls +(-.5, 0) and +(0,0) .. (-2,-.5)--(-2,.5);
				\draw[->,thick] (-1,.75) -- (-1,.25);
				
				\draw (-3,0) node[scale=2] {$\Leftarrow$};
				\draw (-3,1) node {R2};
			\end{scope}
			
			\begin{scope}[xshift=6cm, yshift = -6cm]
				\draw[thick] (-2,-.5) ..controls +(0,0) and +(-2,0) .. (0,2) .. controls +(2,0) and +(0,0) .. (2,-.5);
				\draw[thick] (-2,-.5) ..controls +(0,.25) and +(0,.25) .. (-1.2,-.5) ..controls +(0,.25) and +(0,.25) .. (-.4,-.5) ..controls +(0,.25) and +(0,.25) .. (.4,-.5) ..controls +(0,.25) and +(0,.25) .. (1.2,-.5) ..controls +(0,.25) and +(0,.25) .. (2,-.5);
				\draw[->,thick] (0,1.75) -- (0,2.25);
				
				\draw (-3,1) node[scale=2] {$\Leftarrow$};
				\draw (-3, 2) node {unsink};
			\end{scope}
			
			\begin{scope}[yshift=-5.5cm]
				\draw[thick] (0,0) circle[radius=2cm];
				\draw[->,thick] (0,1.75) -- (0,2.25);
				\draw(-.5,0) node{$\times$};
				\draw(.5,0) node{$\times$};
				\draw(-.5,.5) node{$\times$};
				\draw(.5,.5) node{$\times$};
				\draw(-.5,-.5) node{$\times$};
				\draw(.5,-.5) node{$\times$};
			\end{scope}
			
		\end{tikzpicture}
	}
    
				\caption{Homotopies to $g\colon S^1\twist S^3\to S^2$, where a point is removed from $S^2$ so that the base diagrams are drawn on a plane. The $R2$ indicates that a Reidemeister type II homotopy is being performed.}
				\label{fig:s1s3_homotopy1}
			\end{figure}
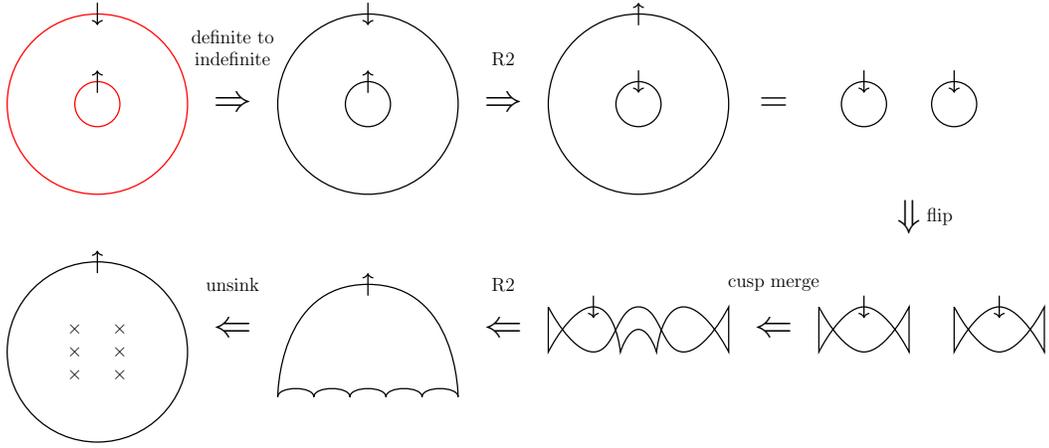

        \end{example}
	
	 \begin{example}[$(\#_k\mathbb{RP}^4)\#(\#_{g-k}S^1\twist S^3)$]
	 	Let's revisit the BLF admitted by $\#_g S^1\twist S^3$ in Figure \ref{fig:notall_eg}. This is a directed BLF with embedded round image, two Lefschetz singularities, and $N_{2g+2}$ as the highest genus fiber. Following the procedure of \cite[Propostion 4.4]{Baykur-Saeki:TAMS}, we may homotope this BLF so that its $g$ fold circles are flipped and then merged into one connected component of folds with $2g+2$ cusps. By replacing the cusps with Lefschetz singularities, we end up with an SBLF with higher and lower genus side regular fibers $N_{2g+4}$ and $N_{2g+2}$ and $2g+4$ Lefschetz singularities. Figure \ref{fig:notall_htpy} shows the sequence of base diagram moves to realize this SBLF. Note that if we first perform $k$ nonorientable blow-ups for $k\leq g$, and go through the same steps now with $k$ additional Lefschetz singularities, we obtain an SBLF for $(\#_k\mathbb{RP}^4)\#(\#_{g-k}S^1\twist S^3)$ instead. The resulting SBLF has the same regular fibers as the one for $\#_g S^1\twist S^3$, but contains $k$ additional Lefschetz singularities.
	 	
	 	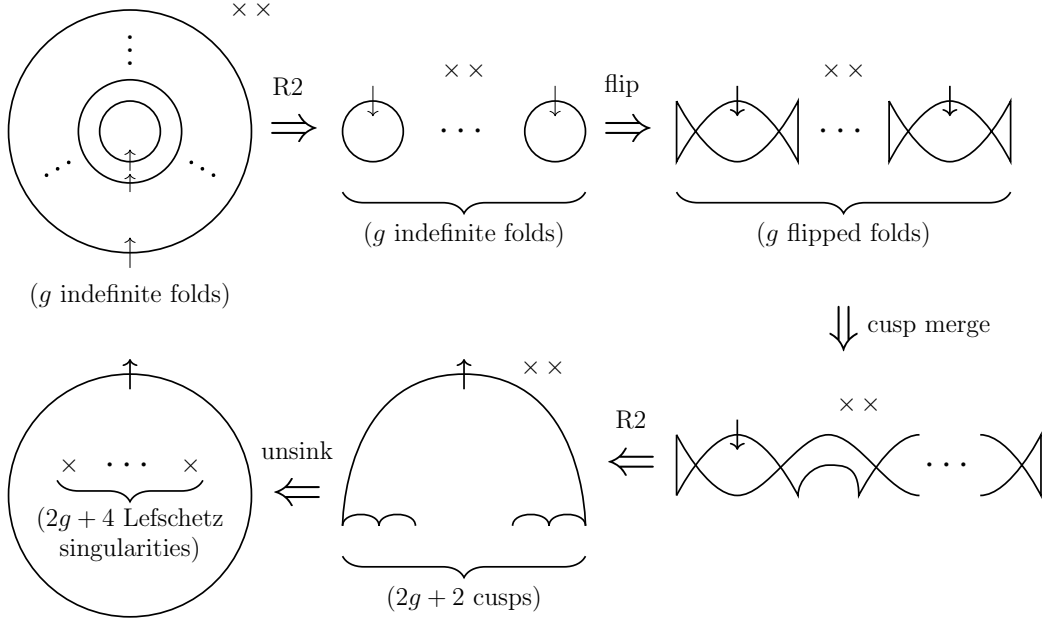
\begin{figure}
	 		\centering

	 		\resizebox{5.5in}{!}{
	 			
	 			\begin{tikzpicture}
	 				\draw[thick] (0,0) circle[radius = .5cm];
	 				\draw[thick] (0,0) circle[radius=.85cm];
	 				\draw[thick] (0,0) circle[radius=2cm];
	 				\draw (0,1.5) node[scale=1.5] {$\vdots$};
	 				\draw (1.329,-.696) node[scale=1.5, rotate=-120] {$\vdots$};
	 				\draw (-1.329,-.696) node[scale=1.5, rotate=120] {$\vdots$};
	 				
	 				\draw[->] (0,-2.25) -- (0,-1.75);
	 				\draw[->] (0, -1) -- (0,-.7);
	 				\draw[->] (0,-.65) -- (0,-.3);
	 				
	 				\draw (0,-2.7) node {($g$ indefinite folds)};
	 				\draw (2,2) node {$\times\,\times$};
	 				
	 				\draw (2.7,0) node [scale=2] {$\Rightarrow$};
	 				\draw (2.6,.75) node {R2};
	 				
	 				\begin{scope}[xshift=5.5cm]
	 					\draw[thick] (-1.5,0) circle[radius=.5cm];
	 					\draw[thick](1.5,0) circle[radius=.5cm];
	 					\draw (0,0) node[scale=1.5] {$\cdots$};
	 					\draw[->] (-1.5,.75) -- (-1.5,.25);
	 					\draw[->] (1.5,.75) -- (1.5,.25);
	 					
	 					\draw(0,1) node {$\times\, \times$};
	 					
	 					\draw [decorate,
	 					decoration = {brace,mirror, amplitude=10pt}, thick] (-2,-1) --  (2,-1);
	 					\draw (0,-1.7) node {($g$ indefinite folds)};
	 					
	 					\draw (2.7,0) node [scale=2]{$\Rightarrow$};
	 					\draw (2.6, .75) node {flip};
	 				\end{scope}
	 				
	 				\begin{scope}[xshift=10cm]
	 					\draw[thick] (-1,.5).. controls +(0,0) and +(-.5,0) .. (0,-.5) .. controls +(.5,0) and +(0,0) .. (1,.5) -- (1,-.5) .. controls +(0,0) and +(.5,0) .. (0,.5) .. controls +(-.5,0) and +(0,0) .. (-1,-.5)-- (-1,.5);
	 					\draw[->,thick] (0,.75) -- (0,.25);
	 					
	 					\draw (1.75,0) node [scale=1.5] {$\cdots$};
	 					
	 					\begin{scope}[xshift=3.5cm]
	 						\draw[thick] (-1,.5).. controls +(0,0) and +(-.5,0) .. (0,-.5) .. controls +(.5,0) and +(0,0) .. (1,.5) -- (1,-.5) .. controls +(0,0) and +(.5,0) .. (0,.5) .. controls +(-.5,0) and +(0,0) .. (-1,-.5)-- (-1,.5);
	 						\draw[->,thick] (0,.75) -- (0,.25);
	 					\end{scope}
	 					
	 					\draw (1.75,1) node {$\times\, \times$};
	 					
	 					\draw [decorate,
	 					decoration = {brace,mirror, amplitude=10pt}, thick] (-1,-1) --  (4.5,-1);
	 					\draw(1.75,-1.7) node {($g$ flipped folds)};
	 					
	 					\draw(1.75, -3.25) node[scale=2]{$\Downarrow$};
	 					\draw(2,-3.25) node [anchor =west]{cusp merge};
	 					
	 				\end{scope}
	 				
	 				\begin{scope}[xshift=12cm,yshift=-5.5cm]
	 					
	 					\draw[thick] (-3,.5).. controls +(0,0) and +(-.5,0) .. (-2,-.5) .. controls +(.5,0) and +(-.5,0) .. (-.5,.5) .. controls +(.5,0) and +(-.5,0) .. (1,-.5);
	 					\draw[thick] (-3,-.5) .. controls +(0,0) and +(-.5,0) .. (-2,.5) ..controls +(.5,0) and +(0,0) .. (-1,-.5)  .. controls +(0,.25) and +(-.25,0) .. (-.5,0) .. controls +(.25,0) and +(0,.5) .. (0,-.5).. controls +(0,0) and +(-.5,0) .. (1,.5);
	 					\draw[thick](-3,.5) -- (-3,-.5);
	 					
	 					\draw[->,thick] (-2,.75) -- (-2,.25);
	 					
	 					\draw (1.5,0) node[scale=1.5] {$\cdots$};
	 					
	 					\draw[thick] (2,.5) .. controls +(.5,0) and +(0,0).. (3,-.5) -- (3,.5) .. controls +(0,0) and +(.5,0) .. (2,-.5);
	 					
	 					\draw(0,1) node {$\times\, \times$};
	 					
	 					\draw(-3.75,0) node[scale=2]{$\Leftarrow$};
	 					\draw(-3.75,.75) node {R2};
	 					
	 				\end{scope}
	 				
	 				\begin{scope}[xshift=5.5cm,yshift=-6cm]
	 					
	 					\draw[thick] (-2,-.5) ..controls +(0,0) and +(-2,0) .. (0,2) .. controls +(2,0) and +(0,0) .. (2,-.5);
	 					\draw[thick] (-2,-.5) ..controls +(0,.25) and +(0,.25) .. (-1.4,-.5) ..controls +(0,.25) and +(0,.25) .. (-.8,-.5);
	 					
	 					\draw[thick] (.8,-.5) ..controls +(0,.25) and +(0,.25) .. (1.4,-.5) ..controls +(0,.25) and +(0,.25) .. (2,-.5);
	 					\draw[->,thick] (0,1.75) -- (0,2.25);
	 					
	 					\draw(1.3,2.1) node {$\times\,\times$};
	 					
	 					\draw [decorate,
	 					decoration = {brace,mirror, amplitude=10pt}, thick] (-2,-1) --  (2,-1);
	 					\draw(0,-1.7) node {($2g+2$ cusps)};
	 					
	 					\draw(-2.75,0) node[scale=2]{$\Leftarrow$};
	 					\draw(-2.75,1.25) node {unsink,};
                        \draw(-2.75,.75) node {push};
	 					
	 				\end{scope}
	 				
	 				\begin{scope}[yshift=-6cm]
	 					\draw[thick] (0,0) circle[radius = 2cm];
	 					
	 					\draw (-1,.5) node {$\times$};
	 					\draw (0,.5) node[scale=1.5] {$\cdots$};
	 					\draw (1,.5) node {$\times$};
	 					
	 					\draw [decorate,
	 					decoration = {brace,mirror, amplitude=10pt}, thick] (-1.2,.25) --  (1.2,.25);
	 					\draw (0,-.4) node {($2g+4$ Lefschetz};
	 					\draw (0,-.9) node {singularities)};
	 					
	 					\draw [->,thick] (0,1.75) -- (0,2.25);
	 				\end{scope}

	 			\end{tikzpicture}
	 		}

	 		\caption{Homotoping the BLF on $\#_g S^1\twist S^3$ into a SBLF. In the starting base diagram on the top left, the inner-most regular fibers are tori and the outer-most are copies of $N_{2g+2}$.}
	 		\label{fig:notall_htpy}
	 	\end{figure}
	 	\end{example}

	The following proposition shows that nearly all of the $4$--manifolds in the previous example do not admit Lefschetz fibrations or pencils. This conveys that 
	the abundance of nonorientable simplified broken Lefschetz fibrations sets them apart from honest Lefschetz fibrations.

	\begin{prop} \label{prop:notall}
		For $g \geq 2$ and $0\leq k <g$, the $4$--manifolds $(\#_k\mathbb{RP}^4)\#(\#_{g-k}S^1\twist S^3)$ do not admit a Lefschetz fibration or a pencil.
	\end{prop}
	
	\begin{proof}
		For $k\geq 1$, the orientation double cover of $\#_k \mathbb{RP}^4$ is $\#_{k-1}S^1\times S^3$. Using the fact that $(\#_k\mathbb{RP}^4)\#(\#_{g-k}S^1\twist S^3) \cong (\#_k\mathbb{RP}^4)\#(\#_{g-k}S^1\times S^3)$, $(\#_k\mathbb{RP}^4)\#(\#_{g-k}S^1\twist S^3)$ is doubly covered by $\#_{k-1}S^1\times S^3$ and two copies of $\#_{g-k}S^1\times S^3$. Hence the orientation double cover of $(\#_k\mathbb{RP}^4)\#(\#_{g-k}S^1\twist S^3)$ is $\#_{2g-k-1} \, S^1 \times S^3$. Since $\#_g \, S^1\twist S^3$ has $\#_{2g-1} S^1\times S^3$ as its orientation double cover, this extends to the $k=0$ case. The constraints $g\geq 2$ and $0\leq k<g$ imply that the orientation double cover is $\#_m S^1\times S^3$ for some $m\geq 2$. Post-composing the covering map with any Lefschetz fibration (resp. pencil) on $(\#_k\mathbb{RP}^4)\#(\#_{g-k}S^1\twist S^3)$ would yield an achiral Lefschetz fibration (resp. pencil) on the double cover. However, the characteristic class obstruction of \cite[Theorem~8.4.13]{Gompf-Stipsicz} implies that $\#_m \, S^1 \x S^3$ does not admit a Lefschetz fibration or pencil when $m\geq 2$ \cite[Corollary 8.4.14]{Gompf-Stipsicz}.
        %However, the characteristic class obstruction of \cite[Theorem~8.4.13]{Gompf-Stipsicz} implies that no $\#_m \, S^1 \x S^3$ admits a Lefschetz fibration or a pencil when $m \geq 2$ \cite[Corollary 8.4.14]{Gompf-Stipsicz}.
	\end{proof}
	
	\begin{remark}\label{rk:relminimal}
		If a broken Lefschetz fibration $X \to \Sigma$ contains an isolated Lefschetz singularity whose vanishing cycle $c$ bounds a disk on the regular fiber, then the corresponding singular fiber contains an embedded $2$--sphere $S$ of self-intersection $\pm 1$ that can be blown down without affecting the rest of the fibration, yielding a new broken Lefschetz fibration on the blow-down, $X' \to \Sigma$ \cite{Gompf-Stipsicz}. In a similar way, if $c$ bounds a M\"{o}bius band on the fiber, then the singular fiber contains an embedded $\mathbb{RP}^2$ with a regular neighborhood $\mathbb{RP}^2 \twist D^2$, which can be blown down without altering the rest of the fibration, resulting in a broken Lefschetz fibration on the nonorientable blow-down, $X'' \to \Sigma$ \cite{Miller-Ozbagci}. The same goes for pencils. Note that, by Remark~\ref{rk:essential}, in each case the local monodromy $t_c$ is trivial. We call a broken Lefschetz fibration \emph{relatively minimal} if it does not have any isolated (Lefschetz-type) singular fibers containing a $2$--sphere or a real projective plane. Equivalently, all of its Lefschetz vanishing cycles are essential curves. It is clear that every SBLF is a blow-up (orientable or nonorientable) of a relatively minimal SBLF.
	\end{remark}
	
	We end with a further improvement of the existence result of Theorem~\ref{thm:SBLFexistence}:

	\begin{corollary} \label{prop:good_sblf}
		Every closed, connected, nonorientable $4$--manifold admits a relatively minimal SBLF of arbitrarily large, even nonorientable genus.    
	\end{corollary}
	
	\begin{proof}

		Given a nonorientable genus $k$ SBLF, the flip-and-slip move of \cite{Baykur:IMRN} yields a new SBLF with four more Lefschez singularities and nonorientable genus $k+2$, just like how the higher genus side regular fiber changes from $\Sigma_g$ to $\Sigma_{g+1}$ in the orientable case. This is how we get to SBLFs with arbitrarily large genera. 
		
		The claim about the parity of the nonorientable genus follows from the proof of Theorem~\ref{thm:SBLFexistence}. We can start off the construction with a generic map $X \to S^2$ that approximates the map obtained from the Pontrjagin–Thom construction applied to any trivially embedded surface $\Sigma_g \subset D^4 \subset X$. When approximating this continuous map, which is defined by projecting $\nu \Sigma_g \cong \Sigma_g \times D^2 \to D^2 \cong \mathcal{D} \subset S^2$ and collapsing $X \setminus \textrm{Int}(\nu \Sigma_g)$ to a point, we can assume that the resulting generic map agrees with the initial map over $\mathcal{D}$, where it was already smooth and regular. It is then sufficient to observe that in all subsequent steps, the map over $\mathcal{D}$ is only modified by the trespassing fold circles, which change the fibers over $\mathcal{D}$ either by adding or removing $S^2$ components (in the case of definite folds) or by adding or removing tubes (in the case of indefinite folds). The fibers of the resulting SBLF will have Euler characteristic differing from that of $\Sigma_g$ by an even number, thereby justifying our claim.
		
		Now assuming $f\colon X\to S^2$ is an SBLF with even genus fibers, let $\mathcal{C}$ be the set of Lefschetz vanishing cycles for $f$ bounding disks or M\"{o}bius bands. If $\mathcal{C}=\emptyset$, we're done. If not, to proceed, we first ensure that $Z_f\neq \emptyset$ by performing a birth homotopy followed by two unsinks if needed. Let $c$ be the indefinite fold vanishing cycle, and let $F_+$ be the (higher genus) reference fiber for $f$. Let $\gamma\in \mathcal{C}$, and let $h_c$ and $h_{\gamma}$ be the $2$--handles corresponding to the vanishing cycles $c$ and $\gamma$. Suppose we slide $h_{\gamma}$ over $h_c$. If $\gamma$ and $c$ have zero geometric intersection, then the new attaching circle for $h_{\gamma}$ still traces a simple closed curve on $F_+$ and has framing $\pm 1$. We know $[\gamma]=0$ and $[c]\neq 0$ in $H_1(F_+;\Z)$, and the new vanishing cycle is represented by $[\gamma \pm c]$. It follows that the new vanishing cycle is homologically essential, so is non-separating. If $\gamma$ is geometrically disjoint from $c$ for all $\gamma\in \mathcal{C}$, we may use this trick for each $\gamma \in \mathcal C$ to obtain a relatively minimal SBLF, completing the proof. Before reaching this conclusion, we need to consider the possibility that for some $\gamma_0\in \mathcal C$, $\gamma_0$ and $c$ have non-zero geometric intersection, in which case the handleslide described above would ruin the SBLF structure, since the new attaching circle for $h_{\gamma_0}$ would not trace a simple closed curve on $F_+$. In this case, we perform a flip-and-slip homotopy, during which the higher (nonorientable) genus fiber increases by two, and the fold circle gets four new cusps. (See \cite{Hayano_R2} for a more detailed description of this homotopy.) Unsinking each of these cusps makes the fold circle non-cusped again and introduces four new Lefschetz singularities. Let $F'_+$ be the new higher genus regular fiber and let $\mathcal{C}'$ be the image of $\mathcal{C}$ in $F'_+$, i.e. the image of each Lefschetz vanishing cycle bounding a disk or M\"{o}bius band in $F'_+$. After performing the unsink moves \cite{Lekili, Baykur-Saeki:TAMS}, we may choose the reference fiber $F'_+$ such that the new indefinite fold circle $c'$ is disjoint from the curves in $\mathcal{C}'$. For this reason, for each $\gamma \in \mathcal{C}'$, we may slide $h_{\gamma}$ over $h_{c'}$, thereby replacing them all with homologically essential Lefschetz vanishing cycles. After this, the only Lefschetz vanishing cycles that could possibly bound disks or M\"{o}bius bands are the four introduced by unsinking the cusps after the flip-and-slip. We claim that these four vanishing cycles are nonseparating. To see this, first observe that they're each homologous to $[a\pm b]$, where $a$ and $b$ are a pair of geometrically dual fold vanishing cycles created during the flip-and-slip. They each have a geometric dual given by either $a$ or $b$ and so are homologically essential, hence non-separating. This confirms that there are no longer vanishing cycles bounding disks or M\"{o}bius bands, completing the proof.
		\end{proof}	
	
	Lastly, we note the following combinatorial description of nonorientable $4$--manifolds via mapping class group factorizations, building on our discussion in Subsection~\ref{sec:MCGcharacterization} and (the proof of) Corollary~\ref{prop:good_sblf}:

	\begin{corollary} \label{cor:combinatorial}
		Every closed, connected, nonorientable $4$--manifold is prescribed (up to diffeomorphism)
        by an ordered tuple of two-sided essential curves  $(c; c_1, \ldots, c_n)$ on $N_{2k}$ such that 
		\[t_{c_1} \cdots t_{c_n} \in {Ker}(\Phi_{c})\] 
		of the cut-cap-forget homomorphism  $\Phi_{c}\colon \M(N_{2k}; c)\to \M(N_{2k-2})$, for some $k\in \Z_{\geq 1}$.
	\end{corollary}

	\begin{remark}
    	With some care, the correspondence between SBLFs and the monodromy data can be upgraded to a bijection between isomorphism classes of nonorientable genus--$2k$ SBLFs on a given $X$ and tuples $(c; c_1, \ldots, c_n)$ on $N_{2k}$; see \cite{Baykur-Hayano:BLFandMCG} for the account of this correspondence in the orientable case.
	\end{remark}

	\subsection{Simplified trisections} 
    
	\begin{figure}
		\centering
		\includegraphics[width=0.5\linewidth]{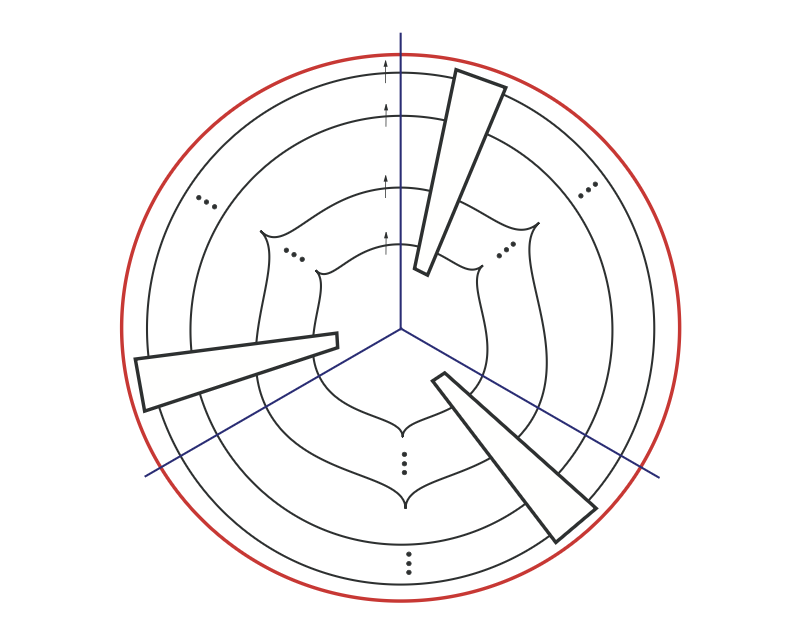}
		\caption{A base for a map $f\colon X\to \R^2$ corresponding to a trisection on $X$. The red outermost circle is a definite fold circle that bounds the image of $f$, whereas the others are indefinite fold arcs and circles, all directed outwards. Note that in each sector, an indefinite fold circle carries at most one cusp, and the total number of cusps is the same across all three sectors.}
		\label{fig:morse2_fn}	
	\end{figure}
	
	The correspondence between trisection decompositions of closed, connected $4$--manifolds and \emph{trisection maps}, which are special generic maps whose base diagrams are as in Figure~\ref{fig:morse2_fn}, plays a central role in the original proof of the existence of trisections on \emph{orientable} $4$--manifolds by Gay and Kirby \cite{Gay-Kirby:trisection}. In the case of simplified trisections, the \emph{Cerf boxes} for the trisection map $f \colon X \to \mathbb{R}^2$, namely, the regions of the base diagram over which the map is a directed indefinite fibration without cusps, are trivial. This is equivalent to the boxes in Figure \ref{fig:morse2_fn} consisting of parallel lines.
	
	There are just a few cosmetic differences when $X$ is nonorientable. Let's say $X$ admits a simplified trisection map with $g$ indefinite fold circles, $k$ of which are thrice-cusped.  The trisection map $f$ determines a $(g,k)$ trisection of $X$ as follows: We take the trisection decomposition of $X=X_1 \cup X_2 \cup X_3$ with sectors $X_i:=f^{-1}(A_i)$, where each $A_i$ is the slice of the base of the trisection map $f \colon X \to \mathbb{R}^2$ shown in Figure~\ref{fig:morse2_fn}. Verbatim to the oriented case \cite{Gay-Kirby:trisection}, each $X_i$ is a $4$--dimensional genus--$k$ handlebody, with pairwise intersections along a $3$--dimensional genus--${g}$ handlebody, and a triple intersection along the same surface $S:=f^{-1}(0)$. There are two $3$--dimensional handlebodies up to diffeomorphism: $\natural_{g} S^1 \times D^2$ and $\natural_{g} S^1 \twist D^2$. Similarly, there are two $4$--dimensional handlebodies: $\natural_{k} S^1 \times D^3$ and $\natural_{k} S^1 \twist D^3$. Since each $3$--dimensional handlebody $X_{i,j} := X_i \cap X_j$ shares the same boundary $S$, they all have the same type of handlebody $H$, determined by the orientability of $S$. And since each $\partial X_i \cong H \cup_S H$, this in turn determines the handlebody type of $X_i$.
	
	Now note that we can induce a handle decomposition of $X$ with all $0$-- and $1$--handles in $X_1$, and all $3$-- and $4$--handles in $X_3$, verbatim to the orientable case \cite[Section~3]{Gay-Kirby:trisection}. Since $X$ is nonorientable, we deduce that $X_i \cong \natural_k S^1 \twist D^3$. In turn, we see that $H \cong \natural_{g} S^1 \twist D^2$ and $S = N_{2g}$.
	
	Alternatively, one can see that $S = N_{2g}$ by examining the base diagram of $f$ and recalling that $X$ is nonorientable, and then deduce the diffeomorphism types of the handlebodies from it, which is fairly straightforward in the simplified case.

	Theorem \ref{thm:STexistence} follows immediately from the existence of a simplified trisection map on $X$. This is because the construction given in \cite[Proof of Theorem 7.1]{Baykur-Saeki:TAMS} to obtain a simplified trisection from an SBLF on an orientable $4$--manifold works just the same in the nonorientable case. We summarize this process below.

	\begin{proof}[Proof of Theorem \ref{thm:STexistence}.]
		Given a simplified broken Lefschetz fibration $f\colon X\to S^2$, partition the base as $S^2=D^2_+\cup A \cup D^2_-$ as in Section \ref{sec:SBLF_topology}. Let $X_+$, $X_0$, and $X_-$ be the respective preimages of these regions under $f$. Then let $g \colon X\to \mathbb{R}^2$ be defined as follows:\begin{itemize}
			\item $g|_{X_+\cup X_0}:=f|_{X_+\cup X_0}$. Identifying $D^2_ +\cup A$ with the unit disk in $\mathbb{R}^2$, we view $g|_{X_+\cup X_0}$ as mapping into $\mathbb{R}^2$. 
			\item Let $F_-$ be a genus $g$ lower-genus fiber of $f$ and let $\nu F_-$ be a neighborhood of $F_-$ contained in $\text{Int}(D_-)$. Set $g|_{\nu F_-} = f|_{\nu F_-}$, where we identify the image of $f|_{\nu F_-}$ with $\mathbb{D}\subset \R^2$ as we did in the previous step.
			\item Let $S= X_+\cup X_0\cup \nu F_-$. We have now defined the function $g$ on $S$ such that $g^{-1}(\partial \mathbb{D})$ is two copies of $F_-\times S^1$. Let $h\colon F_-\to [1,2]$ be the standard Morse function on $F_-$ with one index--$0$ critical point over $1$, $g$ index--$1$ critical points, and one index--$2$ critical point over $2$. Then define $\phi\colon S^1\times [-1,1]\times F_-\to S^1\times [1,3]$ as \[ \phi(u,t, x) = (u, h(x)\cos(\pi t/2)+1).\] By identifying $S^1 \times [-1,1]\times F_-$ with $X - S$ and embedding $S^1\times [1,3]$ into $\mathbb{R}^2$, $\phi$ defines a map $X-S \to \mathbb{R}^2$. Let $g|_{X-S}:=\phi$.
		\end{itemize} In total, $g$ defines a generic map $X\to \mathbb{R}^2$ with concentric fold circles around the origin, and whose Lefschetz singularities lie inside the fold circle of smallest radius. The outer-most fold circle is definite, and there is one inward-directed indefinite fold circle. The rest of the fold circles are indefinite and outward-directed. As in the proof of \cite[Theorem 7.1]{Baykur-Saeki:TAMS}, we may move the inward-directed fold into two thrice-cusped outward-directed folds. Then we perturb Lefschetz singularities into outward-directed fold circles with three cusps using the wrinkle homotopy \cite{Lekili}. The resulting base diagram after these homotopies is a simplified trisection diagram \cite{Baykur-Saeki:TAMS}.
	\end{proof}

	\begin{remark} \label{rk:trisection_to} 
    Suppose the SBLF $f$ has $a$ Lefschetz singularities. If $Z_f\neq \emptyset$ and $\beta_1(F_-;\Z_2)=b$, then the algorithm above produces an $(a+b+3, b+1)$ simplified trisection. If $Z_f=\emptyset$ and $\beta_1(F_-;\Z_2)=\beta_1(F_+;\Z_2)=b$, then the algorithm produces an $(a+b+2,b)$ simplified trisection. Conversely, the proof of \cite[Proposition~7.6]{Baykur-Saeki:TAMS} extends mutadis mutandis to derive an SBLF on a nonorientable $4$--manifold $X$ from a given simplified trisection. From a nonorientable $(g,k)$ simplified trisection on $X$, one may obtain an SBLF on $X$ with a higher-genus fiber of $N_{2g+4}$, a lower-genus fiber of $N_{2g+2}$, and $2g+3k+4$ Lefschetz singularities.
	\end{remark}
	
	Given a closed, connected $4$--manifold $X$, one can ask what the smallest possible genus is of a trisection on $X$. We conclude this section by providing a few examples where our methods realize the minimal genus. We'll make use of the following lower bound for a nonorientable trisection genus, provided in \cite{Spreer-Tillman}.
	
	\begin{lemma}\label{lem:genus_bound}
		If $X$ is nonorientable and admits a $(g,k)$ trisection, then \[g\geq \beta_1(X;\Z_2)+\beta_2(X;\Z_2).\]
	\end{lemma}
	
	\begin{example}[$S^2\times N_g$]
		Consider the trivial SBLF $p\colon N_g\to S^2$, where $N_g$ is a closed nonorientable genus $g$ surface. Let $A$ be an annular neighborhood of the equator in $S^2$. Following the procedure in Theorem \ref{thm:STexistence}, we first define a map on $p\inv(S^2 \setminus A)$, which in this case is a simple projection onto the unit disk. So, to begin, our regular fibers over the unit disk are two disjoint copies of $N_g$. As we extend this map over $p\inv(A)$, we introduce one inward--directed indefinite fold, followed by $g$ outward--directed indefinite folds, and then one inward--directed definite fold. We flip-and-slip the inward-directed indefinite fold to be outward-directed, then unsink one of the cusps so that it becomes a Lefschetz singularity. Lastly we wrinkle the Lefschetz singularity to obtain a $(g+2,g)$ simplified trisection. Since $g+2 = \beta_1(N_g\times S^2;\Z_2)+\beta_2(N_g\times S^2;\Z_2)$, we deduce from Lemma~\ref{lem:genus_bound} that this trisection is of minimal genus.
	\end{example}
	
	\begin{example}[$S^1\twist S^3$ and $\mathbb{RP}^4$]
		Consider the map $\tilde f:S^1\twist S^3 \to A^2$ from Example \ref{eg:s1s3}. By embedding the annulus into $\R^2$ and then applying a definite-to-indefinite homotopy to the inner fold circle, we obtain a trisection map corresponding to a simplified $(1,1)$ trisection. As discussed in \cite{Miller-Naylor}, this is the only nonorientable genus one trisection. The fact that it's of minimal genus follows immediately from the fact that any lower--genus trisection would be $S^4$. By applying a nonorientable blow-up, we obtain a map $\mathbb{RP}^4\to \R^2$ with the same base diagram except for having one new Lefschetz singularity.  We may wrinkle the Lefschetz singularity into a thrice cusped fold circle to obtain a $(2,1)$ simplified trisection of $\mathbb{RP}^4$. Lemma \ref{lem:genus_bound} shows that this trisection is also minimal genus.\\
	\end{example}

	\begin{example}($S^1 \x Y^3$ and spin of $Y^3$) 
		Let $Y$ be a closed, connected, nonorientable $3$--manifold, and set $k=\beta_1(Y;\Z_2)$, which by duality also implies $k=\beta_2(Y;\Z_2)$. We look at two general constructions, following \cite{Baykur-Saeki:PNAS}. 
        
        Let $f \colon Y\to I$ be a Morse function on $Y$ with one index $0$ critical value at $0$, $m$ index $1$ critical points with images in $(0,0.5)$, $m$ index two critical points in $(0.5, 1)$, and one index $3$ critical point at $1$, where $m\geq k$. Then the map $\text{id}_{S^1}\times F\colon S^1\times Y \to S^1\times I\hookrightarrow \R^2$ has one outward-directed definite fold circle, followed by $m$ inward-directed indefinite fold circles, then $m$ outward-directed indefinite fold circles, followed by one definite fold circle. We may replace the inner-most definite fold circle with an outward-directed indefinite fold circle using the definite-to-indefinite flip-and-slip. After that, the map may be homotoped to give a $(3m+1, m+1)$--simplified trisection for $S^1\times Y$. We have $H_1(S^1\times Y;\Z_2)\cong \Z_2^{k+1}$ and $H_2(S^1\times Y;\Z_2)\cong \Z_2^{2k}$. So by Lemma \ref{lem:genus_bound}, this is a minimal genus trisection of $S^1\times Y$ when $m=k$. 
		
		We may also obtain a simplified trisection for the \emph{spin} of $Y$, $S_Y$ by modifying the above procedure slightly. Recall that $S_Y$ is obtained from $S^1\times Y$ by replacing $S^1\times D^3$ with $S^2\times D^2$ for some $D^3\subset Y$. From the $(3m+1,m+1)$--trisection for $S^1\times Y$, a $(3m, m)$--trisection is obtained for the $S_Y$ by surgering out the preimage of a (non-cusped) indefinite fold circle. Since $\beta_1(S_Y;\Z_2)=k$ and $\beta_2(S_Y;\Z_2)=2k$, this is also a minimal genus trisection when $m=k$.

        In either case, the above process results in a minimal genus trisection whenever $Y$ admits a $\Z_2$--perfect Morse function. This criteria is met for instance by a connected sum of any number of copies of  
        $S^1\twist S^2$ and $S^1 \x N_k$. 
	\end{example}

\section{Classification of low genera fibrations}\label{sec:classification}
	
	Here we give a detailed discussion of all SBLFs whose higher--genus fiber is a Klein bottle. This starts with SBLFs with no singularities (i.e. fiber bundles), then progresses to include indefinite fold singularities, then ends by adding in Lefschetz singularities. 
	
	\begin{remark}\label{rmk:low_genus_2}
		We skip over the cases when the higher genus side regular fiber $F_+\cong S^2$ or $\RP$ because such SBLFs would not have an indefinite fold and there are no \emph{non-minimal} Lefschetz fibrations with $S^2$ or $\RP$ fibers. Thus, the smallest nonorientable genus of SBLFS worth studying is two, that is when $F_+ = Kb$.
	\end{remark}

	\subsection{Klein bottle bundles over the 2-sphere} \label{sec:KB_bundles} 
    The homotopy type of $Kb$ bundles over $S^2$ is already established in \cite[Chapter 5]{Hillman}. These bundles can be identified with explicit elements of $\pi_1(\Diff(Kb))$, as we review below. We will identify their orienation double--covers, which using Proposition \ref{prop:dble_cover_alg}, confirms their Kirby diagrams. This way, we will present a smooth classification of their total spaces via explicit handlebody descriptions.
	
	Let $X$ be a fiber bundle over $S^2$ with $Kb$ fibers. Then $X$ is obtained from gluing two copies of $Kb\times D^2$ via $\phi\colon S^1\times Kb\to S^1\times Kb$ such that $pr_1\circ \phi = pr_1$. Such a map $\phi$ can be written as $\phi(t,p) = (t,\phi_2(t,p))$ with $\phi_2\colon S^1\times Kb\to Kb$. Thus $\phi$ corresponds to a loop in $\Diff(Kb)$. Since homotopic loops give diffeomorphic fiber bundles, each element of $\pi_1(\Diff(Kb))$ gives rise to a $Kb$ bundle over $S^2$. Earle and Eells have shown that this group is isomorphic to $\Z$ \cite{Earle_Eells}. As we are interested in a smooth classification of total spaces, we will examine when any two $Kb$--bundles give the same total space. 
    
    Regard $Kb$ as a quotient space of its universal cover $\R^2$ as follows: let $\sigma\colon \R^2\to \R^2$ and $\tau\colon \R^2\to \R^2$ be defined \begin{align*}
		\sigma(x,y) &= (x+1,y)\\
		\tau(x,y) &= (-x,y+1).
	\end{align*} Then $Kb=\R^2/\lan \sigma,\tau\ran$. Let $\tilde \phi_t\colon \R^2\to \R^2$ be defined \[
	\tilde \phi_t(x,y) = (x,y+2t), \,\, t\in [0,1].
	\] One can check that for all $t$, $\tilde \phi_t$ descends to a diffeomorphism $\phi_t\colon Kb\to Kb$, and that $\phi_0=\phi_1$. Hence we get a loop of diffeomorphisms of $Kb$, which we denote as $\phi\colon S^1\times Kb\to Kb$ and regard as an element of $\pi_1(\Diff(Kb))$.
    
    \begin{lemma}\label{lem:diff_KB_gen}
    $\pi_1(\Diff(Kb))$ is generated by $\phi$.
    \end{lemma}
	
    \begin{proof}
	Let $a,b$ be the generators of $\pi_1(Kb)$, such that any $Kb$ bundle over $S^2$ is obtained by adding one relation to the group $\langle a,b|aba\inv b\rangle$. Let $K_1$ be the $Kb$ bundle over $S^2$ with gluing map $\phi$. Then $\pi_1(K_1)$ is obtained from $\pi_1(Kb)$ by adding the image of $\phi(S^1\times \{pt\})$ as a relation. Since $\phi(S^1\times \{pt\})$ is homotopic to $a^2$, $\pi_1(K_1)=\lan a,b|aba\inv b,a^2\ran$. More generally, if $K_n$ is the $Kb$ bundle corresponding to $\phi^n$, we have $\pi_1(K_n)\cong \lan a,b|aba\inv b,a^{2n}\ran$, and in turn, $H_1(K_n) \cong \Z_{2n} \oplus \Z_2$. We see that for each integer $n\geq 0$, the total spaces $K_n$ have distinct homotopy types. 
    
    Let $\psi$ be the generator of $\pi_1(\Diff(Kb)) \cong \Z$ so that $\psi^m =\phi$ for some $m> 0$, then $\psi(S^1\times \{pt\})$ is homotopic to $a^{2k}$ for some integer $k$, otherwise $\pi_1(\Diff(Kb))$ is not cyclic. Then $\psi^m(S^1\times \{pt\})=a^{2mk} = \phi(S^1\times \{pt\})= a^2$. Since $a \in \pi_1(Kb)$ has infinite order, we conclude that $m=k=1$. Hence $\psi=\phi$, and so $\phi$ generates $\pi_1(\Diff(Kb))$.
    \end{proof}
	
	When $n=0$, $K_0=Kb\times S^2$, and it's well--known what the handlebody structure is and that its orientation double cover is $T^2\times S^2$. We extend this knowledge to any $K_n$ for $n\geq 1$. It suffices to restrict our attention to $K_n$ with $n\geq 0$ for the following reason: the inverse map $\phi^{-1}$ is given by $\phi\circ (\mathfrak{a}\times \text{id}):S^1\times Kb\to Kb$, where $\mathfrak{a}$ is the antipodal map. This is isotopic to $\phi$ itself, hence $K_{-n}\cong K_{n}$.
	\begin{figure}
		\centering

		\resizebox{6in}{!}{
		\begin{tikzpicture}
			
			\begin{scope}[scale=.7]
				
				\begin{knot}[
					%draft mode = crossings,
					clip width = 5pt,
					flip crossing/.list={2}
					]
					\strand[thick] (.5,3) -- (2,3) .. controls +(1,0) and +(0,1) .. (3,2) -- (3,.5);
					\strand[thick, rotate = 90] (.5,3) -- (2,3) .. controls +(1,0) and +(0,1) .. (3,2) -- (3,.5);
					\strand[thick, rotate = 180](.5,3) -- (2,3) .. controls +(1,0) and +(0,1) .. (3,2) -- (3,.5);
					\strand[thick, rotate = -90] (.5,3) -- (2,3) .. controls +(1,0) and +(0,1) .. (3,2) -- (3,.5);

					\strand[thick] (-3,0) .. controls +(.1,1) and +(-2,0) .. (0,1.25) .. controls +(2,0) and +(-.5,.5).. (3,0);
					\strand[thick] (-3,0) .. controls +(1,.2) and +(-2,0) .. (0,.2) .. controls +(2,0) and + (-.3,.75) .. (3.25,2) .. controls +(.3,-.75) and +(-.75,1.5) .. (3,0) ;
					
					\strand[thick] (-3,0) .. controls +(.1,-1) and +(-2,0) .. (0,-1.25) .. controls +(2,0) and +(-.1,-1).. (3,0);
					\strand[thick] (-3,0) .. controls +(1,-.2) and +(-2,0) .. (0,-.2) .. controls +(2,0) and +(-1,-.2).. (3,0);

				\end{knot}
				
				% FRAMING AND LABELS
				
				\filldraw[fill=white] (0,.2) circle [radius = .25cm] (0,.2) node[scale=.8] {$1$};
				\draw (-1.75,.85) node{$\vdots$};
				\draw (-1.6, .8) node[anchor = west, scale=.75] {($n$ strands)};
				
				\draw  (-1,-.6) node {$\vdots$};
				\draw (-.9,-.75) node[anchor = west, scale =.75] {($n$ strands)};
				
				% 3 AND 4 HANDLES
				
				\draw (3,-2.5) node[anchor=west] {$\cup\, \tilde h^3, h^3, h^4$};
				\draw (-3.2,3.2) node{$0$};
				
				% 1 HANDLE ATTACHING FEET
				
				\fill[color=white] (-3,0) circle [radius = .5cm];
				\fill[color=white] (3,0) circle [radius = .5cm];
				\fill[color=white] (0,3) circle [radius = .5cm];
				\fill[color=white] (0,-3) circle [radius = .5cm];
				
				\fill[color = pink] (-3.5,0) .. controls +(0,.2775) and +(-.27775,0) .. (-3,.5) .. controls +(.2775,0) and +(0,.2775) .. (-2.5,0);
				\fill[color = pink] (3.5,0) .. controls +(0,-.2775) and +(.27775,0) .. (3,-.5) .. controls +(-.2775,0) and +(0,-.2775) .. (2.5,0);
				
				\draw[thick] (-3,0) circle [radius = .5cm];
				\draw[thick] (3,0) circle [radius = .5cm];
				\draw[thick] (0,3) circle [radius = .5cm];
				\draw[thick] (0,-3) circle [radius = .5cm];

			\end{scope}

			\begin{scope}[xshift = 9cm, scale = .7]
				
				\begin{knot}[
					%draft mode = crossings,
					clip width = 5pt,
					flip crossing/.list={2}
					]
					\strand[thick] (.5,3) -- (2,3) .. controls +(1,0) and +(0,1) .. (3,2) -- (3,.5);
					\strand[thick, rotate = 90] (.5,3) -- (2,3) .. controls +(1,0) and +(0,1) .. (3,2) -- (3,.5);
					\strand[thick, rotate = 180](.5,3) -- (2,3) .. controls +(1,0) and +(0,1) .. (3,2) -- (3,.5);
					\strand[thick, rotate = -90] (.5,3) -- (2,3) .. controls +(1,0) and +(0,1) .. (3,2) -- (3,.5);

					\strand[thick] (-3,0) .. controls +(.1,1) and +(-2,0) .. (0,.75) .. controls +(2,0) and +(-.5,.5).. (3,0);
					\strand[thick] (-3,0) .. controls +(.1,-1) and +(-2,0) .. (0,-.75) .. controls +(2,0) and + (-1,0) .. (3,2.25) .. controls +(.5,0) and +(.5,.2) .. (3,1.5) .. controls +(0,0) and +(-.75,1.5) .. (3,0);
					
				\end{knot}
				
				% LABELS AND FRAMING
				
				\draw (-2,.2) node {$\vdots$};
				\draw (-1.8,0) node[anchor = west, scale = .8]{($n$ strands)};
				\draw (-3.2,3.2) node{$0$};
				\draw (3.5,2) node[anchor = west] {$1$};

				% 3 AND 4 HANDLES
				
				\draw (3,-2.5) node[anchor=west] {$\cup\, 2\, h^3,h^4$};
				
				% 1 HANDLE ATTACHING FEET
				
				\fill[color=white] (-3,0) circle [radius = .5cm];
				\fill[color=white] (3,0) circle [radius = .5cm];
				\fill[color=white] (0,3) circle [radius = .5cm];
				\fill[color=white] (0,-3) circle [radius = .5cm];

				\draw[thick] (-3,0) circle [radius = .5cm];
				\draw[thick] (3,0) circle [radius = .5cm];
				\draw[thick] (0,3) circle [radius = .5cm];
				\draw[thick] (0,-3) circle [radius = .5cm];
				
			\end{scope}
			
		\end{tikzpicture}
		}

		\caption{Left: Klein bottle fiber bundles over $S^2$. Right: $T^2$ fiber bundles over $S^2$. In both figures, all $n$ strands run parallel to each other except for the strands that link around the fiber $2$--handle.}
		\label{fig:KB_bundles}
	\end{figure}
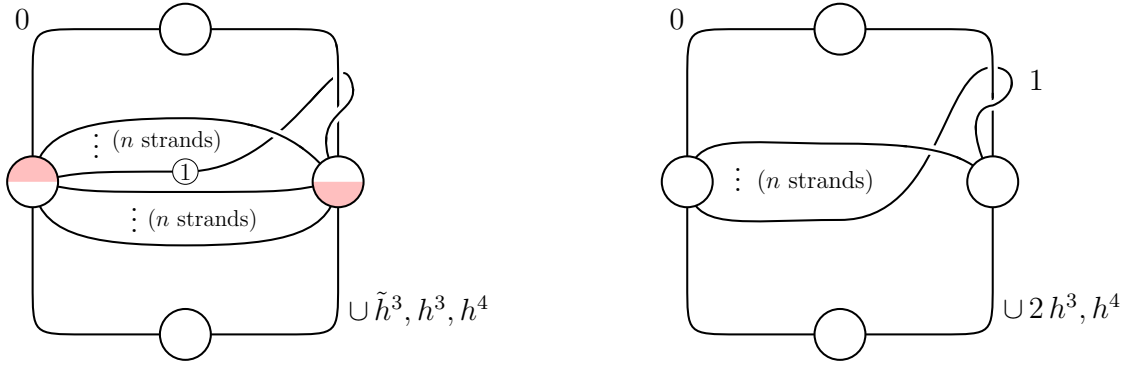
	
	\begin{proposition} \label{prop:KB_bundle_handlebody}
        $K_n$ has a Kirby diagram as shown on the left in Figure \ref{fig:KB_bundles}, and the orientation double cover of $K_n$ is $S^1\times L(n,1)$, which is shown on the right in Figure \ref{fig:KB_bundles}.
	\end{proposition}
	
	\begin{proof}
        Let $f\colon K_n\to S^2$ be a Klein Bottle bundle and let $p\colon M_n\to K_n$ be the orientation double cover. We first claim that $M_n$ is a locally trivial torus fibration over $S^2$. To see this, fix $x\in M_n$. There is a disk neighborhood $U_1\ni f\circ p(x)$ over which $f\colon K_n\to S^2$ is modeled by $pr_1\colon U_1\times Kb\to U_1$. There is another disk neighborhood $U_2\ni f \circ p(x)$ such that $p|_{(f \circ p)\inv(U_2)}$ is an orientation double cover of $p(U_2)$. Let $D$ be a disk neighborhood in $\text{Int}(U_1\cap U_2)$. Then $f$ is modeled by $pr_1\colon D\times Kb\to D$ over $f\inv(D)$ and $p$ is modeled by the standard orientation double cover $D^2\times T^2\to D^2\times Kb$ over $(f\circ p)\inv(D)$. This confirms that $f\circ p\colon M_n\to S^2$ is a locally trivial $T^2$ bundle. A diffeomorphism classification of total spaces of these bundles was given in \cite{Baykur_Kamada}, where it was shown that the only closed, orientable 4-manifolds admitting a locally trivial torus fibration over $S^2$ are $S^2\times T^2$, $S^1\times S^3$, and $S^1\times L(n,1)$. These bundles all admit Kirby diagrams as shown in the right of Figure \ref{fig:KB_bundles}.

		Next we compute the subgroup of $\pi_1(K_n)$ corresponding to $p$, which is isomorphic to $\pi_1(M_n)$. This subgroup is the kernel of $w_1$, the first Stiefel--Whitney class, which sends orientation--preserving loops to $0$ and orientation--reversing loops to $1$.
        Note that $w_1(K_n)(a)=1$ and $w_1(K_n)(b)=0$. Let $H$ be the subgroup of $\pi_1(K_n)$ generated by $a^2$ and $b$. Since $\pi_1(K_n)/H\cong \Z_2$, and both $a^2$ and $b$ are in $\text{ker}(\omega_1(K_n))$, $H=\ker(\omega_1(K_n))$. Applying the Reidemeister-Schreier method, the relations in $H$ are given by conjugating the relations in $G$ by $a$. Hence $H = \langle a^2,b|a(a^{2n})a^{-1},a(aba^{-1}b)a^{-1}\rangle$. Using the relations in $\pi_1(K_n)$, this group may be rewritten as \[H = \langle a^2,b|a^{2n},[a^2,b]\rangle = \langle c,d | c^n, [c,d]\rangle,\] where $c=a^2$ and $d=b$. Then by the classification of total spaces of $T^2$--bundles over $S^2$, we see that $M_n\cong S^1\times L(n,1)$. We can characterise $p_*\colon \pi_1(M_n)\to \pi_1(K_n)$ as sending $c$ to $a^2$ and $d$ to $b$.
        
        To verify that the Kirby diagram in the left of Figure \ref{fig:KB_bundles} is $K_n$, we compute its orientation double--cover using Proposition \ref{prop:dble_cover_alg}, and verify it's indeed is a copy of $S^1\times L(n,1)$. This Kirby calculation is shown in Figure \ref{fig:double_cover_pf}. Part (1) of the figure shows the orientation double cover of the Kirby diagram on the left of Figure \ref{fig:KB_bundles}. Then in (2), we cancel the $0$--handle on the right with the green-shaded $1$--handle, and simultaneously cancel a pair of $3$-- and $4$--handles. When we cancel the $0$--handle, the right--hand side of (1) is ``dragged'' through the green $1$--handle to become part of the left--hand side. In the process, the entire right--hand side reverses orientation, which is why the blue $2$--handle framing changes sign between (1) and (2). Next slide the black $2$--handle over the gray one and cancel the grey $2$--handle with an unshaded $1$--handle to arrive at (3). Sliding the blue $2$--handle over the red one gets us to (4). After this slide, the blue $2$-handle may be isotoped to be disjoint from the pink $1$--handle, as in (5). At this point, the blue $2$--handle is a zero framed unknot which may be isotoped away from the rest of the diagram. Canceling it with a $3$--handle gives (6), which is a Kirby diagram of $S^1\times L(n,1)$.

		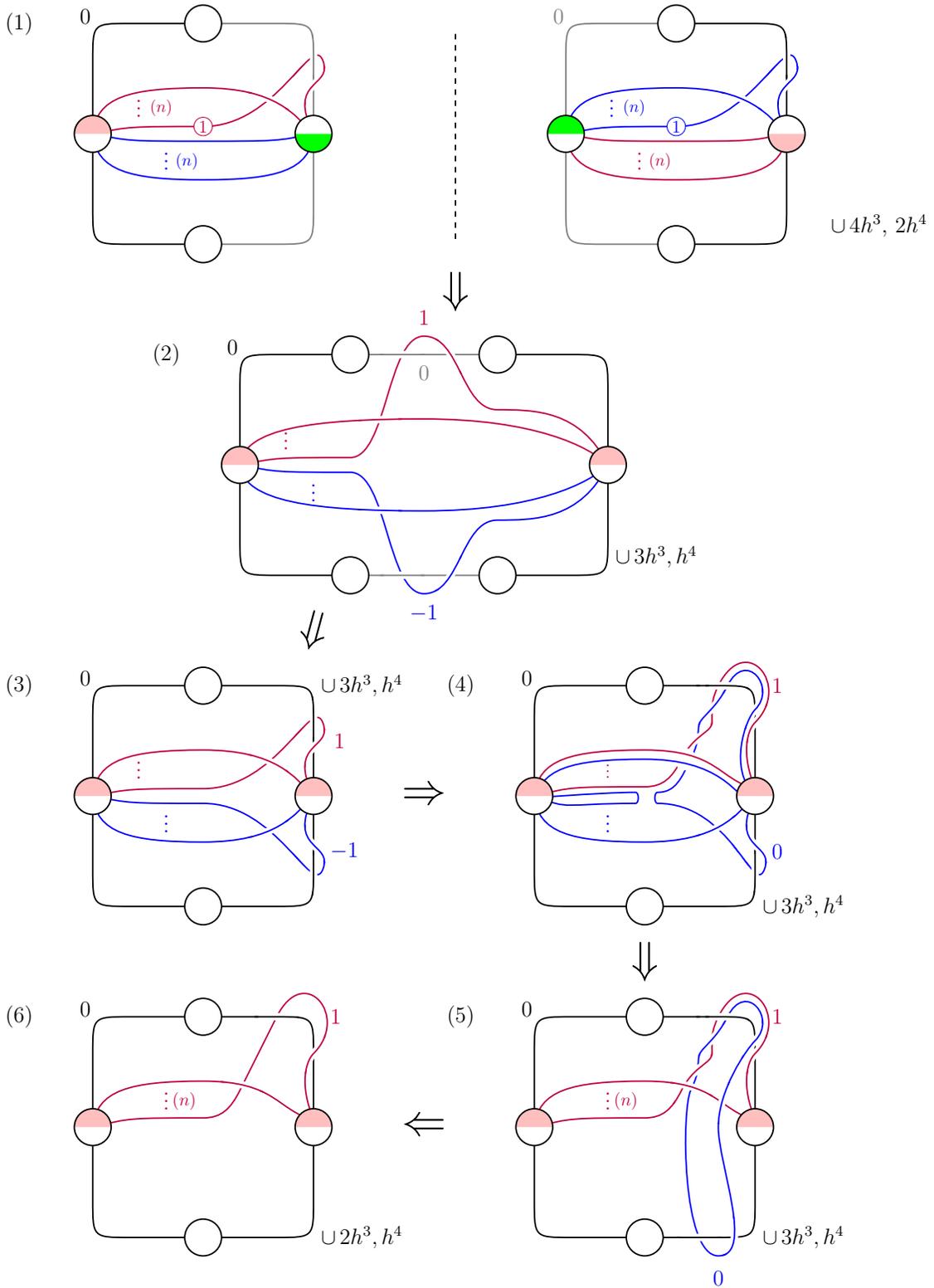
\begin{figure}
			\centering
			\resizebox{6in}{!}{
				
				\begin{tikzpicture}
					
					\begin{scope}[scale=.7]
						
						\draw(-5,3) node {(1)};
						
						\begin{knot}[
							%draft mode = crossings,
							clip width = 5pt,
							flip crossing/.list={2}
							]
							\strand[thick,color=gray] (.5,3) -- (2,3) .. controls +(1,0) and +(0,1) .. (3,2) -- (3,.5);
							\strand[thick, rotate = 90] (.5,3) -- (2,3) .. controls +(1,0) and +(0,1) .. (3,2) -- (3,.5);
							\strand[thick, rotate = 180](.5,3) -- (2,3) .. controls +(1,0) and +(0,1) .. (3,2) -- (3,.5);
							\strand[thick, rotate = -90,color=gray] (.5,3) -- (2,3) .. controls +(1,0) and +(0,1) .. (3,2) -- (3,.5);

							\strand[thick,color=purple] (-3,0) .. controls +(.1,1) and +(-2,0) .. (0,1.25) .. controls +(2,0) and +(-.5,.5).. (3,0);
							\strand[thick,color=purple] (-3,0) .. controls +(1,.2) and +(-2,0) .. (0,.2) .. controls +(2,0) and + (-.3,.75) .. (3.25,2) .. controls +(.3,-.75) and +(-.75,1.5) .. (3,0) ;
							
							\strand[thick,color=blue] (-3,0) .. controls +(.1,-1) and +(-2,0) .. (0,-1.25) .. controls +(2,0) and +(-.1,-1).. (3,0);
							\strand[thick,color=blue] (-3,0) .. controls +(1,-.2) and +(-2,0) .. (0,-.2) .. controls +(2,0) and +(-1,-.2).. (3,0);

						\end{knot}
						
						% FRAMING AND LABELS
						
						\filldraw[fill=white,draw=purple] (0,.2) circle [radius = .25cm] (0,.2) node[scale=.8,color=purple] {$1$};
						\draw (-1.75,.85) node[color=purple]{$\vdots$};
						\draw (-1.6, .7) node[anchor = west, scale=.75,color=purple] {($n$)};
						
						\draw  (-1,-.6) node[color=blue] {$\vdots$};
						\draw (-.9,-.75) node[anchor = west, scale =.75,color=blue] {($n$)};
						
						% 3 AND 4 HANDLES
						
						\draw (-3.2,3.2) node{$0$};
						
						% 1 HANDLE ATTACHING FEET
						
						\fill[color=white] (-3,0) circle [radius = .5cm];
						\fill[color=white] (3,0) circle [radius = .5cm];
						\fill[color=white] (0,3) circle [radius = .5cm];
						\fill[color=white] (0,-3) circle [radius = .5cm];
						
						\fill[color = pink] (-3.5,0) .. controls +(0,.2775) and +(-.27775,0) .. (-3,.5) .. controls +(.2775,0) and +(0,.2775) .. (-2.5,0);
						\fill[color = green] (3.5,0) .. controls +(0,-.2775) and +(.27775,0) .. (3,-.5) .. controls +(-.2775,0) and +(0,-.2775) .. (2.5,0);
						
						\draw[thick] (-3,0) circle [radius = .5cm];
						\draw[thick] (3,0) circle [radius = .5cm];
						\draw[thick] (0,3) circle [radius = .5cm];
						\draw[thick] (0,-3) circle [radius = .5cm];

					\end{scope}
					
					\draw[thick,dashed] (4.8,-2) -- (4.8,2);
					\draw (4.8,-3) node[scale=2] {$\Downarrow$};
					
					\begin{scope}[xshift = 9cm, scale = .7]
						\begin{knot}[
							%draft mode = crossings,
							clip width = 5pt,
							flip crossing/.list={2}
							]
							\strand[thick] (.5,3) -- (2,3) .. controls +(1,0) and +(0,1) .. (3,2) -- (3,.5);
							\strand[thick, rotate = 90, color=gray] (.5,3) -- (2,3) .. controls +(1,0) and +(0,1) .. (3,2) -- (3,.5);
							\strand[thick, rotate = 180,color=gray](.5,3) -- (2,3) .. controls +(1,0) and +(0,1) .. (3,2) -- (3,.5);
							\strand[thick, rotate = -90] (.5,3) -- (2,3) .. controls +(1,0) and +(0,1) .. (3,2) -- (3,.5);

							\strand[thick,color=blue] (-3,0) .. controls +(.1,1) and +(-2,0) .. (0,1.25) .. controls +(2,0) and +(-.5,.5).. (3,0);
							\strand[thick,color=blue] (-3,0) .. controls +(1,.2) and +(-2,0) .. (0,.2) .. controls +(2,0) and + (-.3,.75) .. (3.25,2) .. controls +(.3,-.75) and +(-.75,1.5) .. (3,0) ;
							
							\strand[thick,color=purple] (-3,0) .. controls +(.1,-1) and +(-2,0) .. (0,-1.25) .. controls +(2,0) and +(-.1,-1).. (3,0);
							\strand[thick,color=purple] (-3,0) .. controls +(1,-.2) and +(-2,0) .. (0,-.2) .. controls +(2,0) and +(-1,-.2).. (3,0);

						\end{knot}
						
						% FRAMING AND LABELS
						
						\filldraw[fill=white,draw=blue] (0,.2) circle [radius = .25cm] (0,.2) node[scale=.8,color=blue] {$1$};
						\draw (-1.75,.85) node[color=blue]{$\vdots$};
						\draw (-1.6, .7) node[anchor = west, scale=.75,color=blue] {($n$)};
						
						\draw  (-1,-.6) node[color=purple] {$\vdots$};
						\draw (-.9,-.75) node[anchor = west, scale =.75,color=purple] {($n$)};
						
						% 3 AND 4 HANDLES
						
						\draw (4,-2.5) node[anchor=west] {$\cup\, 4 h^3,\, 2 h^4$};
						\draw (-3.2,3.2) node[color=gray]{$0$};
						
						% 1 HANDLE ATTACHING FEET
						
						\fill[color=white] (-3,0) circle [radius = .5cm];
						\fill[color=white] (3,0) circle [radius = .5cm];
						\fill[color=white] (0,3) circle [radius = .5cm];
						\fill[color=white] (0,-3) circle [radius = .5cm];
						
						\fill[color = green] (-3.5,0) .. controls +(0,.2775) and +(-.27775,0) .. (-3,.5) .. controls +(.2775,0) and +(0,.2775) .. (-2.5,0);
						\fill[color = pink] (3.5,0) .. controls +(0,-.2775) and +(.27775,0) .. (3,-.5) .. controls +(-.2775,0) and +(0,-.2775) .. (2.5,0);
						
						\draw[thick] (-3,0) circle [radius = .5cm];
						\draw[thick] (3,0) circle [radius = .5cm];
						\draw[thick] (0,3) circle [radius = .5cm];
						\draw[thick] (0,-3) circle [radius = .5cm];

					\end{scope}
					
					\begin{scope} [scale=.7,yshift = -9cm, xshift = 6cm]
						
						\draw(-7,3) node {(2)};
						
						\begin{knot}[
							%draft mode = crossings,
							clip width = 5pt,
							flip crossing/.list={2,4}
							]
							\strand[thick,xshift=2cm] (.5,3) -- (2,3) .. controls +(1,0) and +(0,1) .. (3,2) -- (3,.5);
							\strand[thick,  xshift=-2cm, rotate = 90] (.5,3) -- (2,3) .. controls +(1,0) and +(0,1) .. (3,2) -- (3,.5);
							\strand[thick, xshift=-2cm, rotate = 180](.5,3) -- (2,3) .. controls +(1,0) and +(0,1) .. (3,2) -- (3,.5);
							\strand[thick,xshift=2cm, rotate = -90] (.5,3) -- (2,3) .. controls +(1,0) and +(0,1) .. (3,2) -- (3,.5);
							\strand[thick,color=gray](-2.5,3) -- (2.5,3);
							\strand[thick,color=gray](-2.5,-3) -- (2.5,-3);
							
							\strand[thick,color=purple] (-5,0) .. controls +(.1,1) and +(-3,0) .. (0,1.25) .. controls +(3,0) and +(-.5,.5).. (5,0);
							\strand[thick,color=purple] (-5,0) .. controls +(1,.2) and +(-2,0) .. (-2,.2) .. controls +(1,0) and + (-1,0) .. (0,3.5) .. controls + (1,0) and +(-1,0) .. (2,1.5) .. controls +(1,0) and +(-.75,1.5) .. (5,0) ;
							
							\strand[thick,color=blue] (-5,0) .. controls +(.1,-1) and +(-3,0) .. (0,-1.25) .. controls +(3,0) and +(-.5,-.5).. (5,0);
							\strand[thick,color=blue, yscale=-1] (-5,0) .. controls +(1,.2) and +(-2,0) .. (-2,.2) .. controls +(1,0) and + (-1,0) .. (0,3.5) .. controls + (1,0) and +(-1,0) .. (2,1.5) .. controls +(1,0) and +(-.75,1.5) .. (5,0) ;
							
						\end{knot}
						
						% LABELS AND FRAMING
						
						\draw(0,3.6) node[anchor=south, color = purple] {$1$};
						\draw(0,-3.6) node[anchor=north,color=blue] {$-1$};
						\draw (5,-2.5) node[anchor=west] {$\cup\, 3 h^3, h^4$};
						\draw (-5.2,3.2) node{$0$};
						\draw(0,2.9) node[anchor=north, color=gray] {$0$};
						\draw (-3.75,.8) node[color=purple]{$\vdots$};
						\draw  (-3,-.6) node[color=blue] {$\vdots$};
						
						% 1 HANDLE ATTACHING FEET
						
						\fill[color=white] (-5,0) circle [radius = .5cm];
						\fill[color=white] (5,0) circle [radius = .5cm];
						\fill[color=white] (-2,3) circle [radius = .5cm];
						\fill[color=white] (-2,-3) circle [radius = .5cm];
						\fill[color=white] (2,3) circle [radius = .5cm];
						\fill[color=white] (2,-3) circle [radius = .5cm];
						
						\fill[color = pink] (-5.5,0) .. controls +(0,.2775) and +(-.27775,0) .. (-5,.5) .. controls +(.2775,0) and +(0,.2775) .. (-4.5,0);
						\fill[color = pink] (5.5,0) .. controls +(0,+.2775) and +(.27775,0) .. (5,.5) .. controls +(-.2775,0) and +(0,.2775) .. (4.5,0);
						
						\draw[thick] (-5,0) circle [radius = .5cm];
						\draw[thick] (5,0) circle [radius = .5cm];
						\draw[thick] (-2,3) circle [radius = .5cm];
						\draw[thick] (-2,-3) circle [radius = .5cm];
						\draw[thick] (2,3) circle [radius = .5cm];
						\draw[thick] (2,-3) circle [radius = .5cm];
						
						\draw (-3,-4.5) node[scale=2,rotate = -15] {$\Downarrow$};
						
					\end{scope}
					
					\begin{scope}[scale=.7,xshift = 0cm, yshift = -18cm]
						
						\draw(-5,3) node {(3)};
						
						\begin{knot}[
							%draft mode = crossings,
							clip width = 5pt,
							flip crossing/.list={2,3}
							]
							\strand[thick] (.5,3) -- (2,3) .. controls +(1,0) and +(0,1) .. (3,2) -- (3,.5);
							\strand[thick, rotate = 90] (.5,3) -- (2,3) .. controls +(1,0) and +(0,1) .. (3,2) -- (3,.5);
							\strand[thick, rotate = 180](.5,3) -- (2,3) .. controls +(1,0) and +(0,1) .. (3,2) -- (3,.5);
							\strand[thick, rotate = -90] (.5,3) -- (2,3) .. controls +(1,0) and +(0,1) .. (3,2) -- (3,.5);

							\strand[thick,color=purple] (-3,0) .. controls +(.1,1) and +(-2,0) .. (0,1.25) .. controls +(2,0) and +(-.5,.5).. (3,0);
							\strand[thick,color=purple] (-3,0) .. controls +(1,.2) and +(-2,0) .. (0,.2) .. controls +(2,0) and + (-.3,.75) .. (3.25,2) .. controls +(.3,-.75) and +(-.75,1.5) .. (3,0) ;
							
							% REFLECTED KNOT
							
							\strand[thick,color=blue, yscale=-1] (-3,0) .. controls +(.1,1) and +(-2,0) .. (0,1.25) .. controls +(2,0) and +(-.5,.5).. (3,0);
							\strand[thick,color=blue,yscale=-1] (-3,0) .. controls +(1,.2) and +(-2,0) .. (0,.2) .. controls +(2,0) and + (-.3,.75) .. (3.25,2) .. controls +(.3,-.75) and +(-.75,1.5) .. (3,0) ;

						\end{knot}
						
						% FRAMING AND LABELS
						
						\draw (-1.75,.85) node[color=purple]{$\vdots$};
						\draw  (-1,-.6) node[color=blue] {$\vdots$};
						\draw (3.35, 1.5) node[anchor=west,color=purple] {$1$};
						\draw (3.25, -1.5) node[anchor=west,color=blue] {$-1$};
						
						% 3 AND 4 HANDLES
						
						\draw (3,3) node[anchor=west] {$\cup\, 3 h^3, h^4$};
						\draw (-3.2,3.2) node{$0$};
						
						% 1 HANDLE ATTACHING FEET
						
						\fill[color=white] (-3,0) circle [radius = .5cm];
						\fill[color=white] (3,0) circle [radius = .5cm];
						\fill[color=white] (0,3) circle [radius = .5cm];
						\fill[color=white] (0,-3) circle [radius = .5cm];
						
						\fill[color = pink] (-3.5,0) .. controls +(0,.2775) and +(-.27775,0) .. (-3,.5) .. controls +(.2775,0) and +(0,.2775) .. (-2.5,0);
						\fill[color = pink] (3.5,0) .. controls +(0,+.2775) and +(.27775,0) .. (3,.5) .. controls +(-.2775,0) and +(0,.2775) .. (2.5,0);
						
						\draw[thick] (-3,0) circle [radius = .5cm];
						\draw[thick] (3,0) circle [radius = .5cm];
						\draw[thick] (0,3) circle [radius = .5cm];
						\draw[thick] (0,-3) circle [radius = .5cm];

					\end{scope}
					
					\begin{scope}[scale=.7]
						\draw(6,-18) node [scale=2] {$\Rightarrow$};
					\end{scope}

					\begin{scope}[scale=.7, xshift = 12cm, yshift =-18cm]
						
						\draw(-5,3) node {(4)};
						
						\begin{knot}[
							%draft mode = crossings,
							clip width = 5pt,
							flip crossing/.list={2,4,5,8,9}
							]
							\strand[thick] (.5,3) -- (2,3) .. controls +(1,0) and +(0,1) .. (3,2) -- (3,.5);
							\strand[thick, rotate = 90] (.5,3) -- (2,3) .. controls +(1,0) and +(0,1) .. (3,2) -- (3,.5);
							\strand[thick, rotate = 180](.5,3) -- (2,3) .. controls +(1,0) and +(0,1) .. (3,2) -- (3,.5);
							\strand[thick, rotate = -90] (.5,3) -- (2,3) .. controls +(1,0) and +(0,1) .. (3,2) -- (3,.5);

							\strand[thick,color=purple] (-3,0) .. controls +(1,.25) and +(-2,0) .. (0,.25) .. controls +(.5,0) and +(-.25,-.5) .. (1,1) .. controls + (.3,.6) and + (-.1,-.2) .. (1.8,1.9) .. controls +(.1,.2) and +(-.3,-.6).. (2,3) .. controls +(.75,1.5) and +(1,1) .. (3,2) .. controls +(-.25,-.25) and + (-.5,1) .. (3.1,0) ;
							
							% REFLECTED KNOT
							
							\strand[thick,color=blue,yscale=-1] (.35,.2) .. controls +(2,0) and + (-.3,.75) .. (3.25,2) .. controls +(.3,-.75) and +(-.75,1.5) .. (3,0) ;
							
							\strand[thick,color=blue, yscale=-1] (-3,0) .. controls +(.1,1) and +(-2,0) .. (0,1.25)  .. controls +(2,0) and +(-.5,.5).. (3,0);
							
							% HANDLESLIDE
							
							\strand[thick,color=blue] (.35,.1) .. controls +(.5,0) and +(-.25,-.5) .. (1.25,1) .. controls + (.3,.6) and + (-.1,-.2) .. (1.5,2) .. controls +(.1,.2) and +(-.3,-.6).. (2.25,3) .. controls +(.5,1) and +(.6,.6) .. (3,2.35) .. controls +(-.25,-.25) and + (-.75,.75) .. (3,0);

							\strand[thick,color=blue] (.35,.1) .. controls +(-.1,0) and +(-.1,0) .. (.35,-.2);
							
							\strand[thick,color=blue] (-3,0) .. controls +(1,-.4) and +(-2,0).. (-.25,-.2) .. controls +(.1,0) and +(.1,0) .. (-.25,.1) .. controls +(-1,0) and + (-1,-.1) .. (-2.5,0);

						\end{knot}
						
						%REST OF KNOT
						
						\fill[color=white] (1.1,1) circle [radius = .3cm];
						
						\draw[thick,color=purple] (-3,0) .. controls +(.1,1) and +(-2,0) .. (0,1.25) .. controls +(2,0) and +(-1,.25).. (3,.2);
						
						\draw[thick,color=blue] (-3,-.25) .. controls +(.1,1) and +(-2,0) .. (0,1)  .. controls +(2,0) and +(-.5,.5).. (3,-.4);

						% FRAMING AND LABELS
						
						\draw (-1,.8) node[color=purple,scale=.8]{$\vdots$};
						\draw  (-1,-.6) node[color=blue] {$\vdots$};
						\draw (3.25, 3) node[anchor=west,color=purple] {$1$};
						\draw (3.25, -1.5) node[anchor=west,color=blue] {$0$};
						
						% 3 AND 4 HANDLES
						
						\draw (3,-3) node[anchor=west] {$\cup\, 3 h^3, h^4$};
						\draw (-3.2,3.2) node{$0$};
						
						% 1 HANDLE ATTACHING FEET

						\fill[color=white] (-3,0) circle [radius = .5cm];
						\fill[color=white] (3,0) circle [radius = .5cm];
						\fill[color=white] (0,3) circle [radius = .5cm];
						\fill[color=white] (0,-3) circle [radius = .5cm];
						
						\fill[color = pink] (-3.5,0) .. controls +(0,.2775) and +(-.27775,0) .. (-3,.5) .. controls +(.2775,0) and +(0,.2775) .. (-2.5,0);
						\fill[color = pink] (3.5,0) .. controls +(0,+.2775) and +(.27775,0) .. (3,.5) .. controls +(-.2775,0) and +(0,.2775) .. (2.5,0);
						
						\draw[thick] (-3,0) circle [radius = .5cm];
						\draw[thick] (3,0) circle [radius = .5cm];
						\draw[thick] (0,3) circle [radius = .5cm];
						\draw[thick] (0,-3) circle [radius = .5cm];

					\end{scope}
					
					\begin{scope}[scale=.7]
						\draw(12,-22.5) node [scale=2] {$\Downarrow$};
					\end{scope}
					
					\begin{scope}[scale=.7, xshift = 12cm, yshift = -27cm]
						
						\draw(-5,3) node {(5)};
						
						\begin{knot}[
							%draft mode = crossings,
							clip width = 5pt,
							flip crossing/.list={2,4,5,9,11}
							]
							\strand[thick] (.5,3) -- (2,3) .. controls +(1,0) and +(0,1) .. (3,2) -- (3,.5);
							\strand[thick, rotate = 90] (.5,3) -- (2,3) .. controls +(1,0) and +(0,1) .. (3,2) -- (3,.5);
							\strand[thick, rotate = 180](.5,3) -- (2,3) .. controls +(1,0) and +(0,1) .. (3,2) -- (3,.5);
							\strand[thick, rotate = -90] (.5,3) -- (2,3) .. controls +(1,0) and +(0,1) .. (3,2) -- (3,.5);

							\strand[thick,color=purple] (-3,0) .. controls +(1,.25) and +(-2,0) .. (0,.25) .. controls +(.5,0) and +(-.25,-.5) .. (1,1) .. controls + (.3,.6) and + (-.1,-.2) .. (1.8,1.9) .. controls +(.1,.2) and +(-.3,-.6).. (2,3) .. controls +(.75,1.5) and +(1,1) .. (3,2) .. controls +(-.25,-.25) and + (-.5,1) .. (3.1,0) ;
							
							\strand[thick,color=purple] (-3,0) .. controls +(.1,1) and +(-2,0) .. (0,1.25) .. controls +(2,0) and +(-1,.25).. (3,.2);

							% HANDLESLIDE
							
							\strand[thick,color=blue] (2,0) .. controls +(0,-1) and +(1,0) .. (2,-3.5) .. controls +(-1,0) and +(-.25,-.5) .. (1.25,1) .. controls + (.3,.6) and + (-.1,-.2) .. (1.5,2) .. controls +(.1,.2) and +(-.3,-.6).. (2.25,3) .. controls +(.5,1) and +(.6,.6) .. (3,2.35) .. controls +(-.25,-.25) and + (0,1) .. (2,0);

						\end{knot}

						% FRAMING AND LABELS
						
						\draw (-1,.85) node[color=purple]{$\vdots$};
						\draw (-1,.7) node[color=purple,anchor=west, scale=.8]{($n$)};
						\draw (3.25, 3) node[anchor=west,color=purple] {$1$};
						\draw (2,-3.7) node [anchor=north, color=blue] {$0$};
						
						% 3 AND 4 HANDLES
						
						\draw (3,-3) node[anchor=west] {$\cup\, 3 h^3, h^4$};
						\draw (-3.2,3.2) node{$0$};
						
						% 1 HANDLE ATTACHING FEET

						\fill[color=white] (-3,0) circle [radius = .5cm];
						\fill[color=white] (3,0) circle [radius = .5cm];
						\fill[color=white] (0,3) circle [radius = .5cm];
						\fill[color=white] (0,-3) circle [radius = .5cm];
						
						\fill[color = pink] (-3.5,0) .. controls +(0,.2775) and +(-.27775,0) .. (-3,.5) .. controls +(.2775,0) and +(0,.2775) .. (-2.5,0);
						\fill[color = pink] (3.5,0) .. controls +(0,+.2775) and +(.27775,0) .. (3,.5) .. controls +(-.2775,0) and +(0,.2775) .. (2.5,0);
						
						\draw[thick] (-3,0) circle [radius = .5cm];
						\draw[thick] (3,0) circle [radius = .5cm];
						\draw[thick] (0,3) circle [radius = .5cm];
						\draw[thick] (0,-3) circle [radius = .5cm];

					\end{scope}
					
					\begin{scope}[scale=.7]
						\draw(6,-27) node [scale=2] {$\Leftarrow$};
					\end{scope}
					
					\begin{scope}[scale=.7, xshift = 0cm, yshift = -27cm]
						
						\draw(-5,3) node {(6)};
						
						\begin{knot}[
							%draft mode = crossings,
							clip width = 5pt,
							flip crossing/.list={2,3}
							]
							\strand[thick] (.5,3) -- (2,3) .. controls +(1,0) and +(0,1) .. (3,2) -- (3,.5);
							\strand[thick, rotate = 90] (.5,3) -- (2,3) .. controls +(1,0) and +(0,1) .. (3,2) -- (3,.5);
							\strand[thick, rotate = 180](.5,3) -- (2,3) .. controls +(1,0) and +(0,1) .. (3,2) -- (3,.5);
							\strand[thick, rotate = -90] (.5,3) -- (2,3) .. controls +(1,0) and +(0,1) .. (3,2) -- (3,.5);

							\strand[thick,color=purple] (-3,0) .. controls +(1,.25) and +(-2,0) .. (0,.25) .. controls +(.5,0) and +(-.25,-.5) .. (1,1) -- (2,3) .. controls +(.75,1.5) and +(1,1) .. (3,2) .. controls +(-.25,-.25) and + (-.5,1) .. (3.1,0) ;
							
							\strand[thick,color=purple] (-3,0) .. controls +(.1,1) and +(-2,0) .. (0,1.25) .. controls +(2,0) and +(-1,.25).. (3,.2);

						\end{knot}

						% FRAMING AND LABELS
						
						\draw (-1,.85) node[color=purple]{$\vdots$};
						\draw (-1,.7) node[color=purple,anchor=west, scale=.8]{($n$)};
						\draw (3.25, 3) node[anchor=west,color=purple] {$1$};
						
						% 3 AND 4 HANDLES
						
						\draw (3,-3) node[anchor=west] {$\cup\, 2 h^3, h^4$};
						\draw (-3.2,3.2) node{$0$};
						
						% 1 HANDLE ATTACHING FEET

						\fill[color=white] (-3,0) circle [radius = .5cm];
						\fill[color=white] (3,0) circle [radius = .5cm];
						\fill[color=white] (0,3) circle [radius = .5cm];
						\fill[color=white] (0,-3) circle [radius = .5cm];
						
						\fill[color = pink] (-3.5,0) .. controls +(0,.2775) and +(-.27775,0) .. (-3,.5) .. controls +(.2775,0) and +(0,.2775) .. (-2.5,0);
						\fill[color = pink] (3.5,0) .. controls +(0,+.2775) and +(.27775,0) .. (3,.5) .. controls +(-.2775,0) and +(0,.2775) .. (2.5,0);
						
						\draw[thick] (-3,0) circle [radius = .5cm];
						\draw[thick] (3,0) circle [radius = .5cm];
						\draw[thick] (0,3) circle [radius = .5cm];
						\draw[thick] (0,-3) circle [radius = .5cm];

					\end{scope}
					
				\end{tikzpicture}
				
			}

			\caption{The ellipses and label $(n)$ indicate that there are $n$  strands. In (1), all $n$ strands run parallel to each other except for the strands with $+1$--framing.}
			\label{fig:double_cover_pf}
		\end{figure}
		
	\end{proof}

	\subsection{Nonorientable genus two SBLFs with only indefinite fold singularity.}\label{sec:g2_fold_only}
	
	Next, we consider nonorientable genus two SBLFs with an indefinite fold but no Lefschetz singularities. The total space $X$ of this class of SBLFs may be decomposed as $X=Kb\times D^2\cup H \cup S^2\times D^2$, where $H$ is a round $2$--handle cobordism, such that $\partial H = \partial_+ H \sqcup \partial_- H$ with $\partial_+ H \cong Kb \x S^1$ and $\partial_-H\cong S^1\times S^2$. We obtain a handle diagram for the total space of the SBLF by considering how $H$ is glued to $\partial (Kb\times D^2)$, and then in turn how $S^2\times D^2$ is glued to $\partial(Kb\times D^2\cup H)=\partial_- H$.

    \begin{proposition} \label{prop:fold_only_g2}
	Let $f \colon X \to S^2$ be an SBLF with $Z_f \neq \emptyset$, $C_f= \emptyset$ and higher-genus fiber $F_+=Kb$. Then $X$ is diffeomorphic to one of the following:\begin{itemize}
        \item $S^1\twist S^3 \# S^2\times S^2$
	\item $S^1\twist S^3 \# 2 \CP$
	\item $N_n$, whose orientation--double cover is $L_n\# S^2\times S^2$
	\item $N'_n$, whose orientation--double cover is $L'_n\# S^2\times S^2$, 
    \end{itemize} where the manifolds $N_n$ and $N'_n$ have the Kirby diagrams as shown in Figure \ref{fig:fold_only_1}.
    \end{proposition}
	
    \begin{proof}
		
	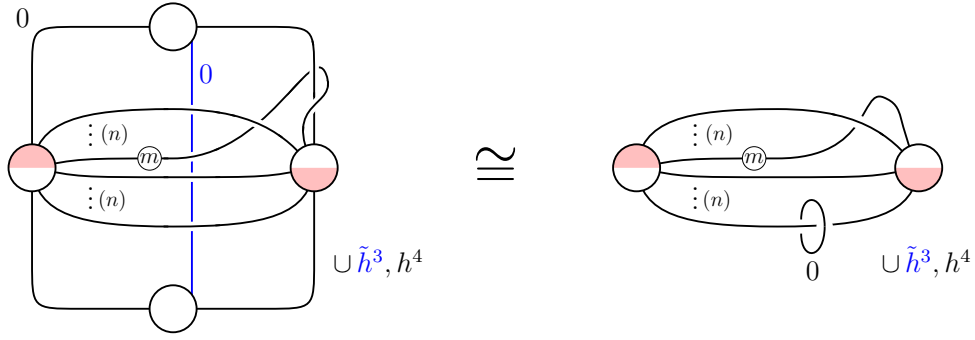
\begin{figure}
			\centering

			\resizebox{!}{1.75in}{
			\begin{tikzpicture}
				
				\begin{scope}[scale=.7]
					
					\begin{knot}[
						%draft mode = crossings,
						clip width = 5pt,
						flip crossing/.list={2}
						]
						\strand[thick] (.5,3) -- (2,3) .. controls +(1,0) and +(0,1) .. (3,2) -- (3,.5);
						\strand[thick, rotate = 90] (.5,3) -- (2,3) .. controls +(1,0) and +(0,1) .. (3,2) -- (3,.5);
						\strand[thick, rotate = 180](.5,3) -- (2,3) .. controls +(1,0) and +(0,1) .. (3,2) -- (3,.5);
						\strand[thick, rotate = -90] (.5,3) -- (2,3) .. controls +(1,0) and +(0,1) .. (3,2) -- (3,.5);

						\strand[thick] (-3,0) .. controls +(.1,1) and +(-2,0) .. (0,1.25) .. controls +(2,0) and +(-.5,.5).. (3,0);
						\strand[thick] (-3,0) .. controls +(1,.2) and +(-2,0) .. (0,.2) .. controls +(2,0) and + (-.3,.75) .. (3.25,2) .. controls +(.3,-.75) and +(-.75,1.5) .. (3,0) ;
						
						\strand[thick] (-3,0) .. controls +(.1,-1) and +(-2,0) .. (0,-1.25) .. controls +(2,0) and +(-.1,-1).. (3,0);
						\strand[thick] (-3,0) .. controls +(1,-.2) and +(-2,0) .. (0,-.2) .. controls +(2,0) and +(-1,-.2).. (3,0);
						
						\strand[thick,color = blue] (.4,3)--(.4,-3);

					\end{knot}
					
					% FRAMING AND LABELS
					
					\draw (.7,2) node[color=blue] {$0$};
					
					\filldraw[fill=white] (-.5,.2) circle [radius = .25cm] (-.5,.2) node[scale=.8] {$m$};
					\draw (-1.75,.85) node{$\vdots$};
					\draw (-1.7, .7) node[anchor = west, scale=.75] {($n$)};
					
					\draw  (-1.75,-.5) node {$\vdots$};
					\draw (-1.75,-.7) node[anchor = west, scale =.75] {($n$)};
					
					% 3 AND 4 HANDLES
					
					\draw (3.2,-2) node[anchor=west] {$\cup\, \textcolor{blue}{\tilde h^3}, h^4$};
					\draw (-3.2,3.2) node{$0$};
					
					% 1 HANDLE ATTACHING FEET
					
					\fill[color=white] (-3,0) circle [radius = .5cm];
					\fill[color=white] (3,0) circle [radius = .5cm];
					\fill[color=white] (0,3) circle [radius = .5cm];
					\fill[color=white] (0,-3) circle [radius = .5cm];
					
					\fill[color = pink] (-3.5,0) .. controls +(0,.2775) and +(-.27775,0) .. (-3,.5) .. controls +(.2775,0) and +(0,.2775) .. (-2.5,0);
					\fill[color = pink] (3.5,0) .. controls +(0,-.2775) and +(.27775,0) .. (3,-.5) .. controls +(-.2775,0) and +(0,-.2775) .. (2.5,0);
					
					\draw[thick] (-3,0) circle [radius = .5cm];
					\draw[thick] (3,0) circle [radius = .5cm];
					\draw[thick] (0,3) circle [radius = .5cm];
					\draw[thick] (0,-3) circle [radius = .5cm];
					
				\end{scope}
				
				\draw (4.8,0) node {\huge $\cong$};
				
				\begin{scope}[xshift = 9cm, scale = .7]
					
					\begin{knot}[
						%draft mode = crossings,
						clip width = 7pt,
						]

						\strand[thick] (-3,0) .. controls +(.1,1) and +(-2,0) .. (0,1.25) .. controls +(2,0) and +(-.5,.5).. (3,0);
						\strand[thick] (-3,0) .. controls +(1,.2) and +(-2,0) .. (0,.2) .. controls +(2,0) and + (-.5,.2) .. (2.25,1.5) .. controls +(.5,-.2) and +(-.5,1) .. (3,0) ;
						
						\strand[thick] (-3,0) .. controls +(.1,-1) and +(-2,0) .. (0,-1.25) .. controls +(2,0) and +(-.1,-1).. (3,0);
						\strand[thick] (-3,0) .. controls +(1,-.2) and +(-2,0) .. (0,-.2) .. controls +(2,0) and +(-1,-.2).. (3,0);

					\end{knot}
					
					% MERIDIAN 2 HANDLE
					
					\fill[white](1,-1.25) circle[radius = 5pt];
					\draw[thick] (1,-1.25) .. controls + (0,.7) and + (0,.6) .. (.5,-1.1);
					\draw[thick] (1,-1.25) .. controls + (0,-.7) and + (0,-.6) .. (.5,-1.4);
					\draw(.75, -2.2) node {$0$};

					% FRAMING AND LABELS
					
					\filldraw[fill=white] (-.5,.2) circle [radius = .25cm] (-.5,.2) node[scale=.8] {$m$};
					\draw (-1.75,.85) node{$\vdots$};
					\draw (-1.7, .7) node[anchor = west, scale=.75] {($n$)};
					
					\draw  (-1.75,-.5) node {$\vdots$};
					\draw (-1.75,-.7) node[anchor = west, scale =.75] {($n$)};
					
					%  3 AND 4 HANDLES
					
					\draw (2,-2) node[anchor=west] {$\cup\, \textcolor{blue}{\tilde h^3}, h^4$};
					
					% 1 HANDLE ATTACHING FEET
					
					\fill[color=white] (-3,0) circle [radius = .5cm];
					\fill[color=white] (3,0) circle [radius = .5cm];
					
					\fill[color = pink] (-3.5,0) .. controls +(0,.2775) and +(-.27775,0) .. (-3,.5) .. controls +(.2775,0) and +(0,.2775) .. (-2.5,0);
					\fill[color = pink] (3.5,0) .. controls +(0,-.2775) and +(.27775,0) .. (3,-.5) .. controls +(-.2775,0) and +(0,-.2775) .. (2.5,0);
					
					\draw[thick] (-3,0) circle [radius = .5cm];
					\draw[thick] (3,0) circle [radius = .5cm];
					
				\end{scope}
				
			\end{tikzpicture}
			}

			\caption{A Kirby diagram of $N_n$ for $m$ even and $N'_n$ for $m$ odd.}
			\label{fig:fold_only_1}
		\end{figure}
		
		A Kirby diagram for $X$ is obtained by adding the round $2$--handle cobordism $H$ to $Kb\times D^2$. The round $2$--handle is attached via a fiber preserving diffeomorphism of $S^1\times Kb$, which may be regarded as an element of $\pi_1(\Diff(Kb))$. By Lemma \ref{lem:diff_KB_gen}, $H$ is glued along $\partial_+ H$ via $\phi^k$ for some $k\geq 0$. For the lower-genus fibers of $X$ to be connected, the $0$--framed $2$--handle of $H$ must trace the unique (up to isotopy) nonseparating two-sided curve on $Kb$. This completely determines the Kirby diagram for $D^2\times Kb\cup H$; it's a copy of $D^2\times Kb$ with an added $0$--framed $2$--handle going over the untwisted $1$--handle, plus a $3$--handle. A priori, this handle decomposition does not say anything about the gluing map $\phi^k\colon \partial_+ H\to \partial(Kb\times D^2)$. Rather, this choice of gluing becomes apparent in the next step, when we attach $S^2\times D^2$. Let $h$ denote the $2$--handle of $S^2\times D^2$. When $H$ is attached to $Kb\times D^2$ by $\phi^k$, the entire fibration structure on $\partial_+H$ is translated, in the sense that it goes $2k$ times over the twisted $1$--handle on the $Kb$ fiber. As a result, the attaching circle of $h$ goes $2k$ times over the twisted $1$--handle. (For an orientable analog, cf. the broken Lefschetz fibration on $S^4$ \cite[Example 3.5]{baykurPJM}, \cite[Section 8.2]{ADK}.) The last thing to consider is that there are two fiber-preserving diffeomorphisms of $S^1\times S^2$, which give two possible ways to glue $\partial(S^2\times D^2)$ to $\partial_- H$. These two choices correspond to $h$ having an even or odd framing $m$, as in Figure \ref{fig:fold_only_1}.
		
	With that, we've exhausted all possible handle decompositions for $X$. A Kirby diagram is now obtained from that of $Kb \times D^2 \cup H$ by adding  $h$ and the remaining $4$--handle. The attaching circle of $h$ is either a meridian of the $2$--handle of the regular $Kb$ fiber or more generally is as in the left of Figure \ref{fig:fold_only_1}. In the latter case, one may slide the fiber $2$--handle over the fold $2$--handle and then cancel the fold $2$--handle and a $1$--handle, so that the Kirby diagram simplifies as shown on the right of Figure \ref{fig:fold_only_1}. If the attaching circle of $h$ forms a meridian, one can use a few simple handle moves to identify the total space as $S^1\twist S^3\# S^2\times S^2$ or $S^1\twist S^3\# \CP\# \overline{\CP} \cong S^1\twist S^3\# 2 \CP$, depending on whether $m$ is even or odd. If $h$ goes over the twisted $1$--handle $2n$ times and has even (resp. odd) framing, its orientation double-cover is  $L_n\# S^2\times S^2$ (resp.  $L'_n\# S^2\times S^2$). Figure \ref{fig:fold_dble_cover} shows the Kirby calculus verifying this last claim, and we proceed by summarizing the Kirby moves involved below.
    
    Figure \ref{fig:fold_dble_cover} (1) applies Proposition \ref{prop:dble_cover_alg} to give the orientation double cover of $M$, based off its Kirby diagram on the right of Figure \ref{fig:fold_only_1}. Then to arrive at (2), we cancel the green $1$--handle and the right-most $0$--handle. Similar to step (1) $\Rightarrow$ (2) in Figure \ref{fig:double_cover_pf}, this step reverses the orientation of the right--hand side of (1). Then slide the black $(-m)$--framed $2$--handle over the blue $m$--framed $2$--handle to arrive at (3). After being slid, the black $2$--handle may be isotoped to not go over the pink $1$--handle, as in (4). Then slide the black (non-meridian) $2$--handle $m+1$ times over the blue meridian to arrive at (5). After sliding the blue meridian over the black one, we arrive at the final Kirby diagram given in (6), which is $L_n\#S^2\times S^2$ for $m$ even and $L'_n\# S^2\times S^2$ for $m$ odd. Note that when the attaching circle of $h$ goes over the twisted $1$--handle twice, Figure \ref{fig:fold_only_1} is an $S^2$--bundle over $\RP$, and the calculation of the orientation double cover in Figure \ref{fig:fold_dble_cover} gives a final space of $S^2\times S^2$,
    as one would expect. Setting $N_0:= S^1\twist S^3 \# S^2\times S^2$ and $N'_0 := S^1\twist S^3 \# 2\CP$, the families $\{N_n\}_{n\geq 0}$ and $\{N'_n\}_{n\geq 0}$ cover the total spaces of all nonorientable genus two SBLFs in this case.
      
			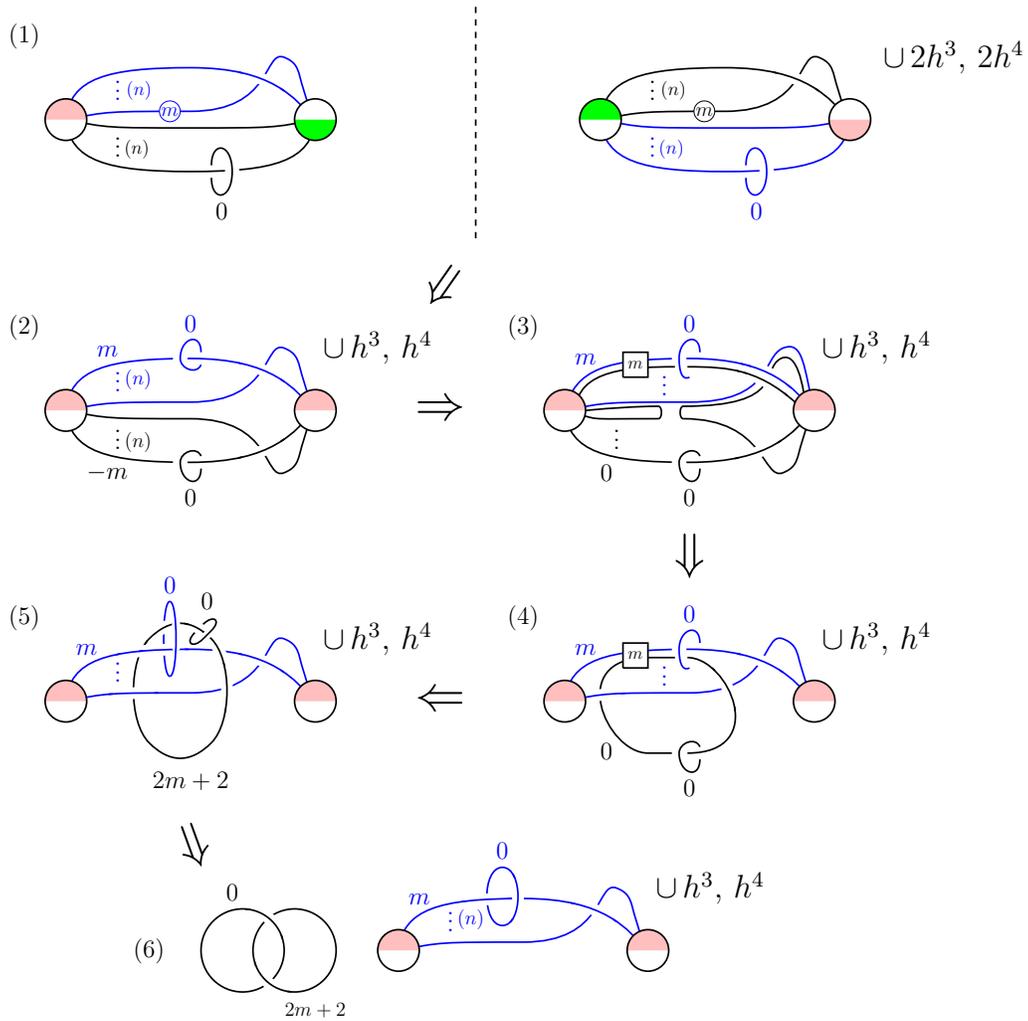
\begin{figure}
			\centering
            
			\resizebox{5.5in}{!}{
				
				\begin{tikzpicture}
					
					\begin{scope}[scale=.7]
						
						\draw(-4,2) node {(1)};
						
						\begin{knot}[
							%draft mode = crossings,
							clip width = 7pt,
							]

							\strand[thick,color=blue] (-3,0) .. controls +(.1,1) and +(-2,0) .. (0,1.25) .. controls +(2,0) and +(-.5,.5).. (3,0);
							\strand[thick,color=blue] (-3,0) .. controls +(1,.2) and +(-2,0) .. (0,.2) .. controls +(2,0) and + (-.5,.2) .. (2.25,1.5) .. controls +(.5,-.2) and +(-.5,1) .. (3,0) ;
							
							\strand[thick] (-3,0) .. controls +(.1,-1) and +(-2,0) .. (0,-1.25) .. controls +(2,0) and +(-.1,-1).. (3,0);
							\strand[thick] (-3,0) .. controls +(1,-.2) and +(-2,0) .. (0,-.2) .. controls +(2,0) and +(-1,-.2).. (3,0);

						\end{knot}
						
						% MERIDIAN 2 HANDLE
						
						\fill[white](1,-1.25) circle[radius = 5pt];
						\draw[thick] (1,-1.25) .. controls + (0,.7) and + (0,.6) .. (.5,-1.1);
						\draw[thick] (1,-1.25) .. controls + (0,-.7) and + (0,-.6) .. (.5,-1.4);
						\draw(.75, -2.2) node {$0$};

						% FRAMING AND LABELS
						
						\filldraw[fill=white,draw=blue] (-.5,.2) circle [radius = .25cm] (-.5,.2) node[scale=.8,color=blue] {$m$};
						\draw (-1.75,.85) node[color=blue]{$\vdots$};
						\draw (-1.7, .7) node[anchor = west, scale=.75,color=blue] {($n$)};
						
						\draw  (-1.75,-.5) node {$\vdots$};
						\draw (-1.75,-.7) node[anchor = west, scale =.75] {($n$)};

						% 1 HANDLE ATTACHING FEET
						
						\fill[color=white] (-3,0) circle [radius = .5cm];
						\fill[color=white] (3,0) circle [radius = .5cm];
						
						\fill[color = pink] (-3.5,0) .. controls +(0,.2775) and +(-.27775,0) .. (-3,.5) .. controls +(.2775,0) and +(0,.2775) .. (-2.5,0);
						\fill[color = green] (3.5,0) .. controls +(0,-.2775) and +(.27775,0) .. (3,-.5) .. controls +(-.2775,0) and +(0,-.2775) .. (2.5,0);
						
						\draw[thick] (-3,0) circle [radius = .5cm];
						\draw[thick] (3,0) circle [radius = .5cm];
						
					\end{scope}
					
					\draw[thick,dashed] (4.8,-2) -- (4.8,2);
					
					\draw (4.25,-2.8) node[scale=2,rotate = -35] {$\Downarrow$}; 
					
					\begin{scope}[xshift = 9cm, scale = .7]
						
						\begin{knot}[
							%draft mode = crossings,
							clip width = 7pt,
							]

							\strand[thick] (-3,0) .. controls +(.1,1) and +(-2,0) .. (0,1.25) .. controls +(2,0) and +(-.5,.5).. (3,0);
							\strand[thick] (-3,0) .. controls +(1,.2) and +(-2,0) .. (0,.2) .. controls +(2,0) and + (-.5,.2) .. (2.25,1.5) .. controls +(.5,-.2) and +(-.5,1) .. (3,0) ;
							
							\strand[thick,color=blue] (-3,0) .. controls +(.1,-1) and +(-2,0) .. (0,-1.25) .. controls +(2,0) and +(-.1,-1).. (3,0);
							\strand[thick,color=blue] (-3,0) .. controls +(1,-.2) and +(-2,0) .. (0,-.2) .. controls +(2,0) and +(-1,-.2).. (3,0);

						\end{knot}
						
						% MERIDIAN 2 HANDLE
						
						\fill[white](1,-1.25) circle[radius = 5pt];
						\draw[thick, color=blue] (1,-1.25) .. controls + (0,.7) and + (0,.6) .. (.5,-1.1);
						\draw[thick, color=blue] (1,-1.25) .. controls + (0,-.7) and + (0,-.6) .. (.5,-1.4);
						\draw(.75, -2.2) node[color=blue] {$0$};

						% FRAMING AND LABELS
						
						\filldraw[fill=white] (-.5,.2) circle [radius = .25cm] (-.5,.2) node[scale=.8] {$m$};
						\draw (-1.75,.85) node{$\vdots$};
						\draw (-1.7, .7) node[anchor = west, scale=.75] {($n$)};
						
						\draw  (-1.75,-.5) node[color=blue] {$\vdots$};
						\draw (-1.75,-.7) node[anchor = west, scale =.75,color=blue] {($n$)};
						
						%  3 AND 4 HANDLES
						
						\draw (5.5,1.5) node[scale=1.3] {$\cup\, 2 h^3,\, 2 h^4$};
						
						% 1 HANDLE ATTACHING FEET
						
						\fill[color=white] (-3,0) circle [radius = .5cm];
						\fill[color=white] (3,0) circle [radius = .5cm];
						
						\fill[color = green] (-3.5,0) .. controls +(0,.2775) and +(-.27775,0) .. (-3,.5) .. controls +(.2775,0) and +(0,.2775) .. (-2.5,0);
						\fill[color = pink] (3.5,0) .. controls +(0,-.2775) and +(.27775,0) .. (3,-.5) .. controls +(-.2775,0) and +(0,-.2775) .. (2.5,0);
						
						\draw[thick] (-3,0) circle [radius = .5cm];
						\draw[thick] (3,0) circle [radius = .5cm];
						
					\end{scope}
					
					\begin{scope}[scale=.7, yshift = -7cm]
						
						%ONE 
						
						\draw(-4,2) node {(2)};
						
						\begin{knot}[
							%draft mode = crossings,
							clip width = 7pt,
							]
                            
                            \strand[thick,color=blue] (-3,0) .. controls +(.1,1) and +(-2,0) .. (0,1.25) .. controls +(2,0) and +(-.5,.5).. (3,0);
							\strand[thick,color=blue] (-3,0) .. controls +(1,.2) and +(-2,0) .. (0,.2) .. controls +(2,0) and + (-.5,.2) .. (2.25,1.5) .. controls +(.5,-.2) and +(-.5,1) .. (3,0) ;

						\end{knot}
						
						% MERIDIAN 2 HANDLE
						
						\fill[white](-.25,1.25) circle[radius = 5pt];
						\draw[thick,color=blue] (.25,1.4) .. controls +(0,.5) and +(0,.5) .. (-.25,1.25) .. controls + (0,-.25) and +(-.1,-.25) .. (.25,1.1);

						% REFLECTED KNOT
						
						\begin{scope}[yscale=-1]
							
							\begin{knot}[
								%draft mode = crossings,
								clip width = 7pt,
								]
								
								\strand[thick] (-3,0) .. controls +(.1,1) and +(-2,0) .. (0,1.25) .. controls +(2,0) and +(-.5,.5).. (3,0);
								\strand[thick] (-3,0) .. controls +(1,.2) and +(-2,0) .. (0,.2) .. controls +(2,0) and + (-.5,.2) .. (2.25,1.5) .. controls +(.5,-.2) and +(-.5,1) .. (3,0) ;

							\end{knot}
							
							% MERIDIAN 2 HANDLE
							
							\fill[white](-.25,1.25) circle[radius = 5pt];
							\draw[thick] (.25,1.4) .. controls +(0,.5) and +(0,.5) .. (-.25,1.25) .. controls + (0,-.25) and +(-.1,-.25) .. (.25,1.1);

						\end{scope}
						
						% ELLIPSES AND FRAMING
						\draw (-1.75,-.55) node[]{$\vdots$};
						\draw (-1.75,.85) node[color=blue]{$\vdots$};
						\draw (0,1.7) node[anchor=south,color=blue] {$0$};
						\draw (0,-1.7) node[anchor=north] {$0$};
						\draw (-2,1.1) node[anchor=south, color=blue] {$m$};
						\draw (-2,-1.1) node[anchor=north] {$-m$};
						\draw (-1.75,.75) node [anchor = west,scale=.8, color=blue] {($n$)};
						\draw (-1.75,-.75) node [anchor = west,scale=.8] {($n$)};
						
						%  3 AND 4 HANDLES
						
						\draw (4.5,1.5) node[scale=1.3] {$\cup\, h^3,\, h^4$};
						
						% 1 HANDLE ATTACHING FEET
						
						\fill[color=white] (-3,0) circle [radius = .5cm];
						\fill[color=white] (3,0) circle [radius = .5cm];
						
						\fill[color = pink] (-3.5,0) .. controls +(0,.2775) and +(-.27775,0) .. (-3,.5) .. controls +(.2775,0) and +(0,.2775) .. (-2.5,0);
						\fill[color = pink] (3.5,0) .. controls +(0,.2775) and +(.27775,0) .. (3,.5) .. controls +(-.2775,0) and +(0,.2775) .. (2.5,0);
						
						\draw[thick] (-3,0) circle [radius = .5cm];
						\draw[thick] (3,0) circle [radius = .5cm];
						
						\draw (6,0) node[scale=2] {$\Rightarrow$};
                        
						\begin{scope}[xshift= 12cm]
							
							\draw(-4,2) node {(3)};
							\begin{knot}[
								%draft mode = crossings,
								clip width = 7pt,
								flip crossing/.list={3,4}
								]

								\strand[thick,color=blue] (-3,0) .. controls +(.25,1) and +(-2,0) .. (0,1.25) .. controls +(2,0) and +(-.5,.5).. (3,0);
								\strand[thick,color=blue] (-3,0) .. controls +(1,.2) and +(-2,0) .. (0,.2) .. controls +(2,0) and + (-.5,-.2) .. (2.25,1.5) .. controls +(.5,.2) and +(-.25,1) .. (3,0) ;
								
								% HANDLESLIDE
								
								\strand[thick] (-.15,-.2) .. controls +(-.1,0) and +(-.1,0) .. (-.15, .1) .. controls +(2.5,0) and +(-.5,-.25) .. (2.2,1.25) .. controls + (.5,.25) and + (0, 0) .. (2.8,0);

								\strand[thick] (-.75,-.2) .. controls +(.1,0) and +(.1,0) .. (-.75,.1) .. controls +(-1,0) and + (-1,-.1) .. (-2.5,0);

								\strand[thick] (-2.73,0) .. controls +(0,1) and +(-1,0) .. (-.35, 1.05);
								
								\strand[thick] (-.05,1.05) ..controls +(1,0) and +(-1,1) .. (2.75,0);

							\end{knot}
							
							% MERIDIAN 2 HANDLE
							
							\fill[white](-.25,1.25) circle[radius = 5pt];
							\draw[thick,color=blue] (.25,1.4) .. controls +(0,.5) and +(0,.5) .. (-.25,1.25) .. controls + (0,-.25) and +(-.2,-.1) .. (0,.8);

							% REFLECTED KNOT
							
							\begin{scope}[yscale=-1]
								
								\begin{knot}[
									%draft mode = crossings,
									clip width = 7pt,
									]

									\strand[thick] (-3,0) .. controls +(.1,1) and +(-2,0) .. (0,1.25) .. controls +(2,0) and +(-.5,.5).. (3,0);

									\strand[thick] (-3,0) .. controls +(1,.25) and +(-2,0) .. (-.75,.2);
									
									\strand[thick] (-.15,.2) .. controls +(2,0) and + (-.5,.2) .. (2.25,1.5) .. controls +(.5,-.2) and +(-.5,1) .. (3,0) ;

								\end{knot}
								
								% MERIDIAN 2 HANDLE
								
								\fill[white](-.25,1.25) circle[radius = 5pt];
								\draw[thick] (.25,1.4) .. controls +(0,.5) and +(0,.5) .. (-.25,1.25) .. controls + (0,-.25) and +(-.1,-.25) .. (.25,1.1);

							\end{scope}

							% ELLIPSES AND FRAMING
							\draw (-1.75,-.55) node[]{$\vdots$};
							\draw (-.6,.75) node[color=blue]{$\vdots$};
							\draw (0,1.7) node[anchor=south,color=blue] {$0$};
							\draw (0,-1.7) node[anchor=north] {$0$};
							\draw (-2.5,.9) node[anchor=south, color=blue] {$m$};
							\draw (-2,-1.1) node[anchor=north] {$0$};
							
							% LINKING
							
							\filldraw[fill=white,thick] (-1.6,.8) -- (-1,.8) -- (-1,1.4) -- (-1.6,1.4) -- (-1.6,.8);
							\draw (-1.3,1.1) node [scale=.7] {$m$};
							
							%  3 AND 4 HANDLES
							
							\draw (4.5,1.5) node[scale=1.3] {$\cup\, h^3,\, h^4$};
							
							% 1 HANDLE ATTACHING FEET
							
							\fill[color=white] (-3,0) circle [radius = .5cm];
							\fill[color=white] (3,0) circle [radius = .5cm];
							
							\fill[color = pink] (-3.5,0) .. controls +(0,.2775) and +(-.27775,0) .. (-3,.5) .. controls +(.2775,0) and +(0,.2775) .. (-2.5,0);
							\fill[color = pink] (3.5,0) .. controls +(0,.2775) and +(.27775,0) .. (3,.5) .. controls +(-.2775,0) and +(0,.2775) .. (2.5,0);
							
							\draw[thick] (-3,0) circle [radius = .5cm];
							\draw[thick] (3,0) circle [radius = .5cm];
							
							\draw (0,-3.5) node[scale=2] {$\Downarrow$};
						\end{scope}
						
						\begin{scope}[xshift = 12cm, yshift = -7cm]
							\draw(-4,2) node {(4)};
							\begin{knot}[
								%draft mode = crossings,
								clip width = 7pt,
								flip crossing/.list={3,4}
								]

								\strand[thick,color=blue] (-3,0) .. controls +(.1,1) and +(-2,0) .. (0,1.25) .. controls +(2,0) and +(-.5,.5).. (3,0);
								\strand[thick,color=blue] (-3,0) .. controls +(1,.2) and +(-2,0) .. (0,.2) .. controls +(2,0) and + (-.5,.2) .. (2.25,1.5) .. controls +(.5,-.2) and +(-.5,1) .. (3,0) ;
								
								% HANDLESLIDE
								
								\strand[thick]  (-.35, -1.25) -- (-1,-1.25) .. controls + (-1,0) and +(-2,0) .. (-1,1.05) -- (-.35, 1.05);
								
								\strand[thick] (-.05,1.05) ..controls +(1,0) and +(2,0) .. (-.05,-1.25);

							\end{knot}
							
							% MERIDIAN 2 HANDLEs
							
							\fill[white](-.25,1.25) circle[radius = 5pt];
							\draw[thick,color=blue] (.25,1.4) .. controls +(0,.5) and +(0,.5) .. (-.25,1.25) .. controls + (0,-.25) and +(-.2,-.1) .. (0,.8);

							\fill[white](-.25,-1.25) circle[radius = 5pt];
							\draw[thick] (.25,-1.4) .. controls +(0,-.5) and +(0,-.5) .. (-.25,-1.25) .. controls + (0,.25) and +(-.1,.25) .. (.25,-1.1);

							% FRAMING AND LABELS

							\draw (-.6,.75) node[color=blue]{$\vdots$};
							\draw (0,1.7) node[anchor=south,color=blue] {$0$};
							\draw (0,-1.7) node[anchor=north] {$0$};
							\draw (-2.5,.9) node[anchor=south, color=blue] {$m$};
							\draw (-2,-.8) node[anchor=north] {$0$};
							
							% LINKING
							
							\filldraw[fill=white,thick] (-1.6,.8) -- (-1,.8) -- (-1,1.4) -- (-1.6,1.4) -- (-1.6,.8);
							\draw (-1.3,1.1) node [scale=.7] {$m$};
							
							%  3 AND 4 HANDLES
							
							\draw (4.5,1.5) node[scale=1.3] {$\cup\, h^3,\, h^4$};
							
							% 1 HANDLE ATTACHING FEET
							
							\fill[color=white] (-3,0) circle [radius = .5cm];
							\fill[color=white] (3,0) circle [radius = .5cm];
							
							\fill[color = pink] (-3.5,0) .. controls +(0,.2775) and +(-.27775,0) .. (-3,.5) .. controls +(.2775,0) and +(0,.2775) .. (-2.5,0);
							\fill[color = pink] (3.5,0) .. controls +(0,.2775) and +(.27775,0) .. (3,.5) .. controls +(-.2775,0) and +(0,.2775) .. (2.5,0);
							
							\draw[thick] (-3,0) circle [radius = .5cm];
							\draw[thick] (3,0) circle [radius = .5cm];
							
							\draw (-6,0) node[scale=2]{$\Leftarrow$};
						\end{scope}
						
						\begin{scope}[yshift = -7cm]
							\draw(-4,2) node {(5)};
							\begin{knot}[
								%draft mode = crossings,
								clip width = 7pt,
								flip crossing/.list={3,8,6}
								]

								\strand[thick,color=blue] (-3,0) .. controls +(.1,1) and +(-2,0) .. (0,1.25) .. controls +(2,0) and +(-.5,.5).. (3,0);
								\strand[thick,color=blue] (-3,0) .. controls +(1,.2) and +(-2,0) .. (0,.2) .. controls +(2,0) and + (-.5,.2) .. (2.25,1.5) .. controls +(.5,-.2) and +(-.5,1) .. (3,0) ;
								
								\strand[thick] (-1,-1) .. controls +(-.5,.5) and +(-.5,-.5) .. (-1,1.5) .. controls +(.5,.5) and +(-.5,.5) .. (.5,1.5) .. controls +(.5,-.5) and +(.5,.5) .. (.5,-1) .. controls +(-.5,-.5) and + (.5,-.5) .. (-1,-1);
								
								\strand[thick,color=blue] (-.5,.6) .. controls +(-.2,0) and +(-.2,0) .. (-.5,2.4) ;

							\end{knot}
							
							% REST OF KNOT
							
							\fill[color=white] (-.35,1.3) circle[radius = .1cm];
							\fill[color=white] (-.35,1.85) circle[radius = .1cm];
							\draw [thick,color=blue] (-.5,.6) .. controls +(.2,0) and +(.2,0) .. (-.5,2.4) ;
							
							\fill[color=white] (.4,1.6) circle[radius = .1cm];
							\draw[thick] (.1,1.65) .. controls +(-.25,-.25) and +(-.25,-.25) .. (.3,1.5) -- (.5,1.7) .. controls +(.25,.25) and +(.25,.25)..  (.3,1.85);

							% FRAMING AND LABELS

							\draw (-1.75,.85) node[color=blue]{$\vdots$};
							\draw (-.5,2.4) node[anchor=south,color=blue] {$0$};
							\draw (0,-1.5) node[anchor=north] {$2m+2$};
							\draw (-2.5,.9) node[anchor=south, color=blue] {$m$};
							\draw (.4,2) node[anchor=south] {$0$};
							
							%  3 AND 4 HANDLES
							
							\draw (4.5,1.5) node[scale=1.3] {$\cup\, h^3,\, h^4$};
							
							% 1 HANDLE ATTACHING FEET
							
							\fill[color=white] (-3,0) circle [radius = .5cm];
							\fill[color=white] (3,0) circle [radius = .5cm];
							
							\fill[color = pink] (-3.5,0) .. controls +(0,.2775) and +(-.27775,0) .. (-3,.5) .. controls +(.2775,0) and +(0,.2775) .. (-2.5,0);
							\fill[color = pink] (3.5,0) .. controls +(0,.2775) and +(.27775,0) .. (3,.5) .. controls +(-.2775,0) and +(0,.2775) .. (2.5,0);
							
							\draw[thick] (-3,0) circle [radius = .5cm];
							\draw[thick] (3,0) circle [radius = .5cm];
							
							\draw (0,-3.5) node[scale=2, rotate = -70] {$\Rightarrow$};
						\end{scope}
						
						\begin{scope} [xshift = 8cm, yshift = -13cm]
							\draw(-9,0) node {(6)};
							\begin{knot}[
								%draft mode = crossings,
								clip width = 7pt,
								flip crossing/.list={3,8,6,4}
								]

								\strand[thick,color=blue] (-3,0) .. controls +(.1,1) and +(-2,0) .. (0,1.25) .. controls +(2,0) and +(-.5,.5).. (3,0);
								\strand[thick,color=blue] (-3,0) .. controls +(1,.2) and +(-2,0) .. (0,.2) .. controls +(2,0) and + (-.5,.2) .. (2.25,1.5) .. controls +(.5,-.2) and +(-.5,1) .. (3,0) ;
								
								\strand[thick,color=blue] (-.5,.6) .. controls +(-.5,0) and +(-.5,0) .. (-.5,2) .. controls +(.5,0) and +(.5,0) .. (-.5,.6) ;
								
								\strand[thick] (-5.5,0) circle[radius=1cm];
								\strand[thick] (-6.75,0) circle[radius=1cm];

							\end{knot}
							
							% MERIDIAN 2 HANDLES

							% FRAMING AND LABELS

							\draw (-1.75,.85) node[color=blue]{$\vdots$};
							\draw (-1.75,.75) node [anchor = west,scale=.8, color=blue] {($n$)};
							\draw (-.5,2) node[anchor=south,color=blue] {$0$};
							\draw (-2.5,.9) node[anchor=south, color=blue] {$m$};
							\draw (-5,-1.1) node[anchor=north,scale=.8] {$2m+2$};
							\draw (-7,1) node [anchor = south] {$0$};
							
							%  3 AND 4 HANDLES
							
							\draw (4.5,1.5) node[scale=1.3] {$\cup\, h^3,\, h^4$};
							
							% 1 HANDLE ATTACHING FEET
							
							\fill[color=white] (-3,0) circle [radius = .5cm];
							\fill[color=white] (3,0) circle [radius = .5cm];
							
							\fill[color = pink] (-3.5,0) .. controls +(0,.2775) and +(-.27775,0) .. (-3,.5) .. controls +(.2775,0) and +(0,.2775) .. (-2.5,0);
							\fill[color = pink] (3.5,0) .. controls +(0,.2775) and +(.27775,0) .. (3,.5) .. controls +(-.2775,0) and +(0,.2775) .. (2.5,0);
							
							\draw[thick] (-3,0) circle [radius = .5cm];
							\draw[thick] (3,0) circle [radius = .5cm];
						\end{scope}
						
					\end{scope}
					
				\end{tikzpicture}
				
			}

			\caption{In (1), all $n$ strands run parallel to each other except the stands with $m$--framing. In (3) and (4), the boxes labeled $m$ indicate a linking of $m$ between the two strands going through the box.}
			\label{fig:fold_dble_cover}
		\end{figure}
	\end{proof}

 	\subsection{Nonorientable genus two SBLFs with Lefschetz singularities.}
    
	Recall from Remark \ref{rk:relminimal} that any nonorientable SBLF is related to a relatively minimal SBLF (i.e. one where no Lefschetz vanishing cycle bounds a disk or a Möbius band) through a sequence of orientable and nonorientable blow-downs. This information can be used to classify nonorientable genus two SBLFs with indefinite fold and Lefschetz critical points. Such SBLFs are obtained by first adding a round handle cobordism $X_0$ to a nonorientable genus two Lefschetz fibration $f_+\colon X_+\to D^2$ which we assume is relatively minimal. Relatively minimal genus two Lefschetz fibrations, as already discussed in \cite[Section 4]{Miller-Ozbagci}, have a rather simple structure, since the only Dehn twist generator of $\Mod(Kb)$ is $t_\gamma$, where $\gamma$ is the unique two-sided non-separating curve on $Kb$, up to isotopy. So, the monodromy factorization of $f_+$ is $t_\gamma^k$ for some $k>0$, where each $t_\gamma$ factor can be taken to be either right-- or left--handed. Moreover, the attaching circle of the $2$--handle of the round $2$--handle coming from the indefinite fold would also be isotopic to $\gamma$. Therefore, the $2$--handles of the Lefschetz singularities can be slid off the $1$--handle(s) using this zero framed $2$--handle. After this handle slide, these unknotted $2$--handles with framing $\pm 1$ can be regarded as nullhomotopic vanishing cycles (so the total space is diffeomorphic to a Lefschetz fibration that is not relatively minimal), and can be blown down to get a nonorientable genus two SBLF with only an indefinite fold singularity. (Note that we can blow down these no matter how the lower-genus side $X_-$ is attached to $X_+ \cup X_0$.) As these were classified in the previous section, we conclude the following.
	
	\begin{proposition}
		Let $f\colon X\to S^2$ be a nonorientable genus two SBLF with $Z_f \neq \emptyset$ and $C_f \neq \emptyset$. Then $X$ can be obtained from one of the manifolds listed in Proposition~\ref{prop:fold_only_g2} through a sequence of orientable and nonorientable blow-ups.
	\end{proposition}
	
	Next we consider the case where $f\colon X\to S^2$ is a nonorientable genus two SBLF with $Z_f = \emptyset$, namely a Lefschetz fibration with $Kb$ regular fibers. Most of this case has already been flushed out in \cite{Miller-Ozbagci}, and we provide their Kirby diagrams below.  As before, we assume that $f$ is relatively minimal. The only positive factorizations of the identity in $\Mod(Kb)$ given by Dehn twists are $(t_\gamma)^{2m}=1$. Let $N_m$ be the genus two Lefschetz fibration over a disk whose positive factorization is $(t_\gamma)^{2m}$. We use the notation of \cite{Miller-Ozbagci}, and let $X(m)$ denote the Lefschetz fibration obtained by gluing a copy of $D^2\times Kb$ to $N_m$ via the identity map, shown on the left in Figure \ref{fig:Xm}. Indeed, $X(m)$ is the fiber sum of $m$ copies of $X(1)$. More generally, any relatively minimal genus two Lefschetz fibration is obtained from $X(m)$ by removing a regular fiber and re-gluing it via a fiber-preserving diffeomorphism of $S^1\times Kb$. These diffeomorphisms have been classified in Lemma \ref{lem:diff_KB_gen}, and this determines all relatively minimal nonorientable genus two Lefschetz fibrations up to diffeomorphism. We summarize these findings as follows:

	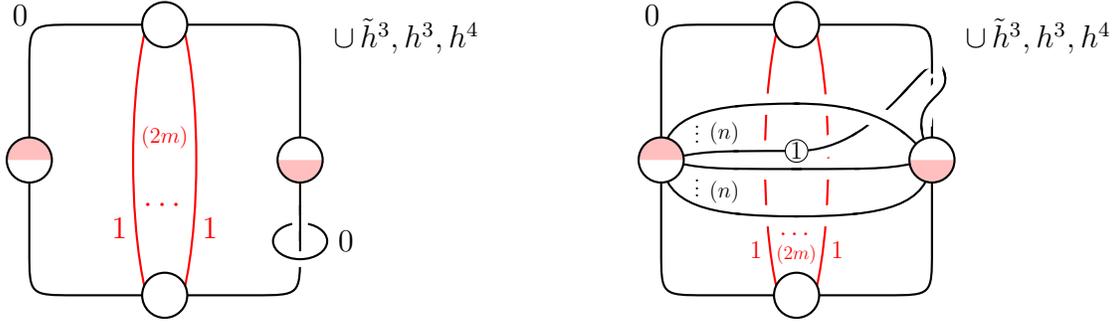
\begin{figure}
		\centering

		\begin{tikzpicture}
	\begin{scope}[scale=.6]
		
		\begin{knot}[
			%draft mode = crossings,
			clip width = 10pt,
			flip crossing/.list={2}
			]
			\strand[thick] (.5,3) -- (2,3) .. controls +(1,0) and +(0,1) .. (3,2) -- (3,.5);
			\strand[thick, rotate = 90] (.5,3) -- (2,3) .. controls +(1,0) and +(0,1) .. (3,2) -- (3,.5);
			\strand[thick, rotate = 180](.5,3) -- (2,3) .. controls +(1,0) and +(0,1) .. (3,2) -- (3,.5);
			\strand[thick](3,-.5) -- (3,-2) .. controls +(0,-1) and +(1,0) .. (2,-3) -- (.5,-3);
			
			\strand[thick,color = red] (-.4,3) .. controls +(0,0) and +(0,2) .. (-.7,0) .. controls +(0,-2) and +(0,0) .. (-.4,-3);
			\strand[thick,color = red] (.4,3) .. controls +(0,0) and +(0,2) .. (.7,0) .. controls +(0,-2) and +(0,0) .. (.4,-3);

		\end{knot}
		
		\fill[color=white] (3,-2.2) circle [radius=5pt];
		
		\draw[thick] (3,-1.8) ellipse (.6cm and .4cm);
		
		\fill[color=white] (3,-1.4) circle [radius=5pt];
		
		\draw[thick] (3,-1.8) -- (3,-1);
		
		% 1 HANDLE ATTACHING FEET
		
		\fill[color = pink] (-3.5,0) .. controls +(0,.2775) and +(-.27775,0) .. (-3,.5) .. controls +(.2775,0) and +(0,.2775) .. (-2.5,0);
		\fill[color = pink] (3.5,0) .. controls +(0,-.2775) and +(.27775,0) .. (3,-.5) .. controls +(-.2775,0) and +(0,-.2775) .. (2.5,0);
		
		\draw[thick] (-3,0) circle [radius = .5cm];
		\draw[thick] (3,0) circle [radius = .5cm];
		\filldraw[thick, fill=white] (0,3) circle [radius = .5cm];
		\filldraw[thick, fill=white] (0,-3) circle [radius = .5cm];
		
		\draw (-1,-1.5) node[color=red]{$1$};
		\draw (1, -1.5) node[color=red]{$1$};
		\draw (0,.5) node[color=red,scale=.7] {$(2m)$};
		\draw (0,-1) node[color = red] {$\cdots$};

		\draw (3.5,2.75) node[anchor=west] {$\cup\, \tilde h^3,h^3,h^4$};
		\draw (-3.2,3.2) node{$0$};
		\draw (3.6,-1.8) node[anchor=west] {$0$};
		
		\begin{scope}[xshift = 14cm]
			\begin{knot}[
				%draft mode = crossings,
				clip width = 10pt,
				flip crossing/.list={2,3,4,5,6,7,8,9,10}
				]
				\strand[thick] (.5,3) -- (2,3) .. controls +(1,0) and +(0,1) .. (3,2) -- (3,.5);
				\strand[thick, rotate = 90] (.5,3) -- (2,3) .. controls +(1,0) and +(0,1) .. (3,2) -- (3,.5);
				\strand[thick, rotate = 180](.5,3) -- (2,3) .. controls +(1,0) and +(0,1) .. (3,2) -- (3,.5);
				\strand[thick](3,-.5) -- (3,-2) .. controls +(0,-1) and +(1,0) .. (2,-3) -- (.5,-3);
				
				\strand[thick,color = red] (-.4,3) .. controls +(0,0) and +(0,2) .. (-.7,0) .. controls +(0,-2) and +(0,0) .. (-.4,-3);
				\strand[thick,color = red] (.4,3) .. controls +(0,0) and +(0,2) .. (.7,0) .. controls +(0,-2) and +(0,0) .. (.4,-3);
				
				% \PHI 2-HANDLE ATTACHING CIRCLE
				
				\strand[thick] (-3,0) .. controls +(.1,1) and +(-2,0) .. (0,1.25) .. controls +(2,0) and +(-.5,.5).. (3,0);
				\strand[thick] (-3,0) .. controls +(1,.2) and +(-2,0) .. (0,.2) .. controls +(2,0) and + (-.3,.75) .. (3.25,2) .. controls +(.3,-.75) and +(-.75,1.5) .. (3,0) ;
				
				\strand[thick] (-3,0) .. controls +(.1,-1) and +(-2,0) .. (0,-1.25) .. controls +(2,0) and +(-.1,-1).. (3,0);
				\strand[thick] (-3,0) .. controls +(1,-.2) and +(-2,0) .. (0,-.2) .. controls +(2,0) and +(-1,-.2).. (3,0);
				
			\end{knot}
			
			% 1 HANDLE ATTACHING FEET
			
			\fill[color=white] (-3,0) circle [radius=.5cm];
			\fill[color=white] (3,0) circle [radius=.5cm];
			
			\fill[color = pink] (-3.5,0) .. controls +(0,.2775) and +(-.27775,0) .. (-3,.5) .. controls +(.2775,0) and +(0,.2775) .. (-2.5,0);
			\fill[color = pink] (3.5,0) .. controls +(0,-.2775) and +(.27775,0) .. (3,-.5) .. controls +(-.2775,0) and +(0,-.2775) .. (2.5,0);
			
			\draw[thick] (-3,0) circle [radius = .5cm];
			\draw[thick] (3,0) circle [radius = .5cm];
			\filldraw[thick, fill=white] (0,3) circle [radius = .5cm];
			\filldraw[thick, fill=white] (0,-3) circle [radius = .5cm];
			
			% FRAMING AND LABELS
			
			\draw (-.9,-2) node[color=red,scale=.8]{$1$};
			\draw (.9, -2) node[color=red,scale=.8]{$1$};
			\draw (0,-2.1) node[color=red,scale=.6] {$(2m)$};
			\draw (0,-1.65) node[color = red,scale=.8] {$\cdots$};
			
			\draw (3.5,2.75) node[anchor=west] {$\cup\, \tilde h^3,h^3,h^4$};
			\draw (-3.2,3.2) node{$0$};
			\filldraw[fill=white] (0,.2) circle [radius = .25cm] (0,.2) node[scale=.8] {$1$};
			
			\draw (-2.2,.7) node [scale = .7] {$\vdots$};
			\draw (-2.2,-.5) node [scale = .7] {$\vdots$};
			\draw (-1.6, .6) node [scale=.7] {$(n)$};
			\draw (-1.6, -.7) node [scale=.7] {$(n)$};
		\end{scope}
		
	\end{scope}

\end{tikzpicture}

		\caption{Left: The Kirby diagram for the relatively minimal nonorientable genus two Lefschetz fibration $X(m)$ with a section. Right: The  diagram for the more general Lefschetz fibration $M_{m,n}$.}
		\label{fig:Xm}
	\end{figure}
	
	\begin{proposition}\label{prop:KB_Lefschetz_sing}
		Let $M_{m,n}$ denote the genus two Lefschetz fibration \[
		M_{m,n} \cong (X(m) \setminus \nu(F)) \cup_{\phi^n} (D^2\times Kb),
		\] shown on the right in Figure \ref{fig:Xm}. The total space of any nonorientable genus two Lefschetz fibration is related to exactly one copy of $M_{m,n}$ through a sequence of orientable and nonorientable blow-downs.
	\end{proposition}

    Here $M_{m,n}$ has $\pi_1=\Z_{2n}$ and Euler characteristic $m$. A Kirby diagram for $M_{m,n}$ is given by taking the Kirby diagram for $X(m)$, and replacing the section $2$--handle with a $2$--handle which is attached along the curve corresponding to $\phi^n$ from Lemma \ref{lem:diff_KB_gen}. Theorem \ref{thm:g2_classification} now follows from Propositions \ref{prop:KB_bundle_handlebody}, \ref{prop:fold_only_g2}, and \ref{prop:KB_Lefschetz_sing}.

\begin{remark} \label{rk:genus3}
        Using the conversion formulae in Remark \ref{rk:trisection_to}, we conclude that the algorithm given in the proof of Theorem \ref{thm:STexistence} produces $(4,2)$ simplified trisections for each $K_n$ and $(3,1)$ simplified trisections for each $N_n$ and $N'_n$. 
        By Lemma \ref{lem:genus_bound}, these are all of minimal genus. The latter examples extend the conjectural list of genus--$3$ trisections with nonorientable examples;~see \cite{meier, Baykur-Saeki:PNAS}. Akin to how the orientable examples $L_n$ and $L'_n$ are obtained from $S^1 \x S^3$ by surgering a curve representing $n$ times the generator of $\pi_1(S^1 \x S^3)$ \cite{Pao, meier}, the nonorientable $4$--manifolds $N_n$ and $N'_n$ can be obtained by surgering a curve representing $2n$ times the generator of $\pi_1(S^1 \twprod S^3) \cong \Z$.
\end{remark}

\section{A surgery classification for nonorientable $4$--manifolds}

In this final section, we bring torus surgeries to the study of nonorientable $4$--manifolds. We recall the definition first.
Let $T$ be a torus embedded in a $4$--manifold $X$ with a trivial normal bundle, so it has a regular neighborhood $\nu T$ diffeomorphic to $T^2 \x D^2$ via some framing $\tau$. A \emph{torus surgery} along $T$ constructs a new $4$-manifold $X' = (X \setminus \nu T) \, \cup_{\phi} T^2 \x D^2$ using a gluing diffeomorphism $\phi \colon \partial (T^2 \x D^2) \to \partial (X \setminus \nu T)$. 

The gluing data $\phi$ can be discretized for a more effective study. First, note that $X'$ is built from $X \setminus \nu T$ by attaching a $2$--handle, two $3$--handles, and a $4$--handle---that is, the handle decomposition of $T^2 \x D^2$ turned upside down. Since 
attaching of $3$-- and $4$--handles result in diffeomorphic $4$--manifolds \cite{laudenbach-poenaru, Miller-Naylor}, $X'$ is determined by the attaching circle of the $2$--handle, namely, $C:=\phi(\{pt \} \x \partial D^2)$. The latter is further determined by the homology class $[C] = p [\mu] + q [\lambda_\tau]$ in $ H_1(\partial \nu T)$, where $\mu$ is a meridian of the torus $T$ and $\lambda_\tau$ is a push-off of a primitive curve $\lambda \subset T$ under $\tau$. Thus, the torus surgery can be encoded as $X' = X(T, \tau, \lambda, p/q)$, where $\lambda$ is called the \emph{surgery curve} and $p/q \in \Q \cup \{ \infty \}$ the \emph{surgery coefficient}. We call it an \emph{integral} torus surgery when $q = \pm 1$ \cite{Baykur_Sunukjian_2012}.

It is a theorem of Iwase that every closed, oriented $4$--manifold can be obtained by a torus surgery along a link of tori in a connected sum of copies of $\CP$, $\CPb$ and $S^1 \x S^3$ \cite{Iwase}. In preparation to prove a version of this result for nonorientable $4$--manifolds, namely Theorem~\ref{thm:surgery}, we first note a simple observation exploited in \cite{Iwase, Baykur_Sunukjian_2012}, a proof of which can be found in \cite{LARSON_2016}. Let $u$ be the unknot in $D^3$. Decomposing $D^4$ as $S^1\times D^3\cup_{S^1\times S^2} (S^2\times D^2 \setminus D^4)$, the unknotted torus $T$ in $D^4$ can be regarded as $S^1\times u\subset S^1\times D^3\subset D^4$. Let $\tau:\nu T\to S^1\times \nu(u)$ denote this identification.

\begin{lemma} \label{lem:Larson}
Let $T \subset D^4$ be an unknotted torus with framing $\tau$ as described above. Let $\lambda_i \subset T$ be a $(1,i)$--curve, for $i=0,1$. Then, 
\[
D^4(T, \tau, \lambda_i, 0) \cong \begin{cases}
	(S^2\times S^2 \# S^1\times S^3) \setminus \text{Int}(D^4) & \text{for } i=0\\
	(\mathbb{C}\mathbb{P}^2\#\overline{\mathbb{C}\mathbb{P}^2} \# S^1\times S^3)\setminus \text{Int}(D^4) & \text{for } i=1.\\
\end{cases}
\]
\end{lemma}

\begin{proof}[Proof of Theorem~\ref{thm:surgery}]
Let $\{\tilde \gamma, \gamma_1,\cdots \gamma_n\}$ be a disjoint collection of curves in the closed, connected, nonorientable $4$--manifold $X$, isotopic (individually) to a geometric basis for $\pi_1(X)$, where $\tilde{\gamma}$ is orientation-reversing and each $\gamma_i$ is orientation-preserving. Let $T_i:=\gamma_i \x u_i$ be an unknotted torus in $\nu(\gamma_i)\cong_{\tau_i} S^1\times D^3$ (with framing $\tau_i$), where $u_i$ corresponds to an unknot in $D^3$, for each $i=1,2, \ldots, n$. Thus $L:= \sqcup_{i=1}^n T_i$ is a link of tori in $X$ with trivial normal bundles. For $\mu_i$ a meridian of $T_i$, consider the geometric basis $\{\gamma_i, u_i, \mu_i\}$ for $\pi_1(\partial(\nu(T_i)))$.

As shown in Figure \ref{fig:t2_surg}, for $(T_i, \tau_i, \gamma_i)$ as chosen above, the surgered manifold $\nu(\gamma_i) (T_i, \tau_i, \gamma_i, 0)$ is diffeomorphic to $S^2\times D^2\#S^1\times S^3$; cf. \cite[Theorem A.2]{Iwase}. In the decomposition $X=(X\setminus \nu(\gamma_i))\cup_{S^1\times S^2} S^1\times D^3$, the inclusion $S^1\times S^2 \hookrightarrow S^1\times D^3$ induces an isomorphism on $\pi_1$ before the torus surgery, whereas its inclusion to the new component in the decomposition $X' =(X\setminus \nu(\gamma_i))\cup_{S^1\times S^2} (S^2\times D^2\#S^1\times S^3)$ is trivial. It follows from Seifert--Van Kampen that after the torus surgery, $\gamma_i$ is nullhomotopic and there is a new $\pi_1$ generator coming from the  new $S^1\times S^3$ summand.
\begin{figure}[ht!]
	\centering
	\resizebox{!}{1.5in}{
		\begin{tikzpicture}
	
	\begin{knot}[
		%draft mode=crossings,
		clip width = 7pt,
		flip crossing/.list={3,4,8,10}, 
		]
		\strand[ultra thick] (1,0) circle[radius=1.5cm] (2.8,-.5) node[scale=1.2] {$0$};
		\strand[ultra thick] (-1,0) circle[radius=1.5cm];
		\strand[ultra thick] (0,{sqrt(3)}) circle[radius=1.5cm];
		%\strand[ultra thick] (-.8,3) circle [radius = .5cm];
		\strand[ultra thick] (0, {sqrt(3)+2}) .. controls +(.4,0) and +(.4,0) .. (0,{sqrt(3)+1}) .. controls +(-.4,0) and +(-.4,0) .. (0, {sqrt(3)+2}) (0,4.1) node[scale=1.2] {$0$};
		
		\strand[thick, color=purple, dashed] (-1,-1.5) circle [radius = .5cm];
		
	\end{knot}
	
	\fill (-2.04,-1.05) circle [radius = 4pt];
	\fill (-1.35,2.4) circle [radius = 4pt];
	\draw (2.5,3) node[scale=1.5] {$\cup h^3$};
	
	\draw [-{Stealth[length=5mm, width=4mm]}, ultra thick] (3.2,1) -- (5,1);
	
	\begin{scope}[xshift = 8cm]
		
		\begin{knot}[
			%draft mode=crossings,
			clip width = 7pt,
			flip crossing/.list={3,4,8,10}, 
			]
			\strand[ultra thick] (1,0) circle[radius=1.5cm];
			\strand[ultra thick] (-1,0) circle[radius=1.5cm] (-2.7,-.75) node[scale=1.2] {$0$};
			\strand[ultra thick] (0,{sqrt(3)}) circle[radius=1.5cm];
			%\strand[ultra thick] (-.8,3) circle [radius = .5cm];
			\strand[ultra thick] (0, {sqrt(3)+2}) .. controls +(.4,0) and +(.4,0) .. (0,{sqrt(3)+1}) .. controls +(-.4,0) and +(-.4,0) .. (0, {sqrt(3)+2}) (0,4.1) node[scale=1.2] {$0$};
			
			\strand[thick, color=purple, dashed] (-1,-1.5) circle [radius = .5cm];
			
		\end{knot}
		
		\fill (2.05,-1.05) circle [radius = 4pt];
		\fill (-1.35,2.4) circle [radius = 4pt];
		\draw (2.5,3) node[scale=1.5] {$\cup h^3$};
	\end{scope}
	
    \end{tikzpicture}
	}
	\caption{The left shows $\nu(\gamma_i)\cong S^1\times D^3$, with $\nu(T_i)$ given by the Borromean sublink. The $(T_i,\tau_i,\gamma_i,0)$ surgery permutes the link components as shown on the right. The dashed meridian is \emph{not} part of the handlebody, but rather conveys where the boundary is. Since the surgery fixes boundary, the handle it links does not cancel. Standard Kirby moves identify the new total space as $S^2\times D^2\# S^1\times S^3$.}
	\label{fig:t2_surg}
\end{figure}
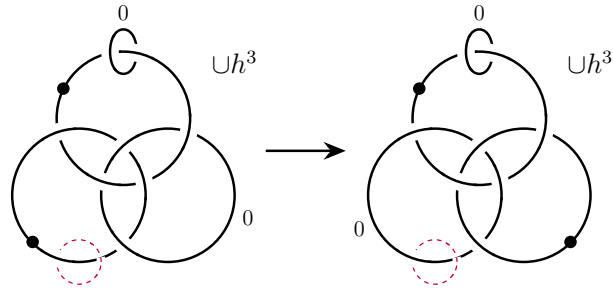

After performing this surgery along the link of tori $L \subset X$, we obtain $Y\# n(S^1\times S^3)\# Z$, where $Y$ is nonorientable with $\pi_1(Y)$ generated by (image of) the curve $\tilde{\gamma}$, and $Z$ is simply connected. By \cite[Theorem A.2]{Iwase} (or \cite[Corollary 14]{Baykur_Sunukjian_2012}), $Z$ can be obtained from a connect sum of copies of $\mathbb{C}\mathbb{P}^2$, $\overline{\mathbb{C}\mathbb{P}^2}$, and $S^1\times S^3$ by a surgery along a link of tori. If needed, by performing an additional torus surgery in $D^4\subset Z$ as in Lemma~\ref{lem:Larson}, we can ensure that there is  at least one connect summand of $\mathbb{C}\mathbb{P}^2\#\overline{\mathbb{C}\mathbb{P}^2}$ in $Z$. Therefore, $Y\# n(S^1\times S^3)\# Z$ yields $Y\# m_1(S^1\times S^3)\# m_2\mathbb{C}\mathbb{P}^2\#m_3\overline{\mathbb{C}\mathbb{P}^2}$ through additional torus surgeries, with $m_2+m_3>0$. Since $Y$ is nonorientable, we simply write the latter as $Y\# m_1(S^1\times S^3)\# (m_2+m_3)\mathbb{C}\mathbb{P}^2$.

Next we perform torus surgery on the nonorientable $4$--manifold $Y$. Let $c$ be a curve in $Y$ which represents $\tilde{\gamma}^2$ in $\pi_1(Y)$. Note that $\nu(c)\cong_{\tau'} S^1\times D^3$, the untwisted $D^3$-bundle over $S^1$. We can thus perform the same type of $0$--surgery as before, along an unknotted torus $T \subset S^1\times D^3$ with framing $\tau'$ and surgery curve $c \subset T$ to obtain a new manifold $Y'$. By Seifert--Van Kampen, we now have $\pi_1(Y') \cong \Z_2$.

We conclude that a  surgery along a link of tori in $X$ yields a nonorientable  $4$--manifold $Y' \# (m_2+m_3)\mathbb{C}\mathbb{P}^2\# m_1(S^1\times S^3)$, where $Y$ is nonorientable with order two $\pi_1$ and $m_2+m_3\geq 1$. By the topological classification of
nonorientable 4-manifolds with $\pi_1=\Z_2$ \cite[Theorem 3]{HKT}, there is a handful of possible homeomorphism types for $Y'\#(m_2+m_3)\mathbb{C}\mathbb{P}^2$. The smooth manifold we have in hand has vanishing Kirby-Siebenmann invariant and $\omega_2$--type (i)---since the universal double cover of $Y' \# (m_2+m_3)\mathbb{C}\mathbb{P}^2$ has odd intersection form.  With this, we deduce that $Y'\#(m_2+m_3)\mathbb{C}\mathbb{P}^2$ is homeomorphic to either $S^2\times \mathbb{R}\mathbb{P}^2\#k \mathbb{C}\mathbb{P}^2$ or $\mathbb{R}\mathbb{P}^4\# k \mathbb{C}\mathbb{P}^2$ for some $k\geq m_2+m_3$. 
As Gompf showed in \cite{GOMPF1984}, this statement can be upgraded to a diffeomorphism for sufficiently large $k$.  By performing either of the two surgeries in Lemma~\ref{lem:Larson} on a collection of disjoint $4$-balls in $Y'$, we can add arbitrarily many copies of $S^1\times S^3\#\mathbb{C}\mathbb{P}^2\#\overline{\mathbb{C}\mathbb{P}^2}$ to our connected sum. This eventually gives a $4$-manifold diffeomorphic to a connected sum of $S^2\times \mathbb{R}\mathbb{P}^2$ or $\mathbb{R}\mathbb{P}^4$ with several copies of $\CP$ and $S^1 \x S^3$.

All torus surgeries we have performed so far can be assumed to occur along disjoint tori in the original $X$; the surgeries that change $\pi_1(X)$ are disjoint by assumption, the surgeries that transform $Z$ into copies of $S^1\times S^3$, $\mathbb{C}\mathbb{P}^2$, and $\overline{\mathbb{C}\mathbb{P}^2}$ occur along disjoint tori away from $\{\tilde \gamma, \gamma_1,\ldots, \gamma_n\}$ \cite[Corollary 14]{Baykur_Sunukjian_2012}, and the surgeries from Lemma~\ref{lem:Larson} occur in other $4$-balls. Therefore we arrive at a $4$--manifold 
$S^2\times \mathbb{R}\mathbb{P}^2\#a \mathbb{C}\mathbb{P}^2\# b(S^1\times S^3)$ or $\mathbb{R}\mathbb{P}^4\# a \mathbb{C}\mathbb{P}^2\#b(S^1\times S^3)$, depending on the unoriented cobordism type of the nonorientable $4$--manifold $X$, through surgery along a link of disjoint tori. Since torus surgery is reversible, the statement follows.	
\end{proof}

In dimension three, it is a theorem of Lickorish that every closed $3$--manifold can be obtained by a Dehn surgery along $a \, S^1 \twprod S^2 \# b \mathbb{RP}^2 \x S^1$, for some $a, b \in \Z_{\geq 0}$ with $a+b \leq 1$, where one can take $a=b=0$ (and the connected sum is understood to be just $S^3$) if $Y$ is orientable \cite{lickorish:nonorientable}. We get the following four dimensional analog of this result:

\begin{corollary} \label{cor:4Dlickorish-wallace}
Any closed, connected, smooth $4$-manifold $X$ is obtained from a universal $4$-manifold\[
Z(a,b,c,d,e)=a(S^2\times \mathbb{R}\mathbb{P}^2)\#b\mathbb{R}\mathbb{P}^4\# c(S^1\times S^3)\# d\mathbb{C}\mathbb{P}^2\# e\overline{\mathbb{C}\mathbb{P}^2}
\] through surgery along a link of tori in $Z(a,b,c,d,e)$, where\begin{enumerate}[(i)]
	\item If $X$ is oriented, $a=b=0$.
	\item If $X$ is nonorientable, $a+b=1$  and we can take $e=0$.
\end{enumerate}
\end{corollary}

By \cite{Baykur_Sunukjian_2012}, this translates to the statement that there exists a cobordism $W$ from any given $X$ to a standard $Z(a,b,c,d,e)$ which is made out of $5$-dimensional round two handles. 

As in the orientable case \cite{Baykur_Sunukjian_2012}, a topic of interest is the minimum number of torus surgeries (one is enough?) needed to pass between two homeomorphic but not diffeomorphic nonorientable $4$--manifolds (cf. \cite{bais-torres}).

\vspace{0.3in}
\bibliography{refs-new.bib}

\begin{thebibliography}{10}

\bibitem{Akbulut1984}
Selman Akbulut.
\newblock A fake {$4$}-manifold.
\newblock In {\em Four-manifold theory ({D}urham, {N}.{H}., 1982)}, volume~35
  of {\em Contemp. Math.}, pages 75--141. Amer. Math. Soc., Providence, RI,
  1984.

\bibitem{ADK}
Denis Auroux, Simon~K. Donaldson, and Ludmil Katzarkov.
\newblock Singular {L}efschetz pencils.
\newblock {\em Geom. Topol.}, 9:1043--1114, 2005.

\bibitem{bais-torres}
Valentina Bais and Rafael Torres.
\newblock Smooth structures on nonorientable 4-manifolds via twisting
  operations.
\newblock {\em Bull. Lond. Math. Soc.}, 57(6):1768--1790, 2025.

\bibitem{Baykur:IMRN}
R.~\.{I}nan\c{c} Baykur.
\newblock Existence of broken {L}efschetz fibrations.
\newblock {\em Int. Math. Res. Not. IMRN}, pages Art. ID rnn 101, 15, 2008.

\bibitem{baykurPJM}
R.~\.{I}nan\c{c} Baykur.
\newblock Topology of broken lefschetz fibrations and near-symplectic
  4-manifolds.
\newblock {\em Pacific Journal of Mathematics}, 240(2), 2009.

\bibitem{Baykur-Hayano:BLFandMCG}
R.~\.{I}nan\c{c} Baykur and Kenta Hayano.
\newblock Broken {L}efschetz fibrations and mapping class groups.
\newblock In {\em Interactions between low-dimensional topology and mapping
  class groups}, volume~19 of {\em Geom. Topol. Monogr.}, pages 269--290. Geom.
  Topol. Publ., Coventry, 2015.

\bibitem{Baykur_Kamada}
R.~\.{I}nan\c{c} Baykur and Seiichi Kamada.
\newblock Classification of broken lefschetz fibrations with small fiber
  genera, 2010.

\bibitem{Baykur-Saeki:PNAS}
R.~\.{I}nan\c{c} Baykur and Osamu Saeki.
\newblock Simplified broken {L}efschetz fibrations and trisections of
  4-manifolds.
\newblock {\em Proc. Natl. Acad. Sci. USA}, 115(43):10894--10900, 2018.

\bibitem{Baykur-Saeki:TAMS}
R.~\.{I}nan\c{c} Baykur and Osamu Saeki.
\newblock Simplifying indefinite fibrations on 4-manifolds.
\newblock {\em Trans. Amer. Math. Soc.}, 376(5):3011--3062, 2023.

\bibitem{Baykur_Sunukjian_2012}
R.~İnanç Baykur and Nathan Sunukjian.
\newblock Round handles, logarithmic transforms and smooth 4-manifolds.
\newblock {\em Journal of Topology}, 6(1):49–63, November 2012.

\bibitem{cesardesa:thesis}
E.~C\'esar de~S\'a.
\newblock {\em Automorphisms of 3-manifolds and representations of
  4-manifolds}.
\newblock Ph.d. thesis, University of Warwick, 1977.

\bibitem{Earle_Eells}
Clifford~J. Earle and James Eells.
\newblock {A fibre bundle description of Teichmüller theory}.
\newblock {\em Journal of Differential Geometry}, 3(1-2):19 -- 43, 1969.

\bibitem{Farb-Margalit}
Benson Farb and Dan Margalit.
\newblock {\em A primer on mapping class groups}, volume~49 of {\em Princeton
  Mathematical Series}.
\newblock Princeton University Press, Princeton, NJ, 2012.

\bibitem{fintushel_stern_blowdown}
Ronald Fintushel and Ronald~J. Stern.
\newblock {Rational blowdowns of smooth 4-manifolds}.
\newblock {\em Journal of Differential Geometry}, 46(2):181 -- 235, 1997.

\bibitem{teichner-etal}
Michael Freedman, Daniel Kasprowski, Matthias Kreck, Alan~W. Reid, and Peter
  Teichner.
\newblock Immersions of punctured 4-manifolds: with applications to quantum
  cellular automata.
\newblock \url{https://arxiv.org/abs/2208.03064}.

\bibitem{Gay-Kirby:trisection}
David Gay and Robion Kirby.
\newblock Trisecting 4-manifolds.
\newblock {\em Geom. Topol.}, 20(6):3097--3132, 2016.

\bibitem{Gay-Kirby:BLF}
David~T. Gay and Robion Kirby.
\newblock Constructing {L}efschetz-type fibrations on four-manifolds.
\newblock {\em Geom. Topol.}, 11:2075--2115, 2007.

\bibitem{GOMPF1984}
Robert~E. Gompf.
\newblock Stable diffeomorphism of compact 4-manifolds.
\newblock {\em Topology and its Applications}, 18(2):115--120, 1984.

\bibitem{Gompf-Stipsicz}
Robert~E. Gompf and Andr\'{a}s~I. Stipsicz.
\newblock {\em {$4$}-manifolds and {K}irby calculus}, volume~20 of {\em
  Graduate Studies in Mathematics}.
\newblock American Mathematical Society, Providence, RI, 1999.

\bibitem{HKT}
Ian Hambleton, Matthias Kreck, and Peter Teichner.
\newblock Nonorientable 4-manifolds with fundamental group of order 2.
\newblock {\em Transactions of the American Mathematical Society},
  344(2):649--665, 1994.

\bibitem{hayano:genus1}
Kenta Hayano.
\newblock On genus-1 simplified broken {L}efschetz fibrations.
\newblock {\em Algebr. Geom. Topol.}, 11(3):1267--1322, 2011.

\bibitem{hayano:genus1-II}
Kenta Hayano.
\newblock Complete classification of genus-1 simplified broken {L}efschetz
  fibrations.
\newblock {\em Hiroshima Math. J.}, 44(2):223--234, 2014.

\bibitem{Hayano_R2}
Kenta Hayano.
\newblock Modification rule of monodromies in an {R2}-move.
\newblock {\em Algebraic \& Geometric Topology}, 14(4):2181--2222, 2014.

\bibitem{Hayano:STS}
Kenta Hayano.
\newblock On diagrams of simplified trisections and mapping class groups.
\newblock {\em Osaka J. Math.}, 57(1):17--37, 2020.

\bibitem{Hillman}
Jonathan Hillman.
\newblock {\em Four-manifolds, Geometries and Knots}, volume~5.
\newblock Geometry and Topology Publications, 2002.

\bibitem{Iwase}
Zju\~{n}ici Iwase.
\newblock Dehn surgery along a torus {$T^2$}-knot. {II}.
\newblock {\em Japan. J. Math. (N.S.)}, 16(2):171--196, 1990.

\bibitem{LARSON_2016}
Kyle Larson.
\newblock Surgery on tori in the 4–sphere.
\newblock {\em Mathematical Proceedings of the Cambridge Philosophical
  Society}, 164(1):109–124, September 2016.

\bibitem{laudenbach-poenaru}
Fran\c{c}ois Laudenbach and Valentin Po\'{e}naru.
\newblock A note on {$4$}-dimensional handlebodies.
\newblock {\em Bull. Soc. Math. France}, 100:337--344, 1972.

\bibitem{Lekili}
Yank{\i} Lekili.
\newblock Wrinkled fibrations on near-symplectic manifolds.
\newblock {\em Geom. Topol.}, 13(1):277--318, 2009.
\newblock Appendix B by R. \.{I}nan\c{c} Baykur.

\bibitem{lickorish:nonorientable}
W.~B.~R. Lickorish.
\newblock Homeomorphisms of non-orientable two-manifolds.
\newblock {\em Proc. Cambridge Philos. Soc.}, 59:307--317, 1963.

\bibitem{meier}
Jeffrey Meier.
\newblock Trisections and spun four-manifolds.
\newblock {\em Math. Res. Lett.}, 25(5):1497--1524, 2018.

\bibitem{Miller-Naylor}
Maggie Miller and Patrick Naylor.
\newblock Trisections of nonorientable 4-manifolds.
\newblock {\em Michigan Math. J.}, 74(2):403--447, 2024.

\bibitem{Miller-Ozbagci}
Maggie Miller and Burak Ozbagci.
\newblock Lefschetz fibrations on nonorientable 4-manifolds.
\newblock {\em Pacific J. Math.}, 312(1):177--202, 2021.

\bibitem{Onaran-Ozbagci}
Sinem Onaran and Burak Ozbagci.
\newblock Minimal number of singular fibers in a nonorientable {L}efschetz
  fibration.
\newblock {\em Geom. Dedicata}, 216(1):Paper No. 4, 6, 2022.

\bibitem{Pao}
Peter~Sie Pao.
\newblock The topological structure of 4-manifolds with effective torus
  actions. i.
\newblock {\em Transactions of the American Mathematical Society},
  227:279--317, 1977.

\bibitem{Rubinstein-Tillman}
J.~Hyam Rubinstein and Stephan Tillmann.
\newblock Multisections of piecewise linear manifolds.
\newblock {\em Indiana Univ. Math. J.}, 69(6):2209--2239, 2020.

\bibitem{saeki:simplifying-general}
Osamu Saeki.
\newblock Simplifying generic smooth maps to the 2-sphere and to the plane.
\newblock \url{https://arxiv.org/abs/2407.10145}.

\bibitem{SAEKI2019}
Osamu Saeki.
\newblock Elimination of definite fold {II}.
\newblock {\em Kyushu Journal of Mathematics}, 73(2):239--250, 2019.

\bibitem{Spreer-Tillman}
Jonathan Spreer and Stephan Tillmann.
\newblock Determining the trisection genus of orientable and non-orientable
  {PL} 4-manifolds through triangulations.
\newblock {\em Exp. Math.}, 31(3):897--907, 2022.

\bibitem{Whitney1955}
Hassler Whitney.
\newblock On singularities of mappings of euclidean spaces. {I.} mappings of
  the plane into the plane.
\newblock {\em Annals of Mathematics}, 62(3):374--410, 1955.

\end{thebibliography}
\bibliographystyle{plain}

\end{document}